\documentclass[12pt,reqno]{amsart}
\usepackage{amssymb,mathrsfs,bbm,stmaryrd}
\usepackage{hyperref}

\newcommand{\M}{\mathcal{M}}
\newcommand{\X}{\mathcal{X}}
\newcommand{\Y}{\mathcal{Y}}
\newcommand{\cL}{L}

\newcommand{\cH}{\mathcal{H}}
\newcommand{\cP}{\mathcal{P}}
\newcommand{\cQ}{\mathcal{Q}}
\newcommand{\cN}{\mathcal{N}}

\newcommand{\N}{\mathbb{N}}
\newcommand{\R}{\mathbb{R}}

\newcommand{\al}{\alpha}
\newcommand{\be}{\beta}
\newcommand{\ga}{\gamma}
\newcommand{\de}{\delta}
\newcommand{\e}{\varepsilon}
\newcommand{\fy}{\varphi}
\newcommand{\om}{\omega}
\newcommand{\la}{\lambda}
\newcommand{\te}{\theta}
\newcommand{\s}{\sigma}
\newcommand{\ta}{\tau}
\newcommand{\ka}{\kappa}
\newcommand{\x}{\xi}
\newcommand{\y}{\eta}
\newcommand{\z}{\zeta}

\newcommand{\De}{\Delta}
\newcommand{\Om}{\Omega}
\newcommand{\Ga}{\Gamma}
\newcommand{\La}{\Lambda}

\newcommand{\p}{\partial}
\newcommand{\na}{\nabla}
\newcommand{\Cu}{\bigcup}
\newcommand{\Ca}{\bigcap}

\newcommand{\weak}{\operatorname{w-}}

\newcommand{\supp}{\operatorname{supp}}
\newcommand{\sign}{\operatorname{sign}}

\newcommand{\Span}{\operatorname{span}}

\newcommand{\lec}{\lesssim}
\newcommand{\gec}{\gtrsim}

\newcommand{\etc}{,\ldots,}

\newcommand{\I}{\infty}

\newcommand{\empt}{\varnothing}

\newcommand{\ti}{\widetilde}
\newcommand{\ba}{\overline}

\newcommand{\U}{\underline}

\newcommand{\BR}[1]{\left[#1\right]}
\newcommand{\LR}[1]{{\langle #1 \rangle}}

\newcommand{\tf}{\tfrac}

\newcommand{\tand}{\ \text{ and }\ }
\newcommand{\tor}{\ \text{ or }\ }

\newcommand{\IN}[1]{\text{ in }#1}

\newcommand{\Del}[1]{}
\newcommand{\CAS}[1]{\begin{cases} #1 \end{cases}}
\newcommand{\mat}[1]{\begin{pmatrix} #1 \end{pmatrix}}
\newcommand{\bmat}[1]{\begin{bmatrix} #1 \end{bmatrix}}
\newcommand{\smat}[1]{\left[\begin{smallmatrix} #1 \end{smallmatrix}\right]}
\newcommand{\EQ}[1]{\begin{equation}\begin{split} #1 \end{split}\end{equation}}

\newcommand{\pt}{&}
\newcommand{\pr}{\\ &}
\newcommand{\pq}{\quad}

\newcommand{\pn}{}
\newcommand{\prq}{\\ &\quad}
\newcommand{\prQ}{\\ &\qquad}

\numberwithin{equation}{section}

\newtheorem{theorem}{Theorem}[section]
\newtheorem{prop}[theorem]{Proposition}
\newtheorem{lemma}[theorem]{Lemma}

\theoremstyle{remark}

\newtheorem{remark}{Remark}[section]

\newcommand{\sL}{\vec{L}}

\newcommand{\bu}{\bullet}
\newcommand{\xu}{{\bf u}}
\newcommand{\xc}{{\bf c}}
\newcommand{\xcu}{{\bf cu}}
\newcommand{\xscu}{{\bf eig}}

\newcommand{\eN}{{\bf q}}
\newcommand{\Si}{\Sigma}
\newcommand{\op}{\operatorname{op}}
\newcommand{\sN}{\mathcal{E}}
\newcommand{\sI}{\mathscr{I}}

\newcommand{\ci}{\circ}
\newcommand{\ck}{\check}

\newcommand{\mm}{\mathbbm{m}}
\newcommand{\dia}{\lozenge}
\newcommand{\ap}{\mathbbm{a}}
\newcommand{\he}{\heartsuit}

\newcommand{\dbr}[1]{\llbracket #1 \rrbracket}
\newcommand{\cI}{\mathcal{I}}

\begin{document}
\subjclass{Primary 35L71; Secondary 37L25, 35B40, 35B30}
\keywords{Nonlinear Klein-Gordon equation, solitons, blow-up, invariant manifolds}
\title[Global dynamics around 2-solitons]{Global dynamics around 2-solitons for the nonlinear damped Klein-Gordon equations}
\author{Kenjiro Ishizuka}
\author{Kenji Nakanishi}
\address{Research Institute for Mathematical Sciences, Kyoto University}
\thanks{This work was supported by JSPS KAKENHI Grant Number JP17H02854.}

\begin{abstract}
Global behavior of solutions is studied for the nonlinear Klein-Gordon equation with a focusing power nonlinearity and a damping term in the energy space on the Euclidean space. We give a complete classification of solutions into 5 types of global behavior for all initial data in a small neighborhood of each superposition of two ground states (2-solitons) with the opposite signs and sufficient spatial distance. 
The neighborhood contains, for each sign of the ground state, the manifold with codimension one in the energy space, consisting of solutions that converge to the ground state at time infinity. 
The two manifolds are joined at their boundary by the manifold with codimension two of solutions that are asymptotic to 2-solitons moving away from each other. 
The connected union of these three manifolds separates the rest of the neighborhood into the open set of global decaying solutions and that of blow-up. 
The main ingredient in the proof is a difference estimate on two solutions starting near 2-solitons and asymptotic to 1-solitons. The main difficulty is in controlling the {\it direction} of the two unstable modes attached to 2-solitons, while the soliton interactions are not uniformly integrable in time. It is resolved by showing that the non-scalar part of the interaction between the unstable modes is uniformly integrable due to the symmetry of the equation and the 2-solitons. 
\end{abstract}

\maketitle
\tableofcontents

\section{Introduction}
We investigate global behavior of solutions for the nonlinear damped Klein-Gordon equation on $\R^N$: 
\begin{equation}
{\partial}_t^2u+2\alpha {\partial}_tu-\Delta u+u = f(u)  \label{NLKG}
\end{equation}
where $u(t,x):\R^{1+N}\to\R$, $\alpha>0$ is a damping constant, and $f(u):=|u|^{p-1}u$ with a power $p$ in the energy sub-critical range, namely
\EQ{
 2<p<p^\star(N):=\CAS{\infty &(N=1,2), \\ \frac{N+2}{N-2} &(N\ge 3).}}
The restriction $p>2$ implies $N\le 5$. 
For any space-time function $u(t,x)$, we denote 
\EQ{
 \vec u(t,x) := (u(t,x),\dot u(t,x)) \in \R^2.}
The Cauchy problem for \eqref{NLKG} is locally well-posed in the energy space 
\EQ{
 \mathcal{H}:=H^1(\R^N)\times L^2(\R^N)} 
(cf.~\cite{BRS}): For any initial data $\vec u(0)\in\cH$, there is a unique solution $\vec u\in C(I;\cH)$ of \eqref{NLKG} on the maximal existence interval $I$. 
The energy functional $E:\cH\to\R$ for \eqref{NLKG} is defined by 
\begin{equation}
 E(u_1,u_2)=\frac{1}{2}\int_{{\mathbb{R}}^N}\{ |\nabla u_1|^2+|u_1|^2+|u_2|^2-2F(u_1)\}dx, \label{energy}
\end{equation}
where $F(u):=\tf{|u|^{p+1}}{p+1}$ is the primitive of $f$. 
For any solution $\vec{u}$ of \eqref{NLKG} in the energy space, the energy is non-increasing in time:
\begin{equation} 
E(\vec{u}(t_2))-E(\vec{u}(t_1))=-2\alpha \int_{t_1}^{t_2}\| {\partial}_tu(t)\|_{L^2}^2dt. \label{Edecay}
\end{equation}
When the nonlinearity is defocusing, namely $f(u)=-|u|^{p-1}u$ or $F(u)=-\tf{|u|^{p+1}}{p+1}$, the above decay of energy leads rather immediately to global existence and exponential decay of all solutions in the energy space. 

However, when the nonlinearity is focusing, $f(u)=|u|^{p-1}u$, which is the case throughout this paper, the equation has more variety of solutions, including blow-up and stationary solutions. 
In this paper, we are interested in solutions starting close to superposition of stationary solutions with the least energy, namely the ground states. 
The ground state exists uniquely \cite{BL,K} up to spatial translation and change of the sign. 
Let $Q$ be the unique positive radial one: 
\begin{equation}
 0<Q(x)=Q(|x|)\in H^1(\R^n), \pq -\Delta Q+Q-f(Q)=0. 
\end{equation}

There has been intensive study on global behavior of general (large) solutions for the nonlinear dispersive equations, where the guiding principle is the {\it soliton resolution conjecture}, which claims that generic global solutions are asymptotic to superposition of solitons for large time. In particular, the conjecture has been proven for all radial solutions of the energy-critical wave equation:
\EQ{
 \ddot u - \De u = |u|^{p^\star(N)-1}u,\pq (N\ge 3),}
including the case of blow-up with bounded energy norm, 
by Duyckaerts, Kenig and Merle \cite{DKM,DKM2} for $N=3,5,7\etc$ and by Duyckaerts, Kenig, Martel and Merle \cite{DKMM} for $N=4$. Similar results have been obtained also in the case with a potential \cite{JLX}, as well as for the energy-critical wave maps \cite{KLLS,DKMM}, all under some rotational symmetry. 

However, the conjecture is widely open for other dispersive equations without size restriction on the initial data, except for the completely integrable case. In particular, for the nonlinear Klein-Gordon equation: 
\EQ{
 \ddot u - \De u + u = |u|^{p-1}u,}
general asymptotic behavior is unknown if the energy is much bigger than the ground state (for example, in the classification of solutions by \cite{NS} into 9 sets, the energy may exceed only slightly the ground state). 
The current state is the same for the nonlinear Schr\"odinger equation. 
This motivates us to look into other PDEs with similar solutions but easier to analyze. 

For the damped nonlinear Klein-Gordon \eqref{NLKG}, the energy decay \eqref{Edecay}, together with the exponential decay for the linear equation, makes the analysis simpler than the undamped case, although the damping does not remove the blow-up or solitons. 
Much earlier than the recent progress for the conservative PDEs, Keller \cite{Ke} constructed stable and unstable manifolds in a neighborhood of any stationary solution of \eqref{NLKG}, and 
Feireisl \cite{F} proved the soliton resolution along a time sequence : for any global solution $u$ of \eqref{NLKG} in the energy space, there is a time sequence $t_n\to\I$ and a set of stationary solutions $\{\fy_k\}_{k=1}^K$ ($K\ge 0$), and sequences $\{c_{k,n}\}_{k=1\etc K,n\in\N}$ such that 
\EQ{ \label{conv seq}
 \vec u(t_n) = \sum_{k=1}^K \vec\fy_k(x-c_{k,n}) + o(1) \pq \min_{j\not=k}|c_{j,n}-c_{k,n}|\to\I}
as $n\to\I$ in the energy space (under some restriction on $p$). 
The energy decay \eqref{Edecay} reduces the proof mostly to the variational analysis of concentration compactness \cite{L}, independent of time. 
Note that for the conservative PDEs, this type of results (in appropriate topology) also requires much more analysis on the evolution (cf. \cite{DJKM} in the non-radial case of the energy-critical wave equation). 

More recently, Burq, Raugel and Schlag \cite{BRS} proved the soliton resolution for \eqref{NLKG} in the full limit $t\to\I$ for all radial solutions. 
On the other hand, C\^{o}te, Martel, Yuan and Zhao \cite{CMYZ} constructed a Lipschitz manifold in the energy space with codimension $2$ of those solutions asymptotic to a sum of two ground states 
\EQ{ \label{2soli}
 \vec u(t) = \vec Q(x-z_1(t)) - \vec Q(x-z_2(t)) + o(1), \pq |z_1(t)-z_2(t)|\to\I \pq(t\to\I).}
More generally, one may consider the following types of solutions in the energy space
\begin{itemize}
\item We say a solution $\vec{u}$ in $\cH$ of \eqref{NLKG} is {\it decaying} if it extends globally for $t>0$ and converges to $0$ strongly in $\mathcal{H}$ as $t\to \infty$. 
\item We say a solution $\vec u$ in $\cH$ of \eqref{NLKG} {\it blows up} if it exists up to some $T\in(0,\I)$ and $\|\vec u(t)\|_\cH\to\I$ as $t\to T-0$. 
\item A solution $\vec{u}$ in $\cH$ of \eqref{NLKG} is called {\it an asymptotic $K$-soliton} for $K\in\N$ if it extends globally for $t>0$ and there exist $\s\in\{\pm 1\}^K$ and $z:[0,\I)\to(\R^N)^K$ such that
\EQ{  \label{asy Ksol}
\lim_{t\to \infty} \|\vec u(t,x)-\sum_{j=1}^K{\sigma}_j\vec Q(x-z_j(t))\|_\cH=0, \pq\lim_{t\to\I}\min_{j\not=k}|z_j(t)-z_k(t)|=\I.}
\end{itemize}

In the one dimensional case $N=1$, 
C\^{o}te, Martel and Yuan \cite{CMY} proved that the above three types exhaust all solutions in the energy space. 
To the best of our knowledge, this is the first complete proof of the soliton resolution conjecture for a non-integrable equation without any size or symmetry restriction and with moving solitons, provided that the damping is acceptable in the conjecture. 

\begin{remark}
It may be slight abuse of words to call those solutions {\it solitons}, because they are unstable, 
but we use this terminology as it is convenient and getting quite common in the analysis of nonlinear wave equations. 
For $N\ge 2$, one may also consider superpositions of other stationary solutions than the ground state $Q$. In this paper, however, {\it $K$-solitons} refer merely to superpositions of the ground state, namely 
\EQ{
 \sum_{k=1}^K \s_k \vec Q(x-z_k)}
for any $\s\in\{\pm 1\}^K$ and $z\in(\R^N)^K$. The center points $\{z_k\}$ of solitons may depend on time, but not necessarily be defined for all $t\to\I$. Note that in \cite{CMY} those solutions satisfying \eqref{asy Ksol} are called $K$-solitons. 
This is because the more focus is put in this paper on the intermediate and transitive dynamics than the asymptotic behavior. 
\end{remark}

Once we know all the possible asymptotic behavior (or even before that), a natural question is which initial data lead to each type of solutions. 
Recall that if the energy is less than the ground state $Q$, then the variational character of the ground state implies the dichotomy into decaying solutions and blow-up solutions by the sign of the initial integral, namely the Nehari functional  
\EQ{
 K_0(u) := \int_{\R^N}[|\na u|^2 + |u|^2 - |u|^{p+1}]dx.}
If $E(\vec u(0))<E(\vec Q)$, then $K_0(u(0))\ge 0$ implies the global existence, while $K_0(u(0))<0$ implies the blow-up. 
As soon as the energy exceeds the ground state, it seems hard to find such an explicit characterization. However, one may still expect some qualitative as well as quantitative relations and classifications among different types of solutions. 
For the undamped equation, Schlag and the second author \cite{NS} obtained such a classification into 9 sets in terms of the invariant manifolds of the ground states for $E(\vec u)<E(\vec Q)+\e$ with small $\e>0$. 

The goal of this paper is to obtain a similar classification around the 2-solitons \eqref{2soli}, exploiting the damping, and investigating all the solutions starting close to but not on the manifold in \cite{CMYZ}. 
The main interest is in the transition from 2-solitons to 1-solitons, both by time evolution and by initial perturbation. 

\subsection{The main result}
To state the main result, we need to introduce the spectral decomposition around the ground state $Q$. The linearized operator around $Q$ of the stationary equation is denoted by 
\begin{equation}
 \cL:=-\Delta+1-f'(Q) = -\De+1-pQ^{p-1}.
\end{equation}
It is well-known (cf.~\cite{CGNT}) that $\cL$ has a unique negative eigenvalue,  and that its kernel is generated exactly by the translation invariance: $\cL^{-1}(\{0\})=\Span\{\na Q\}$. 
We denote by $\phi$ the normalized ground state of $\cL$, and the eigenvalue by $-\nu_0^2$ for a constant $\nu_0>0$
\EQ{
 \cL \phi = -\nu_0^2 \phi, \pq 0<\phi=\phi(|x|) \in H^2(\R^N), \pq \|\phi\|_{L^2(\R^N)}=1.}
The linearized evolution around $\vec Q$ is given by 
\EQ{ \label{linev}
 \dot v = J\sL^\al v, \pq J:=\bmat{0 & 1 \\ -1 & 0}, \pq \sL^\al:=\bmat{\cL & 2\al \\ 0 & 1}.}
Let 
\EQ{
 \pt \nu^{\pm}:=-\alpha \pm \sqrt{{\alpha}^2+{\nu}_0^2}\in\R, 
 \pq Y^\pm:=\smat{1 \\ \nu^\pm}\phi\in\cH.}
Then they are eigenfunctions of $J\sL^\al$ and $J^*(\sL^\al)^*$, namely
\EQ{
 (J\sL^\al-\nu^\pm) Y^\pm = 0 =
 (J^*(\sL^\al)^*-\nu^\pm) Y^\mp.}
Since ${\nu}^+>0$, the solution $v=e^{\nu^+t} Y^+$ of \eqref{linev} exhibits the exponential instability of the ground state $\vec Q$. In particular, we see that the presence of the damping $\alpha>0$ in the equation does not remove the exponential instability. 

Define the symplectic form $\om:\cH^2\to\R$ by 
\EQ{
 \om(\fy,\psi) := \LR{J\fy,\psi} = \int_{\R^N}[\fy_2(x)\psi_1(x)-\fy_1(x)\psi_2(x)]dx=  -\om(\psi,\fy),}
where $\LR{\cdot,\cdot}$ denotes the real inner product on $L^2(\R^N)$. 
For any $z=(z_1,z_2)\in(\R^N)^2$ and $\de>0$, the open ball in the center-stable subspace around the 2-soliton at $z$ is denoted by 
\EQ{
 \Y^\xu_\perp(z;\de) := \{\fy\in\cH\mid \forall k=1,2,\ \om(\fy,Y^-(x-z_k))=0,\ \|\fy\|_\cH<\de\}.}
It is easy to see that the energy space may be decomposed as 
\EQ{
 \cH = \R Y^+(x-z_1) \oplus \R Y^+(x-z_2) \oplus \Y^\xu_\perp(z;\I)}
provided that $|z_1-z_2|$ is large enough. 
The main result of this paper is as follows.
\begin{theorem} \label{thm:main}
For any $N\in\N$, $\al>0$ and $p\in(2,p^\star(N))$, there exists $\de\in(0,1)$ such that for any $\s_1,\s_2\in\{\pm 1\}$ and any $z=(z_1,z_2)\in(\R^N)^2$ satisfying
\EQ{
 \s_1\s_2=-1, \pq |z_1-z_2|>1/\de,}
there are two Lipschitz functions $G_1,G_2:(-\de,\de)\times \Y_\perp^\xu(z;\de)\to(-\de,\de)$ with the following properties. 
For any $h_1,h_2\in(-\de,\de)$ and any $\fy\in\Y_\perp^\xu(z;\de)$, 
let $u$ be the solution of \eqref{NLKG} with the initial data 
\EQ{ \label{2sol init}
 \vec u(0) = \sum_{j=1}^2 \s_j[Q+h_j Y^+](x-z_j)  + \fy.}
Then its global behavior is classified by the initial data as follows. 
Let $j^*:=3-j$ denote the conjugate index. 
\begin{enumerate}
\item If $h_j<G_j(h_{j^*},\fy)$ for both $j=1,2$ then $u$ is global and decaying.  \label{van}
\item If $h_j=G_j(h_{j^*},\fy)$ and $h_{j^*}<G_{j^*}(h_j,\fy)$ for one of $j=1,2$, then $u$ is global and 
\EQ{
 \vec u(t)=\s_j\vec Q(x-z_\I)+o(1) \IN{\cH}} 
as $t\to\I$ for some $z_\I\in\R^N$, namely an asymptotic $1$-soliton. \label{1sol}
\item If $h_j=G_j(h_{j^*},\fy)$ for both $j=1,2$ then $u$ is global and 
\EQ{
 \vec u(t)=\s_1\vec Q(x-z_1(t))+\s_2 \vec Q(x-z_2(t))+o(1) \IN{\cH}} 
as $t\to\I$ for some $z\in C^1([0,\I);(\R^N)^2)$ satisfying $|z_1(t)-z_2(t)|\to\I$, namely an asymptotic $2$-soliton. \label{2sol}
\item Otherwise, $u$ blows up in finite time. 
\end{enumerate}
The case (3) may be characterized simply as $(h_1,h_2)=G_0(\fy)$ for another Lipschitz function $G_0:\Y^\xu_\perp(z;\de)\to(-\de,\de)^2$. 
\end{theorem}
\begin{remark}
\cite{CMYZ} proved that $\s_1\s_2=-1$ is necessary for any asymptotic $2$-soliton to exist. Note also that the two cases with $j=1$ and with $j=2$  in \eqref{1sol} are distinguished by the sign of the remaining soliton $\s_j Q$. 
Hence the above theorem yields a complete classification into 5 non-empty and disjoint sets of solutions starting from a small neighborhood of any asymptotic 2-soliton, or any superposition of two ground states with the opposite signs and sufficient distance. This is an extension of the result in \cite{CMYZ} that established the case of \eqref{2sol}, namely the manifold of asymptotic 2-solitons (without classifying the other cases). 
\end{remark}

In short, we have two hypersurfaces of 1-solitons respectively with $\vec Q$ and $-\vec Q$, given by the Lipschitz graphs $h_j=G_j(h_{j^*},\fy)$ with $h_{j^*}<G_{j^*}(h_j,\fy)$ for $j=1,2$, in the small neighborhood of $\cH$ around each $2$-soliton of the opposite signs and sufficient distance. 
Their joint boundary $h=G_0(\fy)$ is the manifold of 2-solitons constructed in \cite{CMYZ}, and the union of those three separates the rest into the decaying solutions and the blow-up. 

\subsection{Difficulties and ideas for the main result}
The 1-soliton manifolds in the above result are different from those constructed by Keller \cite{Ke}, or any general method that works locally around stationary points, because the manifolds in the above result are approaching the 2-solitons, which are far from any $1$-soliton, while the solutions on the manifolds are asymptotic to 1-solitons as $t\to\I$. In particular, the global behavior is not uniform at all in the small neighborhood of $2$-solitons; indeed, the above result precisely describes the discontinuity of the asymptotic behavior around the 2-solitons. 

In order to analyze the transition between different numbers of solitons, we exploit the finite speed of propagation for \eqref{NLKG} and some uniform estimates for solutions below the ground state. 
It is combined with the linearized analysis around the solitons, taking account of the soliton interactions that have been investigated in \cite{CMYZ}. 

The crucial difference in the linearized analysis both from \cite{CMYZ} and from the case of 1-solitons, is that we have to distinguish the growth {\it direction} of the unstable modes $(h_1,h_2)\in\R^2$, in order both to distinguish the 5 cases and to construct the 1-soliton manifolds all the way to the 2-soliton manifold. 
Even though the conclusion is that the unstable mode does not essentially change the direction, it is highly non-trivial, because the soliton interaction generates $O(1/t)$ error terms from the linearized operator, which do change the growth {\it order} of unstable modes from the linear approximation. 
A simple ODE example to illustrate the problem is 
\EQ{
 \frac{d}{dt}\mat{h_1 \\ h_2} = \mat{\nu^+ & \e^2 e^{-t} \\ \e^2 e^{-t} & \nu^++\frac{\e}{1+t}}\mat{h_1 \\ h_2},}
which mimics the linearized interaction between the two unstable modes. 
It is easy to observe that if $\e>0$ then $h_2(t)/h_1(t)\to\I$ as $t\to\I$ even if it starts with $(h_1(0),h_2(0))=(1,0)$, 
and if $\e<0$ then $h_1(t)/h_2(t)\to\I$ even if $(h_1(0),h_2(0))=(0,1)$. 

Such rotational behavior of the unstable modes is precluded by showing that the non-scalar part of the above matrix is actually uniformly integrable and small, due to the symmetry of the equation as well as the 2-solitons. 
Note however that we need this information not only for the asymptotic 2-solitons but also for those on the 1-soliton manifolds. 
This also requires a detailed description of solutions, as the interaction matrix depends highly on the individual solutions, and their estimates are far from being uniform in time. 

\subsection{Difficulty for more solitons}
After the above result, a natural question would be what happens around the 3-solitons or more. 
Although one might expect a similar structure as boundaries of manifolds with less solitons, the analysis has to face a new and much harder phenomenon, which may be called the {\it soliton merger}. 
A crucial property of the 2-solitons behind the proof of the above result is that the solitons with the opposite signs are repelling each other, so they never approach each other as long as both exist. 
This property is lost as soon as the number of solitons becomes $3$ or more. 
For example, \cite{CMY} constructed the asymptotic 3-solitons in the form on $\R$ ($N=1$)
\EQ{
 \vec u(t) = \vec Q(x+c(t)) - \vec Q(x) + \vec Q(x-c(t)) + o(1), \pq c(t)\to\I,}
and also proved that it is essentially the unique form (up to symmetry). 
If only the middle soliton is collapsed by the instability and disappears (decays), then after that the solution contains only the two solitons with the same signs, which attract each other. 
By using the soliton resolution of \cite{CMY}, as well as stability of decaying and blow-up, it is easy to observe that there are solutions in which the middle soliton collapses first, and afterward the other two solitons together, but eventually a soliton in the middle appears as the final state. 
This behavior is essentially different from the 1-soliton case \eqref{1sol} in Theorem \ref{thm:main}, because the final soliton is not a remaining one of the three starting ones, but created by merger of the two of them. 
It seems very hard to analyze the behavior of such solutions.  

\subsection{Structure of the paper}
In Section 2, some basic properties are recalled about the ground state and the solutions below it, in particular, uniform-in-time estimates for decaying and blow-up solutions. 

The analysis of multi-solitons starts from Section 3. 
Although the main result is about the 2-solitons, we consider the general number $K$ of solitons as much as possible, in order to clarify where the specialty of 2-solitons is used. 
In Section 3, the spectral decomposition is introduced, as well as the corresponding system of equations in terms of the unstable modes and the remainder (radiation). 
In Section 4, the orthogonality condition is introduced to determine the soliton centers, as well as the corresponding modulation equation. 
Those arguments are well known and mere reformulations from \cite{CMYZ} in slightly different coordinates. 
In Section 5, the basic estimates are prepared, one is using the static energy, the other is using the linear damped energy inequality. 
In the case of higher powers, we need to combine also the Strichartz estimate. 
We also introduce the soliton repulsivity conditions, which are abstract and sufficient conditions to avoid soliton merger or collision in the case of general number of solitons.  

The core analysis and the main novelty of this paper starts from Section 6. 
In Section 6, a complete classification and description is given on the behavior of solutions during they stay in a small neighborhood of multi-solitons (which are moving according to the modulation in Section 4). 
This is written for general $K$-solitons, but under the soliton repulsivity conditions, which are satisfied in the 2-soliton case. 
In Section 7, we prove the crucial estimate to prove that the unstable modes do not essentially change the direction. 
It is about the even symmetric part of the solution, for which we need to specialized to the case of 2-solitons with the opposite sign. 
In Section 8, the convergence to $K$-solitons is proven for solutions staying forever around them. 
This is a rather immediate consequence of \eqref{conv seq}. 

In Section 9, the difference of two solutions is estimated during both of them are staying near $K$-solitons. 
Here is the crucial step of preserving the unstable direction, using the estimate in Section 7. 
In Section 10, the collapse of solitons is described, both in the case of decay and of blow-up, using the finite speed of propagation, and assuming that the solitons are separated enough from each other. 
Then in Section 11, the difference of two solutions is estimated in the case where both are collapsing some solitons in the same way. 

Gathering all these estimates, we finally prove the main result in Section 12. 

In Section 13, $K$-solitons with $K\ge 3$ are considered under some symmetry restrictions on the solutions. 
The purpose of this section is to show cases with the soliton repulsivity conditions and $K\ge 3$. 
At the end is a brief explanation on the soliton merger in the simplest case of the symmetry. 

\subsection{Notation}
Let $\{e_1\etc e_N\}$ denote the canonical basis of $\R^N$. 
The $L^2$ scalar product is denoted for any pair of functions $u,v\in L^2(\R^N)$ by
\EQ{
 \LR{u,v}:=\int_{\R^N} u(x)v(x)dx.}

Positive constants are denoted by $C$, which may change from line to line. Subscripts are added to
distinguish different constants. Note that $N,\al,p,K$ are regarded as constants.
$X\lec Y$ means that $X\le CY$ for some constant $C>0$. $X\sim Y$ means that $X\lec Y$ and $Y\lec X$.
$X\ll Y$ means that there is a sufficiently small constant $c>0$ such that $X\le cY$.
The meaning of sufficiency should be clear from the context, or clarified later in more explicit form. 

For any Banach space $B$, the same space with the weak topology is denoted by $\weak{B}$. The operator norm for linear bounded operators on $B$ is denoted by $\op(B)$, and simiarly by $\op(B_1\to B_2)$ for those from $B_1$ to $B_2$.  
For any open set $\Om\subset\R^N$, the energy space is denoted by $\cH(\Om):=H^1(\Om)\times L^2(\Om)$.

\section{Basic properties of the ground state}
In this section, we collect well-known properties of the ground state $Q$. First, it is obtained by the direct variational method for the constrained minimization 
\EQ{ \label{gsedef}
 E(\vec Q) = \inf\{E(\vec\fy) \mid \fy\in H^1(\R^N)\setminus\{0\},\ K_0(\fy)=0\},}
where $K_0:H^1(\R^N)\to\R$ is the Nehari functional defined by 
\EQ{
 K_0(\fy) :=\int_{\R^N} |\na\fy|^2+|\fy|^2-|\fy|^{p+1}dx = \p_{\la=1}E(\la\vec\fy).}
In particular, $K_0(\fy_1)\not=0$ for any $\fy=(\fy_1,\fy_2)\in\cH$ such that $E(\fy)<E(\vec Q)$ and $\fy_1\not=0$. 
Moreover, we have a uniform bound (cf.~\cite[Lemma 2.12]{IMN})
\EQ{ \label{est K0}
 E(\fy)<E(\vec Q)-\de \implies \CAS{K_0(\fy_1) \gec \min\{\de,\|\fy_1\|_{H^1}^2\} \tor\\ K_0(\fy_1)\lec -\de.}}
In particular, we have $K_0\ge 0$ in a neighborhood of $0\in H^1(\R^N)$. 
 
The ground state $Q$ is radial and exponentially decaying. More precisely, 
\EQ{
 Q(x) = \eN(|x|)}
for a $C^\I$ function $\eN:[0,\I)\to(0,\I)$ satisfying 
$\eN(r)>0>\eN'(r)$ and there exists a constant $c_0>0$ depending on $N,p$, such that 
\EQ{\label{asy Q}
 |\eN(r)-c_0 r^{-\frac{N-1}{2}}e^{-r}|+r|\eN'(r)+\eN(r)|\lesssim r^{-\frac{N+1}{2}}e^{-r},}
where the remainder estimate is improved to $r^{-\frac{N-1}{2}}e^{-(p-1)r}$ for $N=1,3$. 
This formula follows immediately from the radial ODE for $\eN$. 

Similarly, the ground state $\phi$ of the linearized operator $\cL$ is radial and exponentially decaying with the better rate than $Q$ (because of the negative eigenvalue)
\EQ{
 |\phi-c_1 r^{-\frac{N-1}{2}}e^{-\LR{\nu_0}r}| + r|\phi_r+\LR{\nu_0}\phi|\lesssim  r^{-\frac{N+1}{2}}e^{-\LR{\nu_0}r},}
for some constant $c_1>0$ depending on $N,p$. $Y^\pm$ have the same asymptotic behavior. 
Moreover, the linearized operator has the coercivity on the subspace orthogonal to $\phi$ and $\na Q$. More precisely, there is some constant $C>0$ depending on $N,p$, such that for all $\varphi \in H^1(\R^N)$, 
\EQ{ \label{L ene}
 \| \varphi \|_{H^1}^2 \sim \LR{\cL\fy|\fy} + C\LR{\fy|\phi}^2 + C\sum_{j=1}^N \LR{\fy|\p_jQ}^2.}

\subsection{Uniform dynamics below the ground state}
It is well known that all the solutions with energy less than the ground state $E(\vec u)<E(\vec Q)$ are split into the decaying (for $K_0(u)\ge 0$) and the blow-up (for $K_0(u)<0$). 
In order to trace the behavior of solutions getting away from $2$-solitons, we need some upper bounds on the time that those solutions take before decaying or blowing-up. They follow essentially from the estimates and arguments in \cite{CMYZ,BRS}. 
Define a functional $\cP:\cH\to\R$ by 
\EQ{
 \cP(\fy):=\LR{\fy_1,\fy_2}+\al\|\fy_1\|_{L^2}^2.}
Then for any solution $u$ of \eqref{NLKG} in $\cH$, 
\EQ{ \label{mass id}
  \p_t \cP(\vec u) = \|\dot u\|_{L^2}^2-K_0(u) = \tf{p+3}{2}\|\dot u\|_{L^2}^2+\tf{p-1}{2}\|u\|_{H^1}^2-(p+1)E(\vec u).}

First in the decaying case, we have the uniform decay rate: 
\begin{lemma}\label{lemma2.2}
For any $N\in\N$, $\al\in(0,\I)$ and $p\in(1,p^\star(N))$, 
there exist $C\in(1,\I)$ and $c\in(0,1)$ such that for any $\de>0$ and any solution $\vec{u}$ of \eqref{NLKG} satisfying $E(\vec u(0))<E(\vec Q)-\de$ and $K_0(u(0))\ge 0$, we have for all $t\ge 0$
\begin{equation} \label{estlemma2.2}
\| \vec{u}(t)\|_{\mathcal{H}}\le Ce^{-c\de t}\|\vec u(0)\|_\cH.
\end{equation}
\end{lemma}
\begin{proof}
Let $B:=\max(1/\al,2\tf{p+1}{p-1}(1+2\al))$. 
Then from \eqref{mass id}, \eqref{Edecay} and \eqref{gsedef}, we obtain
\begin{equation} \label{estlemma2.2-1}
 \p_t[\cP(\vec u)+BE(\vec u)]=(1-2\alpha B)\|{\partial}_tu(t)\|_{L^2}^2-K_0(u)\lesssim -\de\| \vec{u}\|_{\mathcal{H}}^2,
\end{equation}
while we have
\EQ{ \label{bd on P}
 |\cP(\vec u)| \le |\LR{u,\dot u}| + \al \|u\|_{L^2}^2 \le (\al+\tf12)\|u\|_{L^2}^2 + \tf12\|\dot u\|_{L^2}^2,}
and
\EQ{
 E(\vec u) = \tf{1}{p+1}K_0(u) + \tf12\|\dot u\|_{L^2}^2 + (\tf12-\tf1{p+1})\|u\|_{H^1}^2 \ge \tf{p-1}{2(p+1)} \|\vec u\|_\cH^2.}
Since $B\ge 2\tf{p+1}{p-1}(1+2\al)$, the above two estimates imply 
\EQ{
 \|\vec u\|_\cH^2 \lec \cP(\vec u) + BE(\vec u) \lec (\al+\al^{-1})\|\vec u\|_\cH^2.}
Injecting this into \eqref{estlemma2.2-1} yields \eqref{estlemma2.2} with $C\sim B^{1/2}$. 
\end{proof}

For the blow-up, we have the following sufficient condition. 
\begin{lemma}
Let $N\in\N$, $\al\in(0,\I)$ and $p\in(1,p^\star(N))$. Let $u$ be a solution of \eqref{NLKG} in $\cH$. Then $u$ blows up if at some $t\ge 0$
\EQ{ \label{bc-P}
 \cP(\vec u(t)) > \tf{p+1}{p-1}(1+2\al) E(\vec u(t)).}
Moreover, this inequality is preserved as long as the solution exists. 
\end{lemma}
Note that this is also a necessary condition for lower $p$ and $N$, and then the above result is equivalent to uniform bound in $\cH$ of global solutions (cf.~\cite{C,CMY}). For general (subcritical) power, however, it seems to be an open question both with the damping and without it. 
For our purpose, however, the above lemma (namely the sufficiency) is enough. 
\begin{proof}
If $\cP(\vec u(0))\le 0$, then \eqref{bc-P} implies that $E(\vec u(0))<0$, then \eqref{mass id} implies 
\EQ{
 \p_t\cP(\vec u) \ge - (p+1)E(\vec u(0)) > 0,}
so $\cP(\vec u)$ becomes positive in finite time or blows up before that. 
On the other hand, by \eqref{mass id} and \eqref{bd on P}, we have 
\EQ{ \label{dieq cP}
 c:=\tf{p-1}{1+2\al} \implies \p_t\cP(\vec u) \ge c\BR{\cP(\vec u) - \tf{p+1}{p-1}(1+2\al)E(\vec u)}.}
Hence, after we have \eqref{bc-P}, $\cP(\vec u)$ is strictly increasing in $t$, while $E(\vec u)$ is non-increasing, so \eqref{bc-P} is preserved, as long as the solution exists. If the solution exists globally on $t>0$, then $\cP(\vec u)\nearrow\I$ as $t\to\I$. Since $(\p_t+2\al)\|u(t)\|_{L^2}^2=2\cP(\vec u)\to\I$, we have $\|u(t)\|_{L^2}\to \I$, too. 
Let 
\EQ{
 M(t) := \tf12\|u(t)\|_{L^2}^2 + \int_0^t \al \|u(s)\|_{L^2}^2 ds.}
Then, using \eqref{mass id} and \eqref{Edecay}, we obtain 
\EQ{
 \pt \dot M = \cP(\vec u) = \LR{u,\dot u} + \al\|u(0)\|_{L^2}^2 + \int_0^t 2\al\LR{u,\dot u}(s)ds,
 \pr \ddot M = \tf{p+3}{2}\|\dot u\|_2^2 + \tf{p-1}{2}\|u\|_{H^1}^2 -(p+1)E(\vec u(0)) + (p+1)\int_0^t 2\al\|\dot u(s)\|_{L^2}^2 ds.}
Hence for any $\e\in(0,1)$, there is $T>0$ such that for all $t\ge T$ we have 
\EQ{
 \p_t \cP(\vec u)>0, \pq  \al\|u(0)\|_{L^2}^2 < \tf{\e}{1+\e} \cP(\vec u), 
 \pq \tf{p-1}{2}\|u\|_{L^2}^2 > (p+1)E(\vec u(0)).}
Then for $t>T$, using Schwarz in $L^2(\R^N)$, in $L^2((0,t)\times\R^N)$ and in $\R^2$, we obtain 
\EQ{
  \pt(1+\e)^{-1}\dot M = (1-\tf{\e}{1+\e})\dot M < \LR{u,\dot u} + \int_0^t 2\al\LR{u,\dot u}dt
  \pr\le \|u(t)\|_{L^2_x} \|\dot u(t)\|_{L^2_x} + 2\al\|u\|_{L^2((0,t)\times\R^N)}\|\dot u\|_{L^2((0,t)\times\R^N)}
  \pr\le \BR{\tf12\|u(t)\|_{L^2_x}^2+\al\|u\|_{L^2((0,t)\times\R^N)}^2}^{1/2}\BR{2\|\dot u(t)\|_{L^2_x}^2 + 4\al\|\dot u\|_{L^2((0,t)\times\R^N)}^2}^{1/2}
  \pr\le \sqrt{M(t)\tf{4}{p+3}\ddot M(t)}. }
Since $p>1$, we may choose $\e>0$ small so that 
\EQ{
 1<\tf{p+3}{4(1+\e)^2}=:1+\de}
for some $\de>0$. Then the above inequality implies for $t>T$
\EQ{
 \p_t^2 M^{-\de} = -\de M^{-1-\de}\BR{M\ddot M-(1+\de)\dot M^2} < 0,}
which is contradicting $M\to\I$. Hence $u$ can not be global. 
\end{proof}

For blow-up below the ground states, we have the following growth estimate on $\cP$. 
\begin{lemma} \label{lem:bup destiny}
For any $N\in\N$, $\al\in(0,\I)$ and $p\in(1,p^\star(N))$, 
there exist $c\in(0,1)$ and $C\in(1,\I)$ such that for any $\de\in(0,1)$ and any solution $\vec u$ of \eqref{NLKG} satisfying $E(\vec u(0))<E(\vec Q)-\de$, $K_0(u(0))<0$ and $\|\vec u(0)\|_\cH\le B$, 
$\cP(\vec u(t))$ is increasing and satisfying 
\EQ{
 \cP(\vec u)\ge e^{ct-CB^2/\de}-1-B^2,} 
as long as the solution $u$ exists. 
\end{lemma}
The lower bound on $\cP$ may go much beyond the blow-up criterion \eqref{bc-P}, which will be useful for a gluing argument in Section \ref{sect:collapse}. 
\begin{proof}
\eqref{mass id} and \eqref{est K0} imply as long as the solution exists
\EQ{ 
 \cP(\vec u(t)) \ge \cP(\vec u(0)) - \int_0^t K_0(u(s))ds
 \ge -\|\vec u(0)\|_\cH^2 + C t \de \ge -B^2 + Ct\de,}
for some constant $C>0$ implicit in the estimate \eqref{est K0}, 
while $E(\vec u)\le\|\vec u\|_\cH^2/2\le B^2/2$. 
Hence for $t\gg B^2/\de$, we obtain $\p_t\cP \ge c\cP/2$ from \eqref{dieq cP} and so the desired estimate. 
\end{proof}

\section{Expansion around multi-solitons}
In order to investigate the dynamics around the $K$-solitons, decomposition is crucial for the solutions around them using the linearized operators. 
Let $\nu^{\pm}:=-\alpha \pm \sqrt{{\alpha}^2+{\nu}_0^2}$ as before, and for $j=1\etc N$,
\EQ{ 
  \pt \nu^j:=0,\pq \nu^{-j}:=-2\al, \pq Y^\pm:=\bar Y^\mp:=\smat{1 \\ \nu^\pm}\phi\in\cH,
  \pr Y^j:=\bar Y^{-j}:=\smat{1 \\ 0} \p_j Q\in \cH, \pq \bar Y^j:=Y^{-j}:=\smat{1 \\ -2\al}\p_j Q\in \cH.}
They are the eigenfunctions of $J\sL^\al$ and $J^*(\sL^\al)^*$ with real eigenvalues, namely
\EQ{
 (J\sL^\al-\nu^\x) Y^\x = 0 =
 (J^*(\sL^\al)^*-\nu^\x) \bar Y^\x, \pq (\x=\pm, \pm 1\etc \pm N).}
In order to distinguish the unstable and center modes, we introduce a few sets of index
\EQ{
 \pt \xu :=\{+\}, \pq \xc:=\{1\etc N\}, \pq \xcu:=\{+,1\etc N\},
 \pr \xscu:=\{\pm,\pm 1,\etc \pm N\}.}
The stable eigenfunctions $Y^-$ and $Y^{-j}$ are used just for some computations. 
Since 
\EQ{
 \p_t\om(u,v) = \om((\p_t-J\sL^\al)u,v) + \om(u,(\p_t+J^*(\sL^\al)^*)v)}
for general $u,v\in C^1_t(I;\cH)$, we have in particular for any $\x\in\xscu$
\EQ{ \label{eq eigen}
 (\p_t-\nu^\x)\om(u,\bar Y^\x) = \om((\p_t-J\sL^\al)u, \bar Y^\x),}
and for $\x,\y\in\xscu$,  
\EQ{ \label{om eig}
 \om(Y^\x,\bar Y^\y) = \CAS{\nu^\pm-\nu^\mp=\pm 2\sqrt{\al^2+\nu_0^2} &(\x=\y=\pm)\\ 
 \pm\frac{2\al}{N}\|\na Q\|_{L^2}^2 &(\x=\y=\pm j)\\ 0 &(\x\not=\y).}}

Let $u$ be a solution of \eqref{NLKG} in $\cH$ on some time interval $I$, and decompose it as 
\EQ{ \label{def v}
\pt Q_k:=\s_k Q(x-z_k), \pq Q_\Si:= \sum_{k=1}^K Q_k, \pq \vec Q_k:=\smat{Q_k \\ 0}, \pq \vec
Q_\Si:=\smat{Q_\Si \\ 0},
 \pr v:=\vec u-\vec Q_\Si,}
for any $K\in\N$, $z=(z_1\etc z_K) \in C^1(I;(\R^N)^K)$ and $\s=(\s_1\etc \s_K)\in\{\pm 1\}^K$. 
Note that $\vec Q_k\not=\smat{Q_k \\ \p_t Q_k}$ when $z_k(t)$ depends on $t$, and similarly for $\vec
Q_\Si$. 
Plugging the above decomposition into the equation \eqref{NLKG} yields for each $k=1\etc K$
\EQ{ \label{eq v}
 \dot v \pt=J(\sL_k^\al v - \vec N_k(v)) - \dot z\cdot \na_z \vec Q_\Si
  \pn=J(\sL_\Si^\al v - \vec N_\Si(v)) - \dot z\cdot \na_z \vec Q_\Si,}
where $\sL_\bu^\al$ and $\vec N_\bu$ with $\bu=1\etc K,\Si$ are defined by 
\EQ{
 \pt \cL_\bu := -\De + 1 - f'(Q_\bu), 
  \pq \sL_\bu^\al :=\bmat{\cL_\bu & 2\al \\ 0 & 1}, 
 \pr f_\Si:=\sum_{l=1}^K f(Q_l),\pq N_\bu(\fy):=f(Q_\Si+\fy)-f_\Si - f'(Q_\bu)\fy, \pq \vec N_\bu(v):=\smat{N_\bu(v_1) \\ 0}.}
The last term of \eqref{eq v} comes from the motion of solitons at $z$: 
\EQ{
 \dot z\cdot \na_z \vec Q_\Si = -\sum_{k=1}^K \sum_{j=1}^N \s_k \dot z_{k,j} \p_{x_j}\vec Q(x-z_k).}

\begin{remark}
Throughout this paper, there are many symbols, such as $Q_\Si$ above, with implicit dependence on the soliton centers and on the soliton signs:
\EQ{
 z\in (\R^N)^K, \pq \s\in\{\pm 1\}^K.}
This is just to avoid too heavy notation. 
In particular, the dependence on $\s$ is rarely explicit in the symbol, because it is always fixed in each context. 
On the other hand, the dependence on $z$ is sometimes written explicitly, as we need to change it. 
\end{remark}

\subsection{Spectral decomposition of the solution}
Next we extract eigenmodes with real eigenvalues around each soliton. 
For any $\x\in\xscu$ and $c\in\R^N$, define the symplectic projection $P^\x(c):\cH\to\cH$ to the eigenmode around $x=c$ by
\EQ{ \label{def P single}
 \pt p^\x(c)\fy:= \om(\fy,\bar Y^\x(x-c))/C^\x_\om \in\R, \pq C^\x_\om:=\om(Y^\x,\bar Y^\x)\in\R\setminus\{0\}, 
 \pr P^\x(c)\fy:=[p^\x(c)\fy]Y^\x(x-c) \in\cH.}
See \eqref{om eig} for the values of $C^\x_\om$. 
The most important is the unstable mode $p^+$, so we introduce a short-cut notation applicable directly to the solutions at any time
\EQ{
 \ap^+_k(z)\fy := p^+(z_k)(\vec \fy-\vec Q_\Si)\in\R, \pq \ap^+(z)\fy := (\ap^+_1(z)\fy\etc \ap^+_L(z)\fy)\in\R^K.}
For any $\dia\subset\xscu$, the corresponding eigenspace is denoted by 
\EQ{
 \Y^\dia(c):=\Span\{Y^\x(x-c) \mid \x\in\dia\}, \pq \Y^\dia(z):=\sum_{k=1}^K \Y^\dia(z_k)}
for any $c\in\R^N$ and $z=(z_1\etc z_K)\in(\R^N)^K$, and the (symplectic) orthogonal complement is denoted by
\EQ{
 \Y_\perp^\dia(z):=\{\fy\in\cH \mid \forall k=1\etc K,\ \forall \x\in\dia,\ p^\x(z_k)\fy=0\}.}
When the distance among the solitons 
\EQ{
 D_z := \min_{1\le k\not=l \le K}|z_k-z_l|} 
is large enough, the unique projection $P_\perp^\dia(z):\cH\to\Y_\perp^\dia(z)$ is defined in the form
\EQ{
 \pt P_\perp^\dia(z) = I - P^\dia(z), \pq P^\dia(z)=\sum_{k=1}^K P^\dia_k(z), 
 \pr P^\dia_k(z)\fy=\sum_{\x\in\dia} [\ti p^\x_k(z)\fy]Y^\x_k, 
 \pq Y^\x_k:=Y^\x(x-z_k),}
for appropriate choice of bounded linear functionals $\ti p^\x_k(z):\cH\to\R$. To determine them, the orthogonality conditions in $\Y_\perp^\dia(z)$ may be written as 
\EQ{
 0 \pt= \om(\fy,\bar Y^\y_l)/C^\y_\om 
  \pn= p^\y(z_l)\fy - \sum_{k=1}^K \sum_{\x\in\dia} \frac{\om(Y^\x_k,\bar Y^\y_l)}{\om(Y^\y,\bar Y^\y)} \ti p^\x_k(z)\fy,}
for $l=1\etc K$ and $\y\in\dia$. 
The exponential decay of the eigenfunctions implies that the second term may be written as multiplication by a square matrix of size $K\#\dia$ with $1$ on the diagonal and $O(\eN(D_z))$ off the diagonal, which is invertible for $D_z\gg 1$. 
In other words, for any $z\in(\R^N)^K$ with $D_z\gg 1$, 
there is a unique linear isomorphism $A^\dia(z)$ of $\R^{K\#\dia}$ such that 
\EQ{
  \pt A^\dia(z)=I+O(\eN(D_z))=(A_{k,l}^{\x,\y}(z))_{1\le k,l\le K,\ \x,\y\in\dia}, 
  \pr \ti p^\x_k(z) = \sum_{l=1}^K \sum_{\y\in\dia} A_{k,l}^{\x,\y}(z) p^\y(z_l),}
and $A^\dia(z)$ is $C^\I$ in $z\in(\R^N)^K$. 
In conclusion, we have 
\begin{lemma} \label{lem:spec coord}
For any $N\in\N$, $\al\in(0,\I)$, $p\in(1,p^\star(N))$ and $K\in\N$, 
there exist $D_\star,C\in(1,\I)$ such that for any $z\in(\R^N)^K$ satisfying $D_z>D_\star$ and any $\dia\subset\xscu$, there is a unique isomorphism between the Hilbert spaces:
\EQ{ \label{coord aga}
 \cH\ni \pt\fy \mapsto (P^\dia(z_1)\fy\etc P^\dia(z_K)\fy, P_\perp^\dia(z)\fy)\in \bigoplus_{k=1}^K\Y^\dia(z_k)\oplus\Y_\perp^\dia(z),}
which may be rewritten in components as
\EQ{  
  \pt P^\dia(z_k)\fy = \sum_{\x\in\dia} [p^\x(z_k)\fy] Y^\x_k, \pq I=\sum_{k=1}^K P^\dia_k(z)+P_\perp^\dia(z),
  \pr P^\dia_k(z)\fy = \sum_{\x\in\dia} [\ti p^\x_k(z)\fy]Y^\x_k,
   \pq \ti p^\x_k(z) = \sum_{l=1}^K \sum_{\y\in\dia} A_{k,l}^{\x,\y}(z) p^\y(z_l),}
where $p^\x(z_k)$ are defined in \eqref{def P single}, while $\ti p^\x_k(z)$ and $A^{\x,\y}_{k,l}(z)$ are
uniquely determined by the above identities. 
$p^\x(z_k)$, $\ti p^\x(z)$ and $A_{k,l}^{\x,\y}(z)$ are all $C^\I$ with respect to $z\in(\R^N)^K$ for $D_z>D_\star$, satisfying 
\EQ{
 \pt \sum_{k,l\in\{1\etc K\}} \sum_{\x,\y\in\dia} [|p^\x(z_k)-\ti p^\x_k(z)| + |A^{\x,\y}_{k,l}(\z) - \de_{k,l}\de_{\x,\y}|] \le C\eN(D_z),
 \pr C^{-1}\|P^\dia(z)\fy\|_\cH \le \sum_{k=1}^K \sum_{\x\in\dia}|p^\x(z_k)\fy| \le C\|P^\dia(z)\fy\|_\cH,}
as well as for $z,z'\in (\R^N)^K$ with $D_z,D_{z'}>D_\star$, 
\EQ{
 \|P^\dia(z)-P^\dia(z')\|_{\op(\cH)} + \|P^\dia_\perp(z)-P^\dia_\perp(z')\|_{\op(\cH)} \le C|z-z'|.}
\end{lemma}
Note that $P^\dia(z_k)$ and $P^\dia_k(z)$ are close to each other but not the same unless $K=1$. 
$P^\dia(z_k)$ is defined with respect to the single soliton at $x=z_k$, so it is just the spatial translate in $\R^N$ by $z_k$ of $P^\dia(0)$, while $P^\dia_k(z)$ is affected by the other solitons at $x=z_l$ of $l\not=k$. 
In particular, $\ti p^\x_k$ and $A^{\x,\y}_{k,l}$ depend also on the choice of $\dia\subset\xscu$.

\subsection{Expansion of the Hamiltonian}
The decomposition $\vec u=\vec Q_\Si+v$ at any fixed time yields the Taylor expansion of the Hamiltonian 
\EQ{ \label{E exp 0}
 E(\vec u) \pt= E(\vec Q_\Si) + \LR{E'(\vec Q_\Si)|v} + \tf12\LR{E''(\vec Q_\Si)v|v} + O(\|v_1\|_{H^1}^3)
  \pr=E(\vec Q_\Si) + \LR{f_\Si-f(Q_\Si)|v} + \tf12\LR{\sL_\Si v|v} + O(\|v_1\|_{H^1}^3),}
where the Sobolev embedding $H^1(\R^N)\subset L^{p+1}(\R^N)$ was used to estimate the remainder, together with the uniform boundedness of $Q_\Si$ in $H^1(\R^N)$ (for fixed $K$). 
The leading term may be further expanded, using 
\EQ{ \label{exp F}
 F(Q_\Si)-\sum_{k=1}^K F(Q_k) \pt= \sum_{k\not=l} f(Q_k)Q_l + O(\sum_{k\not=l}|Q_kQ_l|^{p_2})
 \pr= \sum_{k\not=l} f(Q_k)Q_l + O(\eN(D_z)^{p_2}),}
where $p_2:=\min(2,(p+1)/2)>1$ for $p>1$. Thus we obtain 
\EQ{
  E(\vec Q_\Si) - KE(\vec Q) \pt= \sum_{1\le k<l\le K}\LR{Q_k|Q_l}_{H^1} - \int_{\R^N}[F(Q_\Si)-\sum_{k=1}^K F(Q_k)]  dx
  \pr=\sum_{1\le k<l\le K}\LR{f(Q_k)|Q_l} - \sum_{1\le k\not=l\le K} \LR{f(Q_k)|Q_l} + O(\eN(D_z)^{p_2})
  \pr=V_\s(z) + O(\eN(D_z)^{p_2}),}
where the soliton potential energy $V_\s:(\R^N)^K\to\R$ is defined by
\EQ{ \label{def V}
 \pt V_\s(z):=-\sum_{1\le k<l\le K} \LR{f(Q_k)|Q_l} = -\sum_{1\le k<l\le K} \LR{Q_k|Q_l}_{H^1}
  \prQ=-\sum_{1\le k<l\le K}\s_k\s_l \eN_0(|z_k-z_l|),
 \prq \eN_0(a):=\LR{f(Q)|Q(x-ae_1)}=\LR{f(Q)|e^{x_1}}\eN(a)[1+O(a^{-1})] \sim \eN(a),}
with the asymptotic formula following from the radial symmetry of the ground state $Q$ and its decay behavior \eqref{asy Q}. 
Thus we obtain 
\EQ{ \label{E exp 1}
 E(\vec u) \pt= KE(\vec Q) + V_\s(z) + \tf12\LR{\sL_{\Si}v|v} + O(\eN(D_z)^{p_2}+\|v\|_\cH^3).}

\subsection{Decomposition of the equation}
Next we derive the evolution equations for the components in the spectral decomposition of $v=\vec u-\vec Q_\Si$, namely 
\EQ{
 \pt a^\x_k(t) := p^\x(z_k(t))v(t), \pq \ga(t):=P_\perp^\xcu(z(t))v(t),}
assuming that $D_z\gg 1$ on the time interval $I$. 
\eqref{eq eigen} and \eqref{eq v} imply 
\EQ{ \label{eq a}
 (\p_t-\nu^\x) a_k^\x \pt= p^\x(z_k)[(\p_t-J\sL_k^\al)v + \dot z_k\cdot\na v]
 \pr= p^\x(z_k)[-J\vec N_k(v) + \dot z_k\cdot\na v - \dot z\cdot \na_z \vec Q_\Si],}
where the terms with $\na v$ come from the movement of centers $z_k$ in the operator $p^\x(z_k)$. Thanks to the regularity of eigenfunctions, they are bounded by 
\EQ{
 |p^\x(z_k)[\dot z_k\cdot\na v]| \lec |\dot z_k|\|v\|_\cH,}
while the soliton term is bounded by using the exponential decay
\EQ{
 p^\x(z_k)\dot z\cdot \na_z \vec Q_\Si = \CAS{O(|\dot z|\eN(D_z)) &(\x=+),\\ (\dot z_k)_\x + O(|\dot z|\eN(D_z)) &(\x=1\etc N).} }

Applying $P_\perp^\xcu(z)$ to \eqref{eq v}, we obtain the equation for the radiation
\EQ{ \label{eq ga}
 \dot\ga \pt= P_\perp^\xcu(z)\dot v + \dot P_\perp^\xcu(z)v 
 = P_\perp^\xcu J[\sL_\Si^\al v - \vec N_\Si(v)] + \dot P_\perp^\xcu(z) v
 \pr= J\sL_\Si^\al \ga - P_\perp^\xcu(z) J \vec N_\Si(v) + R_\perp^\xcu(z)v,} 
where the remainder operator $R_\perp^\dia(z):\cH\to\cH$ is defined by 
\EQ{
 R_\perp^\dia(z) \pt:= -P^\dia(z)J\sL_\Si^\al P_\perp^\dia(z) + P_\perp^\dia(z)J\sL^\al_\Si P^\dia(z)-\dot P^\dia(z).}
The invariance of the eigenspace $\Y^\dia(z_k)$ by $J\sL_k^\al$ and $J^*(\sL_k^\al)^*$ allows us to expand
\EQ{
 R_\perp^\dia(z) = -\sum_{k=1}^K \BR{P^\dia_k(z) R_{f,k} P_\perp^\dia(z) + P_\perp^\dia(z) R_{f,k} P^\dia_k(z) - \dot P^\dia_k(z)},}
with 
\EQ{
 R_{f,k} := J\sL_\Si - J\sL_k = \bmat{0 & 0 \\ f'(Q_\Si)-f'(Q_k) & 0}.}
Then using the exponential decay of the solitons and the eigenfunctions, we obtain 
\EQ{
 \|R_\perp^\dia(z)\|_{\op(\cH)} \lec \eN(D_z)+|\dot z|.} 

\subsection{Soliton interactions in the equation}
Next we extract the nonlinear terms that depend only on the solitons, which are also the leading part. Let 
\EQ{ 
 \pt N^0 := f(Q_\Si)-f_\Si, \pq \vec N^0:=\smat{N^0 \\ 0}, \pq \vec N^1_\Si(\fy):=\smat{N^1_\Si(\fy_1) \\ 0},
 \pr N^1_\Si(\fy):=N_\Si(\fy)-N^0=f(Q_\Si+\fy)-f(Q_\Si)-f'(Q_\Si)\fy. }
The Taylor expansion yields
\EQ{ \label{N aprx 0}
  |N^1_\Si(\fy)| \lec \sum_{k=1}^K|Q_k|^{p-2}|\fy|^2+|\fy|^p.}
For the eigenmodes, we have 
\EQ{ \label{N aprx}
 \pt N_k(\fy)-N^0 = N^1_\Si(\fy)+[f'(Q_\Si)-f'(Q_k)]\fy, 
 \pr |p^\x(z_k)J[\vec N_k(\fy)-\vec N^0]| \lec \|\fy\|_\cH^2 + \eN(D_z)\|\fy\|_\cH.  }
The contributions of $N^0$ are computed as 
\EQ{
 \pt p^+(z_k)J\vec N^0 = -\LR{N^0|\phi(x-z_k)}/C^+_\om,
 \pq p^j(z_k)J\vec N^0 = -\LR{N^0|\p_jQ(x-z_k)}/C^1_\om,}
and using 
\EQ{ \label{exp f}
 f(Q_\Si)-f_\Si \pt= \sum_{k\not=l} p|Q_k|^{p-1}Q_l + O(\eN(D_z)^{p_1}), \pq p_1:=\min(2,p/2),}
where $p_1>1$ for $p>2$. This is the main reason that we need $p>2$. 
One may take $p_1=2$ for $p=3$. The above expansion is essentially the same as \eqref{exp F}, but for a different power. 
Thus we obtain
\EQ{ \LR{N^0|\p_jQ_k} \pt= \LR{f(Q_\Si)-f_\Si|\p_jQ(x-z_k)} 
 \pr= \sum_{l\not=k} \LR{f'(Q_k)Q_l|\p_jQ(x-z_k)}+O(\eN(D_z)^{p_1}),}
where we may compute the leading term using the radial symmetry as 
\EQ{ \label{f Q na}
 \LR{f'(Q_k)Q_l|\p_jQ(x-z_k)} \pt= -\s_l\LR{f(Q)|\p_j Q(x-z_l+z_k)} 
 \pr= -\s_l \tf{c_j}{|c|}\LR{f(Q)|\p_1 Q(x-|c|e_1)} = \s_l \tf{c_j}{|c|}\eN_0'(|c|),}
where $c:=z_l-z_k\in\R^N$. Thus we obtain 
\EQ{ \label{pjN0 asy}
 p^j(z_k)J\vec N^0  \pt= -\sum_{l\not=k} \s_l\tf{(z_l-z_k)_j}{|z_l-z_k|} \eN_0'(|z_k-z_l|)/C^1_\om + O(\eN(D_z)^{p_1})
  \pr= -\s_k \p_{z_{k,j}} V_\s(z)/C^1_\om + O(\eN(D_z)^{p_1}).}
In the same way, we obtain
\EQ{
 p^+(z_k)J\vec N^0 =-\sum_{l\not=k}\s_l \eN_+(|z_l-z_k|)/C^+_\om+O(\eN(D_z)^{p_1}),}
with 
\EQ{
 \eN_+(a) := \LR{f'(Q)\phi|Q(x-ae_1)} = \LR{f'(Q)\phi|e^{x_1}}\eN(a)[1+O(a^{-1})] \sim \eN(a).}

In particular, the equation of the unstable modes is estimated as
\EQ{ \label{eq a+ est}
 \pt(\p_t-\nu^+)a^+_k = -p^+(z_k)J\vec N_k(v) + O((\|v\|_\cH+\eN(D_z))|\dot z|)
 \pr=-\sum_{l\not=k}\s_l\eN_+(|z_l-z_k|)/C^+_\om+O(\eN(D_z)^{p_1}+\|v\|_\cH^2+(\|v\|_\cH+\eN(D_z))|\dot z|)
 \pr=O(\eN(D_z)+\|v\|_\cH^2+(\|v\|_\cH+\eN(D_z))|\dot z|). }

For the equation of radiation $\ga$, we have
\EQ{
 \pt \|P^\xcu(z)J\vec N_\Si^1(v)\|_\cH \lec \|v\|_\cH^2, \pq |N_\Si^1(v_1)| \lec |v_1-\ga_1|^2+|\ga_1|^2+|\ga_1|^p,}
and $\||v_1-\ga_1|^2\|_{L^2}\lec |a^+|^2$. Hence, there is some $N(v,z)\in\cH$ such that
\EQ{ \label{eq ga est}
 \pt(\p_t-J\sL_\Si^\al)\ga+P_\perp^\xcu(z)J\vec N^0=:g
 \pr\implies g_1=N_1(v,z), \pq |g_2| \le N_2(v,z)+2|\ga_1|^p,
 \pr \|N(v,z)\|_\cH \lec \|v\|_\cH^2+(\eN(D_z)+|\dot z|)\|v\|_\cH,
 \pq \|J\vec N^0\|_\cH \lec \eN(D_z).}
This decomposition of $g$ is needed only for $p>\tf{N}{N-2}$, where we apply the Strichartz estimate on the pure power part $|\ga_1|^p$. 
For $N\le 2$ or $p\le\tf{N}{N-2}$, we may put $g=N(v,z)$ to satisfy the above estimate, and we do not need the Strichartz estimate. 

\section{Modulation of the centers}
We have some freedom for the choice of centers $z=(z_1\etc z_K)$, but the simplest one to avoid possible growth in the center subspace $\Y^\xc(z)$ is to impose the (symplectic) orthogonality
\EQ{ \label{a0 stat}
 v:=\vec u- \vec Q_\Si \in \Y^\xc_\perp(z),}
or in terms of the projections, $p^j(z_k)v=0$ for all $j=1\etc N$ and $k=1\etc K$. 
From the equation \eqref{eq a} on the eigenmodes, we see that 
\EQ{ \label{modeq z}
 \p_t p^j(z_k)v=0 \iff \dot z_{k,j} = \s_k p^j(z_k)[J\vec N_k(v)-\dot z_k\cdot\na v-\sum_{l\not=k}\dot z_l\cdot\na\vec Q_l].}
Since the coefficients on $\dot z$ on the right side is bounded by $\|v\|_\cH+\eN(D_z)$, we can further solve it for $\dot z$, as long as $\|v\|_\cH\ll 1\ll D_z$. More precisely, there is a $KN$-dim invertible matrix $M(v,z)=I+O(\|v\|_\cH+\eN(D_z))$, for any $v\in\cH$ and $z\in(\R^N)^K$ for which $\|v\|_\cH+\eN(D_z)$ is small enough (depending only on $N,\al,p,K$), such that the system of equations 
\EQ{ \label{eq z}
 (\dot z_k)_{k=1\etc K} = M(v,z)(p^j(z_k)J\vec N_k(v))_{k=1\etc K,\ j=1\etc N}}
is equivalent to \eqref{modeq z}. Moreover, $M(v,z)$ depends smoothly on $v,z$. 
Using \eqref{N aprx} and \eqref{pjN0 asy}, we may extract the leading term in the equation \eqref{eq z} as
\EQ{ \label{eq z asy}
 \dot z_k \pt= \sum_{l\not=k}\s_k\s_l \tf{z_k-z_l}{|z_k-z_l|}\eN_0'(|z_k-z_l|)/C^1_\om+O(\eN(D_z)^{p_1}+\|v\|_\cH^2)
  \pr= -\na_{z_k} V_\s(z)/C^1_\om +O(\eN(D_z)^{p_1}+\|v\|_\cH^2),}
where $V_\s(z)$ is the soliton potential energy defined in \eqref{def V}. 
In other words, the motion of the centers $z$ is approximated by the gradient flow for $V_\s(z)/C^1_\om$. 
In particular we have 
\EQ{ \label{eq z est}
 |\dot z| \lec \eN(D_z)+\|v\|_\cH^2.}
Then the estimates for the equations \eqref{eq a+ est} and \eqref{eq ga est} are simplified to 
\EQ{ \label{eq a+ simp}
 |(\p_t-\nu^+)a^+_k| \lec \eN(D_z)+\|v\|_\cH^2,}
and 
\EQ{ \label{eq ga simp}
 \|N(v,z)\|_\cH \lec \|v\|_\cH^2+\eN(D_z)\|v\|_\cH.} 

On the other hand, by the implicit function theorem, we may configure the centers to satisfy $v\in\Y^\xc_\perp(z)$ at any fixed time. 
For any $\fy\in\cH$, $\s\in\{\pm 1\}^K$ and $z\in(\R^N)^K$, the $z$-dependence is explicitly denoted by 
\EQ{
 \vec Q_k[z_k]:=\s_k Q(x-z_k), \pq \vec Q_\Si[z]:=\sum_{k=1}^K Q_k[z_k], \pq \fy[z] := \fy - \vec Q_\Si[z].}
\begin{lemma} \label{lem:center}
For any $N\in\N$, $\al\in(0,\I)$, $p\in(2,p^\star(N))$ and $K\in\N$, 
there exists $\de_I\in(0,1)$ such that for any $\fy\in\cH$ and $z\in(\R^N)^K$ satisfying 
\EQ{
 D_z \ge 1/\de_I,  \pq \|\fy[z]\|_\cH \le \de_I,}
there exists a unique $\ti z\in(\R^N)^K$ satisfying $|z-\ti z|\le\de_I$ and 
\EQ{
  \fy[\ti z] \in \Y^\xc_\perp(\ti z).}
Moreover, $\ti z$ is $C^1$ with respect to $(\fy,z)$, and for $k=1\etc K$
\EQ{ \label{est center change}
 \pt \|\fy[\ti z]-\fy[z]\|_\cH\lec |z-\ti z| \lec \|P^\xc(z)\fy[z]\|_\cH, 
 \pr |p^+(\ti z_k)\fy[\ti z]-p^+(z_k)\fy[z]| \lec |z-\ti z|[\|\fy[z]\|_\cH+\eN(D_z)],
 \pr \|P^\xcu_\perp(\ti z)\fy[\ti z]-P^\xu_\perp(z)\fy[z]\|_\cH \lec |z-\ti z|+\eN(D_z)\|\fy[z]\|_\cH.}
\end{lemma}
Note that the change of centers affects the unstable modes only in the quadratic order, while the remainder estimates are affected by the soliton interactions, since $P^\xcu_\perp(z)\not=P^\xu_\perp(z)$ on $\Y^\xc_\perp(z)$ for $K\ge 2$.
\begin{proof}
The equations for $\ti z$ to satisfy may be rewritten as 
\EQ{
 \Phi_{k,j}(\ti z) := \om(\vec Q_\Si[z] - \vec Q_\Si[\ti z] + \fy[z], \bar Y^j(x-\ti z_k))=0,}
for $k=1\etc K$, and $j=1\etc N$. 
We have $|\Phi_{k,j}(z)| \sim |p^j(z_k)\fy[z]| \lec \|\fy[z]\|_\cH \le\de_I$, 
\EQ{
 \p_{\ti z_{k,i}} \Phi_{k,j}(\ti z) \pt= -\om(\vec Q_\Si[z]-\sum_{l\not=k}\vec Q_l[\ti z_l]+\fy[z], \p_i\bar Y^j(x-\ti z_k))
 \pr= -\om(\vec Q_k[z_k],\p_i\bar Y^j(x-z_k))+O(|z-\ti z|+\eN(D_z)+\|\fy[z]\|_\cH),}
for $i=1\etc N$, where the leading term is integrated by parts
\EQ{
 -\om(\vec Q_k[z_k],\p_i\bar Y^j(x-z_k)) =  \s_k \om(Y^i,\bar Y^j) = \s_k \de_{i,j}C^1_\om.}
For $l\not=k$, we have 
\EQ{
 \p_{\ti z_{l,i}} \Phi_{k,j}(\ti z) = \om(\s_l Y^i(x-\ti z_l),\bar Y^j(x-\ti z_k)) = O(\eN(D_z)).}
Hence the Jacobian of $\Phi$ is invertible for $|z-\ti z|\le\de_I$ if $\de_I>0$ is chosen small enough. Then the inverse function theorem yields a unique $\ti z$ near $z$ such that $\Phi(\ti z)=0$ and $|\ti z-z|\lec |\Phi(z)|$. 

For the estimate \eqref{est center change} on the unstable modes, we use $\LR{\phi|\na Q}=0$ and the Taylor expansion 
\EQ{
 \fy[\ti z] - \fy[z] \pt=\sum_{k=1}^K \s_k[\vec Q(x-\ti z_k)-\vec Q(x-z_k)]
 \pr=\sum_{k=1}^K \s_k(z_k-\ti z_k)\cdot\na\vec Q(x-\ti z_k) + O(|\ti z-z|^2),}
where the error term is bounded in $\cH$. Hence 
\EQ{
 \pt p^+(\ti z_k)\fy[\ti z]-p^+(z_k)\fy[z] 
 \pr= [p^+(\ti z_k)-p^+(z_k)]\fy[z] + \sum_{l\not=k}\s_l(z_l-\ti z_l)\cdot p^+(\ti z_k)\na\vec Q(x-\ti z_l) + O(|\ti z-z|^2)
 \pr= O(|\ti z-z|\|\fy[z]\|_\cH + |\ti z-z|\eN(D_z) + |\ti z-z|^2).}
Using the bound $|\ti z-z|\lec\|\fy[z]\|_\cH$ in the last error term leads to the desired estimate. 
The estimate on the remainder follows from that on $\fy[\ti z]-\fy[z]$ and 
\EQ{
 P^\xcu_\perp(\ti z) = P^\xu_\perp(\ti z)P^\xc_\perp(\ti z) + O(\eN(D_z))}
in the operator norm on $\cH$. 
\end{proof}

The above lemma implies that if $\fy\in\cH$ is close to a $K$-soliton in the distance of $\cH$:  
\EQ{
 d(\fy) := \inf_{z\in(\R^N)^K,\ D_z\ge R} \|\fy-\vec Q_\Si[z]\|_\cH \ll 1}
for some large fixed $R>1$, then there exists a unique $z\in(\R^N)^K$ near any minimizing point of the above distance such that 
\EQ{
  \fy-\vec Q_\Si[z] \in P^\xc_\perp(z), \pq \|\fy-\vec Q_\Si[z]\|_\cH \sim d(\fy).}

Once the centers $z$ is configured such that $v\in\Y^\xc_\perp(z)$, it is conserved for later time along the solution, as long as  $\|v\|_\cH+\eN(D_z)$ remains small, if and only if $z$ is evolved by the system \eqref{eq z}. 
We will impose the orthogonality $v\in\Y^\xc_\perp(z)$ when investigating the behavior of a single solution. 
When comparing two nearby solutions, we will impose the orthogonality only on one of the solutions and apply the same centers to the other. 
In the collapsing process from 2-soliton to 1-soliton, we need to shift from a coordinate with 2 centers to another with 1 center, where we reconfigure the center for orthogonality at some intermediate time by the above lemma. 

\section{Energy estimates}
In this section we prepare some energy-type estimates, exploiting the damping, and taking account of the variational and the spectral structures. 
For large powers and dimensions ($p>\tf{N}{N-2}$), we also need to use the Strichartz estimate. 

\subsection{Static energy estimate}
First we investigate the static linearized energy, which appears as the quadratic part in expanding the Hamiltonian \eqref{E exp 0}, namely $\LR{\sL_\Si v|v}$. 

In order to decompose any $\fy\in H^1(\R^N)$ around the $K$-solitons, fix a cut-off function $\chi\in C^\I(\R)$ satisfying 
\EQ{ \label{def chi}
 \CAS{\chi(s)=1 &(s\le 1/3), \\ \chi(s)=0 &(s\ge 2/3).} }
With a parameter $R>0$ to be chosen shortly, consider the decomposition of any $\fy\in H^1(\R^N)$ defined by
\EQ{
 \fy = \sum_{k=1}^K \fy_k, \pq 1\le k\le K-1 \implies \fy_k:=\chi(|x-z_k|-R)\fy.}
Let $A_k^R:=\{x\in\R^N\mid R<|x-z_k|<R+1\}$. 
The decomposition is almost disjoint in the support if $R+1<D_z/2$: 
\EQ{ 
 1\le k<K \implies \supp\fy_k\cap\supp\fy_l \CAS{=\empt &(k<l<K),\\ \subset A_k^R &(l=K).} }
Moreover, $A_k^R$ are mutually disjoint for $\N\ni R<D_z/2-1$ and $k<K$. Hence we can choose $R\in(D_z/4,D_z/2-1)$, depending on $\fy$ and $z$, such that 
\EQ{ \label{scarse est}
 \max_{1\le k<K} \|\fy\|_{H^1(A_k^R)}^2 \lec D_z^{-1}\|\fy\|_{H^1(\R^N)}^2,}
provided that $D_z>1$ is large enough. 
Since 
\EQ{
 \|\fy\|_{H^1(\R^N)}^2 = \sum_{k=1}^K \|\fy_k\|_{H^1(\R^N)}^2 + \sum_{1\le k<K}2\LR{(1-\De)\fy_k|\fy_K} }
and the last summand may be bounded by \eqref{scarse est}, we deduce that 
\EQ{ \label{H1 dcp}
 \|\fy\|_{H^1(\R^N)}^2 \sim \sum_{k=1}^K \|\fy_k\|_{H^1(\R^N)}^2,}
for large $D_z>1$. The scalar linearized energy is decomposed similarly
\EQ{ \label{L dcp}
 \LR{\cL_\Si\fy|\fy} \pt= \sum_{k=1}^K \LR{\cL_\Si\fy_k|\fy_k} + 2\sum_{1\le k<K}\LR{\cL_\Si\fy_k|\fy_K}
  \pr= \sum_{k=1}^K \LR{\cL_k\fy_k|\fy_k} + O(D_z^{-1}\|\fy\|_{H^1}^2).}
Applying the linearized energy estimate \eqref{L ene} with the space translation with $z_k$ and summing it over $k=1\etc K$, we obtain 
\EQ{ 
 \sum_{k=1}^K \|\fy_k\|_{H^1}^2 \sim \sum_{k=1}^K\BR{\LR{\cL_k\fy_k|\fy_k} + C|\LR{\fy_k|\phi(x-z_k)}|^2 + C|\LR{\fy_k|\na Q(x-z_k)}|^2}.}
Combining it with \eqref{H1 dcp} and \eqref{L dcp}, we obtain 
\EQ{ \label{cL equiv}
 \|\fy\|_{H^1(\R^N)}^2 \sim \LR{\cL_\Si\fy|\fy} + C \sum_{k=1}^K\BR{ |\LR{\fy|\phi(x-z_k)}|^2 + |\LR{\fy|\na Q(x-z_k)}|^2},}
uniformly for large $D_z>1$, with the same constant $C=C(N,p)>0$ as in \eqref{L ene}. 

Next we consider the vector version of the linearized energy for any $\fy\in\cH$. Applying \eqref{cL equiv} to the first component of $\fy$ yields 
\EQ{ \label{sL equiv}
 \|\fy\|_\cH^2 \sim \LR{\sL_\Si\fy|\fy} + C\sum_{k=1}^K\BR{|\LR{\fy_1|\phi(x-z_k)}|^2 + |\LR{\fy_1|\na Q(x-z_k)}|^2}.}
For the spectral decomposition of $\fy\in\cH$
\EQ{
 y_k:=P^\xscu(z_k)\fy = \sum_{\x\in\xscu}a^\x_k Y^\x_k, \pq \z:=P_\perp^\xscu(z)\fy,}
the linearized energy is expanded as 
\EQ{ 
 \LR{\sL_\Si\fy|\fy} = \sum_{k=1}^K \LR{\sL_ky_k|y_k} + \LR{\sL_\Si\z|\z} + O(\eN(D_z)\|\fy\|_\cH^2),}
using the exponential decay of the solitons and the eigenfunctions, 
\EQ{
 \|\fy-\sum_{k=1}^K y_k-\z\|_\cH\lec \eN(D_z)\|\fy\|_\cH} 
 and the orthogonality of $\z$. 
By the orthogonality of $\z$, \eqref{sL equiv} implies $\LR{\sL_\Si\z|\z} \sim \|\z\|_\cH^2$, 
while the eigenmode part is further expanded by 
\EQ{
 \tf12\LR{\sL_ky_k|y_k} = -\al \nu^+ |a^+_k|^2  - \al \nu^-|a^-_k|^2 - 2\nu_0^2 a^+_k a^-_k + \al C_\om^1 \sum_{j=1}^N |a_k^{-j}|^2.}
Since $\al\nu^->0$ and $\al C_\om^1>0$, adding $C|a^+_k|^2$ with a big constant $C>0$ turns it into a positive quadratic form on $(a^+_k,a^-_k,a^{-1}_k\etc a^{-N}_k)$. 
Therefore, there is a constant $\mm=\mm(\al,N,p)\in[1,\I)$ such that 
\EQ{ \label{def mm}
 \mm\sum_{\x\in\xcu}|a^\x_k|^2 + \LR{\sL_ky_k|y_k} \sim \sum_{\x\in\xscu}|a^\x_k|^2 \sim \|y_k\|_\cH^2.}
Henceforth, this $\mm$ is fixed as a constant (dependent on $N,\al,p$). 
Then for any $z\in(\R^N)^K$ with large $D_z>1$, we may define a uniformly equivalent norm of $\cH$ by
\EQ{ \label{sL est}
 \|\fy\|_{\cH_z}^2 := \LR{\sL_\Si\fy|\fy} + \mm \sum_{k=1}^K \sum_{\x\in\xcu}|p^\x(z_k)\fy|^2 \sim \|\fy\|_\cH^2.}

The linearized energy is expanded for any $\fy\in\cH$ and $z\in(\R^N)^K$ as 
\EQ{
 \tf12\LR{\sL_\Si \fy|\fy} \pt=\tf12\LR{\sL_\Si P^\xcu\fy|P^\xcu\fy}+\LR{\sL_\Si P^\xcu\fy|P^\xcu_\perp\fy}+\tf12\LR{\sL_\Si P^\xcu_\perp\fy|P^\xcu_\perp\fy}
  \pr= -\al\nu^+|a^+|^2-2\nu_0^2 a^+\cdot a^- + \tf12\LR{\sL_\Si\ga|\ga}+O(\eN(D_z)\|\fy\|_\cH^2), }
with $a^\pm_k:=p^\pm(z_k)\fy$ and $\ga:=P^\xcu_\perp\fy$. Note that the $a^-$ component is almost included in $\ga$, so $a^+\cdot a^-$ is a cross term between $P^\xcu\fy$ and $P^\xcu_\perp\fy$. 
Injecting this into the Hamiltonian expansion \eqref{E exp 1} yields 
\EQ{ \label{E exp full}
 E(\vec u) \pt= KE(\vec Q) + V_\s(z) - \al\nu^+ |p^+(z)v|^2 - 2\nu_0^2 p^+(z)v\cdot p^-(z)v 
  \prQ+ \tf12\LR{\sL_\Si P^\xcu_\perp v|P^\xcu_\perp v}+O(\eN(D_z)^{p_2}+\|v\|_\cH^3)
   \pr=KE(\vec Q) + V_\s(z) + \tf 12\|v\|_{\cH_z}^2 - \mm \sum_{k=1}^K \sum_{\x\in\xcu} |p^\x(z_k)v|^2 
 \prQ+ O(\eN(D_z)^{p_2}+\|v\|_\cH^3).}

\subsection{Soliton repulsivity condition}
In the above estimate, we see that if 
\EQ{ \label{sol repul}
 \eN(D_z) \le C_V V_\s(z)}
for some constant $C_V\ge 1$, then for any $z\in(\R^N)^K$ and $v\in\cH$ such that $D_z>1$ is large enough and $\|v\|_\cH$ is small enough, both depending only on $N,\al,p,K,C_V$, 
\EQ{ \label{E exp main}
 E(\vec Q_\Si+v)-KE(\vec Q)+\mm \sum_{k=1}^K \sum_{\x\in\xcu} |p^\x(z_k)v|^2  \sim \|v\|_{\cH_z}^2 + \eN(D_z),}
where the implicit constants of the equivalence depend only on $N,\al,p,K$ and $C_V$. 
If $z$ is chosen to satisfy the orthogonality $v\in \Y^\xc_\perp(z)$, then the smallness of $\|v\|_\cH^2+\eN(D_z)$ is preserved as long as the unstable modes $p^+(z_k)v$ remain small. 
In other words, it can be destroyed only by growth of the unstable modes.  

In this paper, we consider only the cases where this condition \eqref{sol repul} holds a priori and uniformly in time. 
This condition means that the soliton interaction is repulsive in the energy sense. 
It is not true for arbitrary $K$-solitons, but it is satisfied with some absolute constant $C_V$ in the following particular cases
\begin{enumerate}
\item $K=1$
\item $K=2$ and $\s_1\s_2=-1$
\item $N=1$ and each soliton has the opposite sign of its neighbors.
\end{enumerate}
Moreover, \cite{CMYZ} proved that $\s_1\s_2=-1$ is necessary for any asymptotic 2-solitons in any dimension $N\in\N$, and \cite{CMY} proved that the sign alternation as in (3) is necessary for any asymptotic $K$-soliton of any number $K\in\N$ in one dimension $N=1$. 
Hence, if $u$ starts from a small neighborhood of an asymptotic $K$-soliton that is already well-separated, then \eqref{sol repul} holds if $K\le 2$ or $N=1$, as long as $\|v\|_\cH^2+\eN(D_z)$ is small enough. 

For $K\ge 3$ and $N\ge 2$, it seems difficult to have the repulsivity condition a priori in general, but it is possible under some symmetry restrictions. We will consider a few examples in Section \ref{sect:sym Ksol}.

\subsection{Damped energy estimate}
The radiation part $\ga$ on the stable subspace, which includes the continuous spectrum, is controlled by the following decay estimate, both in the solutions and in the difference of two solutions. 

In order to treat higher powers in higher dimensions, we need to use the Strichartz estimate too, when $f$ does not map $H^1(\R^N)$ into $L^2(\R^N)$, namely $p>\tf{N}{N-2}$. Taking advantage of $N\le 6$, however, we may avoid using the dual Strichartz norm, by working solely with the $L^p_tL^{2p}_x$ norm, which is admissible for $p\in[2,p^\star(N))$ and $N\le 6$ (as well as in the energy-critical case $p=p^\star(N)$ for $3\le N\le 6$). For any $\fy\in C((0,T);\cH)$ and $\psi\in L^1((0,T);L^2_x)$, and $0<t_0<t_1<T$, we denote the local Strichartz norm by
\EQ{
 \pt \|\fy\|_{S(t_0,t_1)} := \|\fy_1\|_{L^p_t(t_0,t_1;L^{2p}_x)}, \pq \|\fy\|_{S(t_1)}:=\|\fy\|_{S(\max(0,t_1-1),t_1)}, 
 \pr \|\psi\|_{S^*(t_0,t_1)} := \|\fy\|_{L^{p'}_t(t_0,t_1;L^{2p'}_x)}, \pq \|\psi\|_{S^*(t_1)} := \|\psi\|_{S^*(\max(0,t_1-1),t_1)}, }
where $p'=\tf{p}{p-1}$ is the H\"older conjugate, so $\||\fy_1|^{p-1}\|_{S^*(\cdot)}\le\|\fy\|_{S(\cdot)}^{p-1}$. 

In the simpler case where $f:H^1(\R^N)\to L^2(\R^N)$ or $p\le\tf{N}{N-2}$, we may ignore those  Strichartz norms and the nonlinear part of the following lemma. 
\begin{lemma} \label{lem:linene damp}
For any $N\in\N$, $\al\in(0,\I)$, $p\in(2,p^\star(N))$ and $K\in\N$, there exist $\mu\in(0,\al)$, $\de_\Ga^\star\in(0,1)$ and $C_\Ga\in(1,\I)$ such that the following holds. 
Let $T>0$, $\s\in\{\pm1\}^K$, $z\in C^1([0,T];(\R^N)^K)$, $g\in L^1((0,T);\cH)$, and $\ga\in C([0,T];\cH)$ satisfy
\EQ{ \label{lineq ga}
 (\p_t-J\sL_\Si^\al)\ga = g, \pq \ga\in \Y_\perp^\xcu(z), \pq |\dot z|+\eN(D_z) \le \de_\Ga^\star,}
on $0<t<T$. Then we have for $0<t<T$
\EQ{  \label{ene decay}
 \|\ga(t)\|_\cH + \|\ga\|_{S(t)} \le C_\Ga\BR{e^{-\mu t}\|\ga(0)\|_\cH + \int_0^t e^{-\mu(t-s)}\|g(s)\|_\cH ds}.} 
If the equation \eqref{lineq ga} is replaced with
\EQ{ \label{nonlineq ga}
 (\p_t-J\sL_\Si^\al)\ga = \ti g, \pq \ti g_1=g_1, \pq |\ti g_2| \le |g_2|+B(|\ga_1|^2+|\ga_1|^p+|h\ga|),}
for some space-time functions $\ti g,g,h$ and some constant $B\in(1,\I)$, then we have the same estimate \eqref{ene decay}, provided for $0<t<T$
\EQ{ \label{small gh}
 \text{RHS of \eqref{ene decay}} + \|h\|_{S^*(t)} \le \de_\Ga^\star/B.}
\end{lemma}
\begin{proof}
First we prove the $\cH$ bound in the linear case. 
For a small constant $\ka\in(0,\al)$ to be fixed shortly, let 
\EQ{
 n(t) := \tf12\LR{\sL_\Si^\ka \ga|\ga} - \ka(\al-\ka)\|\ga_1\|_{L^2}^2.}
The orthogonality $\ga\in\Y_\perp^\xcu(z)$ and \eqref{sL equiv} imply that $n\sim\|\ga\|_\cH^2$ if $\ka>0$ is small enough. Fix such $\ka\in(0,\al)$. Using the equation of $\ga$, we have
\EQ{
 \dot n \pt= \LR{\tf12(\sL_\Si^\ka + (\sL_\Si^\ka)^*)\ga - 2\ka(\al-\ka)(\ga_1,0)|J\sL_\Si^\al\ga+g}
 \prQ+\dot z\cdot\tf12\LR{f''(Q_\Si)\na Q_\Si|\ga_1^2} 
 \pr=-\ka n - (\al-\ka)\|\ga_2+\ka\ga_1\|_{L^2}^2 + O(|\dot z|\|\ga\|_\cH^2 + \|\ga\|_\cH\|g\|_\cH)
 \pr\le -\frac{\ka}{2}n + C\|\ga\|_\cH\|g\|_\cH,}
since $|\dot z|\le \de_\Ga^\star$ is small and $n\sim\|\ga\|_\cH^2$. Thus we obtain 
\EQ{
 (\p_t+\ka/4)\sqrt{n} = \tf12 n^{-1/2}(\dot n+\ka n/2) \lec \|g\|_\cH}
and its integral from $t=0$ yields
\EQ{
 \|\ga\|_\cH\sim \sqrt{n} \pt\le e^{-\frac{\ka}{4}t}\sqrt{n(0)} + C\int_0^t e^{-\frac{\ka}{4}(t-s)}\|g(s)\|_\cH ds,}
as desired, with $\mu=\ka/4$. 

It remains to estimate the Strichartz norm. Fix any $t_1\in(0,T)$ and let $t_0:=\max(t_1-1,0)$. 
Applying the Strichartz estimate to the Klein-Gordon equation
\EQ{
 (\p_t-J\smat{-\De+1 & 0 \\ 0 & 1})\ga = g + \smat{0 \\ f'(Q_\Si)\ga_1-2\al\ga_2}, }
on $(t_0,t_1)$, we obtain 
\EQ{
 \|\ga\|_{S(t_0,t_1)} \pt\lec \|\ga(t_0)\|_\cH + \|g\|_{L^1_t(t_0,t_1;\cH)} + \|f'(Q_\Si)\ga_1-2\al\ga_2\|_{L^1_t(t_0,t_1;L^2)}
 \pr\lec \|\ga\|_{L^\I_t(t_0,t_1;\cH)} + \|g\|_{L^1_t(t_0,t_1;\cH)}.  }
Injecting the above bound on $\cH$ yields the same bound on $\|\ga\|_{S(t)}$. 

Next we prove the nonlinear case with the weaker decay exponent $e^{-\frac{\mu}{2}t}$ and using the linear estimate \eqref{ene decay}.  Replacing the constant $\mu>0$ in the end, we get the conclusion with the same $\mu>0$ in both the linear and the nonlinear cases.  

Let $\cN:=|\ga_1|^2+|\ga_1|^p+|h\ga_1|$. On one hand, using the free Strichartz estimate as above, we obtain for small $t_0\in(0,T)$ 
\EQ{
 \|\ga\|_{L^\I_t(0,t_0;\cH\cap S(t))} \lec \|\ga(0)\|_\cH + \|g\|_{L^1_t(0,t_0;\cH)} + t_0\|\ga\|_{L^\I_t(0,t_0;\cH)} + B\|\cN\|_{L^1_t(0,t_0;L^2)},}
where the last term is bounded by H\"older
\EQ{
 \|\cN\|_{L^1_t(I;L^2_x)} \lec [\|\ga\|_{S(I)} + \|\ga\|_{S(I)}^{p-1} + \|h\|_{S^*(I)}]\|\ga\|_{S(I)},}
which holds uniformly on any interval $I\subset[0,T]$ with length $|I|\le 1$. Hence by the dominated convergence, we obtain for some small $t_0>0$ 
\EQ{ \label{ga small t}
 \|\ga\|_{L^\I_t(0,t_0;\cH\cap S(t))} \lec \|\ga(0)\|_\cH + \|g\|_{L^1_t(0,t_0;\cH)}.}
On the other hand, applying \eqref{ene decay} to \eqref{nonlineq ga} yields
\EQ{ \label{damp nonlin case}
 \|\ga(t)\|_\cH + \|\ga\|_{S(t)} \le C_\Ga e^{-\mu t}\|\ga(0)\|_\cH + C_\Ga \int_0^t e^{\mu(s-t)}[\|g(s)\|_\cH+B\|\cN(s)\|_{L^2}]ds. }

Now for some big constant $C_0>C_\Ga$, let $T_0\in[0,T]$ be the maximal time such that for $0<t<T_0$ we have 
\EQ{
 \|\ga(t)\|_\cH+\|\ga\|_{S(t)} \le C_0 M(t), \pq M(t):=e^{-\frac\mu 2 t}\|\ga(0)\|_\cH + \int_0^t e^{-\frac{\mu}{2}(t-s)}\|g(s)\|_\cH ds,}
namely the desired estimate (with the worse constants than the linear case). 
\eqref{ga small t} implies that $T_0>0$ if $C_0$ is larger than some constant depending only on $N$. 
Then for $0<t_0<t_1<T_0$, we have using \eqref{est cN} and the assumption \eqref{small gh} with $\mu/2$,
\EQ{ \label{est cN}
 B\|\cN\|_{L^1_t(t_0,t_1;L^2_x)}\lec C_0C_\Ga^{-1} \de_\Ga^\star\|\ga\|_{S(t_0,t_1)},}
and so 
\EQ{
 C_\Ga B \int_0^t e^{\mu s}\|\cN(s)\|_{L^2_x}ds \pt\lec C_0 \de_\Ga^\star e^{\mu t}\|\ga\|_{S(t)}+C_0 \de_\Ga^\star \int_0^t e^{\mu s}\|\ga\|_{S(s)}ds
 \pr\lec C_0^2 \de_\Ga^\star e^{\mu t}M(t) + C_0^2 \de_\Ga^\star \int_0^t e^{\mu s}M(s)ds,}
and the last integral is estimated directly using Fubini
\EQ{
 \pt\int_0^t \int_0^s e^{\frac{\mu}{2}(2s-s+s_1)}\|g(s_1)\|_\cH ds_1ds
 \le \tf{2}{\mu}\int_0^t e^{\frac{\mu}{2}(t+s_1)}\|g(s_1)\|_\cH ds_1
 \lec e^{\mu t}M(t).}
Thus we obtain 
\EQ{
 C_\Ga B \int_0^t e^{\mu(s-t)}\|\cN(s)\|_{L^2_x}ds \lec C_0^2 \de_\Ga^\star M(t),}
and injecting it into \eqref{damp nonlin case}, 
\EQ{
 \|\ga(t)\|_\cH + \|\ga\|_{S(t)} \le C_\Ga M(t) + C_1 C_0^2 \de_\Ga^\star M(t),}
for some constant $C_1\in(0,\I)$ depending only on $N,\al$. 
After fixing $C_0\ge 2C_\Ga$ such that the above argument works, we may take $\de_\Ga^\star$ small enough such that $C_1C_0\de_\Ga^\star<1/2$. Then the above estimate implies 
\EQ{
 \|\ga(t)\|_\cH + \|\ga\|_{S(t)} \le (C_\Ga+C_0/2)M(t) < C_0M(t),}
and the maximality of $T_0$ implies that $T_0=T$. 
This means the desired estimate \eqref{ene decay} with $C_\Ga$ replaced with $C_0$ and $\mu$ replaced with $\mu/2$ from the linear case. 
\end{proof}

\section{Full-time dynamics near repulsive multi-solitons}
In this section, we investigate all possible behavior of solutions while they are in a small neighborhood around $K$-solitons of sufficient distance. 
We assume in this section that $u$ is a solution of \eqref{NLKG} in $\cH$ near a $K$-soliton at $z\in(\R^K)^N$ on some time interval $I$, where $u$ is decomposed as before in \eqref{def v} with 
\EQ{
 \pt \vec u(t) = \vec Q_\Si(t) + v(t), 
 \pq v\in \Y^\xc_\perp(z), \pq a^+_k:=p^+(z_k)v, \pq \ga:=P_\perp^\xcu(z)v,}
and that $\|v\|_\cH$ is small enough and $D_z$ is large enough on $I$. 
We also assume that the soliton repulsivity \eqref{sol repul} holds on $I$. 
As mentioned above, it is true for every asymptotic $K$-soliton with $K\le 2$ or $N=1$ by \cite{CMYZ,CMY}. 

Note that the orthogonality $v\in \Y^\xc_\perp(z)$ simplifies \eqref{sL est} to 
\EQ{
 \|v\|_{\cH_z}^2 = \LR{\sL_\Si v|v} + \mm|a^+|^2 \sim \|v\|_\cH^2,}
where $a^+:=(a^+_1\etc a^+_K)\in\R^K$. 
Hence, in view of the energy expansion \eqref{E exp 1}, we use the following energy quantity to control both the remainder $v$ and the soliton distance
\EQ{
 \sN(\vec u,z) \pt:= E(\vec u)-KE(\vec Q) + \mm|\ap^+(z)\vec u|^2 
 \pr=  \tf12\|v\|_{\cH_z}^2 + V_\s(z) + O(\eN(D_z)^{p_2} + \|v\|_\cH^3)
 \pn\sim \|v\|_{\cH}^2 + \eN(D_z),}
as long as $\|v\|_\cH\ll 1\ll D_z$, under the soliton repulsivity condition \eqref{sol repul}. 
More precisely, there are a small constant $\de_E^\star\in(0,1)$ and a large constant $C_E\in(1,\I)$, depending only on $N,\al,p,K,C_V$, such that 
\EQ{ \label{bd CE}
 \sN(\vec u,z)\le |\de_E^\star|^2 \implies \CAS{C_E^{-1}\sN(\vec u,z) \le |a^+|^2+\|v\|_\cH^2+\eN(D_z) \le C_E \sN(\vec u,z), \\ \tf34\sN(\vec u,z) \le \tf12\|v\|_{\cH_z}^2 + V_\s(z) \le \tf54\sN(\vec u,z),
 \\ D_z>D_\star,} }
where $D_\star$ is the constant in Lemma \ref{lem:spec coord} to enable the spectral decomposition. 
The basic estimates on the equations in \eqref{eq z est}--\eqref{eq ga simp} may be collected as 
\EQ{ \label{est eqs}
 |\dot z| + |(\p_t-\nu^+)a^+| + \|\cN(v,z)\|_\cH \lec \|v\|_\cH^2 + \eN(D_z) \sim \sN(\vec u,z).}

The advantage of $\sN(\vec u,z)$ is its simple time evolution, but to measure the closeness to the solitons together with their separation, it is more convenient to use 
\EQ{
 \U\sN_0(\fy) := \inf_{z\in(\R^N)^K} \sN_0(\fy,z), \pq \sN_0(\fy,z):=\|\fy-\sum_{k=1}^K \s_k \vec Q(x-z_k)\|_\cH^2 + \eN(D_z),}
then $\U\sN_0:\cH\to[0,\I)$. If $\U\sN_0(\fy)$ is small enough, depending only on $N,\al,p,K$, then Lemma \ref{lem:center} yields a unique $z\in(\R^N)^K$ such that 
\EQ{
 \fy[z]=\fy-\vec Q_\Si\in \Y^\xc_\perp(z), \pq \sN_0(\fy,z) \sim \U\sN_0(\fy).}

Henceforth we will need many positive numbers depending only on $N,\al,p,K,C_V$. 
For brevity we will often call them constants, ignoring the dependence, since those parameters $N,\al,p,K,C_V$ are fixed. 

\subsection{Before instability}
First we consider the solution $u$ satisfying initially at $t=0$ 
\EQ{  \label{cond 0}
 2\de_0 \mm^{1/2}|a^+|  < \sN(\vec u,z) < \de_0^2.}
for some $\de_0\in(0,\de_E^\star]$. The right inequality implies trivially
\EQ{ \label{init energy bd}
 E(\vec u)-KE(\vec Q) = \sN(\vec u,z)-\mm|a^+|^2 < \de_0^2,}
which is preserved for $t>0$. Then for $t>0$ and as long as 
\EQ{ \label{region 0}
 2\de_0 \mm^{1/2}|a^+| \le \sN(\vec u,z) \le 2\de_0^2,}
holds, we have 
\EQ{
 \de_0^2 > E(\vec u)-KE(\vec Q) \ge (1-(2\de_0)^{-2}\sN(\vec u,z))\sN(\vec u,z) \ge \sN(\vec u,z)/2,}
so the right inequality of \eqref{region 0} is preserved, as long as the left one. 

\subsection{Instability at $O(\de^2)$} \label{ss:inst order2}
Next we consider the solution $u$ satisfying initially at $t=0$ 
\EQ{ \label{region 1}
 B_1|a^+|^2 \le \sN(\vec u,z) \le \de_1|a^+|,}
for some $B_1\in[2\mm,\I)$ and $\de_1\in(0,\de_1^*]$, where $\de_1^*\in(0,\de_E^\star]$ is a small constant to be determined. 
The above condition roughly means 
\EQ{
 \|v\|_\cH^2+\eN(D_z) \ll |a^+| \ll \|v\|_\cH+\eN(D_z)^{1/2},}
that the unstable modes have some marginal size but much smaller than the others. 
Note that this size comparison is not essentially affected by the modulation of centers, because of \eqref{est center change}. 
\eqref{region 1} also implies smallness for all the components
\EQ{ \label{bd region 1}
 |a^+| \le \de_1/B_1, \pq \sN(\vec u,z) \le \de_1|a^+| \le \de_1^2/B_1.}
In particular, $\de_1^*\le\de_E^\star$ implies 
\EQ{
 \sN(\vec u,z) \le \de_1^2 \le |\de_E^\star|^2}
in this region. The left inequality of \eqref{region 1} may be rewritten as 
\EQ{ \label{linq}
 (B_1-\mm)|a^+|^2 \le E(\vec u) - KE(\vec Q)}
or
\EQ{ \label{e-bd reg1}
  \sN(\vec u,z) \le \tf{B_1}{B_1-\mm}[E(\vec u)-KE(\vec Q)] \le 2[E(\vec u)-KE(\vec Q)].} 

As long as the right inequality of \eqref{region 1} is valid, we have from \eqref{est eqs}
\EQ{ \label{eq a reg1}
 |(\p_t-\nu^+)a^+| \lec \sN(\vec u,z) \le \de_1|a^+|,}
hence fixing $\de_1^*>0$ small enough ensures 
\EQ{ \label{grw a reg1}
 |(\p_t-\nu^+)a^+| \le \tf{\nu^+}{2}|a^+|, \pq
 \p_t|a^+|^2 = 2a^+\cdot\dot a^+ \ge \nu^+|a^+|^2.}
In particular, $|a^+|$ is monotonically and exponentially growing. Since the right side of \eqref{linq} is non-increasing, \eqref{linq} (or the left inequality of \eqref{region 1}) has to break down in finite time, while 
\EQ{ \label{E/a mono}
  \p_t \frac{\sN(\vec u,z)}{|a^+|}=\frac{\p_tE(\vec u)}{|a^+|}+\frac{2\mm|a^+|^2-\sN(\vec u,z)}{|a^+|^2}\p_t|a^+| \le 0}
preserves the right inequality of \eqref{region 1}. 

\subsection{Instability at $O(\de)$}  \label{ss:inst-region2}
Next we consider the solution $u$ satisfying initially at $t=0$ 
\EQ{ \label{region 2}
 \sN(\vec u,z) \le B_2|a^+|^2, \pq |a^+| \le \de_2,}
for some $B_2\in[2\mm,\I)$ and $\de_2\in(0,B_2^{-1}\de_2^*]$, where $\de_2^*\in(0,\de_E^\star]$ is a small constant to be determined. 
The above condition roughly means 
\EQ{
 \|v\|_\cH+\eN(D_z)^{1/2} \lec |a^+| \ll 1,}
that the unstable modes are dominant though their absolute size is still small. 
$\de_2\le B_2^{-1}\de_2^*$ and $\de_2^*\le\de_E^\star$ imply
\EQ{
 \sN(\vec u,z) \le B_2|a^+|^2 \le B_2\de_2^2 \le (\de_2^*)^2 \le |\de_E^\star|^2}
in this region. 
We have seen above that the solutions starting from the region \eqref{region 1} enter in this region in finite time. 
As long as \eqref{region 2} is valid, we have from \eqref{est eqs}
\EQ{
 |(\p_t-\nu^+)a^+| \lec \sN(\vec u,z) \le B_2|a^+|^2 \le B_2\de_2^2 \le \de_2^* \de_2.}
Let $a^+_*(t):=e^{\nu^+t}a(0)$ and assume 
\EQ{ \label{boot asm}
 |a^+|\le 2|a^+_*|\le 4|a^+| \le 4\de_2} 
for some time. Then from the Duhamel formula we obtain 
\EQ{
 \pt |a^+-a^+_*|+|\dot a^+/\nu-a^+_*| 
  \lec \int_0^t e^{\nu^+(t-s)}B_2|a^+(s)|^2 ds
 \lec B_2|a^+|^2 \lec B_2\de_2|a^+|,}
Taking $\de_2^*>0$ small enough ensures that the above inequality implies 
\EQ{ \label{grw a reg2}
 |a^+-a^+_*|+|\dot a^+/\nu-a^+_*| \le \tf13|a^+_*|,}
and so 
\EQ{
 \pt \tf 23 \le |a^+|/|a^+_*|,\ |\dot a^+/\nu^+|/|a^+_*| \le \tf 43, 
 \pq \p_t|a^+|^2 = 2a^+\cdot\dot a^+ \ge \tf89\nu^+|a^+_*|^2.}
In particular, the first inequalities are stronger than \eqref{boot asm} except $|a^+|\le\de_2$. 
Moreover, the last inequality implies that $|a^+|$ is increasing, while the left inequality of \eqref{region 2} is equivalent to 
\EQ{
 E(\vec u)-KE(\vec Q) \le (B_2-\mm)|a^+|^2,}
where the left side is non-increasing. Thus we deduce that all the above inequalities are valid as long as $|a^+|\le\de_2$, which has to break down in finite time, since $|a^+|$ is exponentially increasing.

On the other hand, the centers are almost fixed, since \eqref{est eqs} implies 
\EQ{
 |z_k-z_k(0)| \lec \int_0^t B_2|a^+_*(s)|^2 ds \lec B_2|a^+|^2 \le \de_2^* |a^+| \le \de_2^* \de_2.}
The energy decay Lemma \ref{lem:linene damp} applied to \eqref{eq ga est} with \eqref{est eqs} yields  (for $\de_2^*$ small enough)
\EQ{
 \|\ga\|_{\cH\cap S(t)} \lec e^{-\mu t}\|\ga(0)\|_\cH + \int_0^t e^{-\mu(t-s)}B_2|a^+(s)|^2 ds \lec e^{-\mu t}\|\ga(0)\|_\cH + B_2|a^+|^2.}
So the growth of the other components is much slower than $|a^+|$. 
In particular, 
\EQ{
 \eN(D_z)+\|\ga\|_\cH^2 \lec \sN(\vec u(0),z(0)) + (B_2|a^+|^2)^2 }
and so the energy \eqref{E exp full} is estimated by 
\EQ{ \label{energy eject}
 E(\vec u) - K E(\vec Q) \le -\tf{\al\nu^+}{2}\BR{|a^+|^2 - (2c_E)^{-2}\sN(\vec u(0),z(0))} }
for some constant $c_E\in(0,1)$, provided that $\de_2^*>0$ is chosen small enough. 
The way putting the constant $c_E$ may look artificial, but it is to simplify a condition between $\de_0$ and $\de_2$, namely \eqref{energy dec}.

\subsection{Exponential instability}
We may combine the above three regions to cover all initial data around the $K$-soliton $\vec Q_\Si$, namely 
\EQ{ \label{init region}
 \sN(\vec u,z)<\de_0^2,}
if we choose the parameters $\de_0,\de_1,\de_2,B_1,B_2$ such that 
\EQ{ \label{cond param 1}
 \pt 2\mm\le B_1 \le B_2, \pq \de_1\le \de_1^*, \pq \de_2 \le B_2^{-1}\de_2^*,
 \pq \de_0 \le  \min(\de_E^\star, \tf12 \mm^{-1/2}\de_1).}
Let $u$ be a solution satisfying initially \eqref{init region} at $t=0$, and let $T^*\in(0,\I]$ be its maximal existence time. 
First, we have \eqref{init energy bd} for all $t\in[0,T^*)$. Definte $T_1\in[0,T^*]$ by 
\EQ{
 T_1 := \inf\{t\in[0,T^*) \mid \sN(\vec u,z) \le \de_1|a^+|\}.}
If the set is empty, then let $T_1:=T^*$. 

For $0\le t<T_1$, we have $\sN(\vec u,z)>\de_1|a^+|\ge 2\de_0\mm^{1/2}|a^+|$, so the solution $u$ is in the region \eqref{region 0} before instability. If $T_1<T^*$, then define $T_2\in[T_1,T^*]$ by 
\EQ{
 T_2 := \inf\{t\in[T_1,T^*) \mid \sN(\vec u,z)\le B_2|a^+|^2\}.}
For $T_1\le t<T_2$, we have $B_1|a^+|^2\le B_2|a^+|^2 \le \sN(\vec u,z)\le\de_1|a^+|$, so the solution is in the region \eqref{region 1} of instability at $O(\de^2)$, since in Section \ref{ss:inst order2} we have seen that $\sN(\vec u,z)\le\de_1|a^+|$ is preserved by \eqref{E/a mono}. 
Moreover $\sN(\vec u,z)\ge B_1|a^+|^2$ has to break down in finite time, which implies $T_2<T^*$. Define $T_3\in[T_2,T^*]$ by 
\EQ{
 T_3:=\inf\{t\in[T_2,T^*) \mid |a^+|\ge \de_2\}.}
If the set is empty, then let $T_3:=T^*$. 
For $T_2\le t<T_3$, the solution is in the region \eqref{region 2} for the instability of $O(\de)$, since $\sN(\vec u,z)<B_2|a^+|^2$ is preserved as long as $|a^+|\le\de_2$, as observed in Section \ref{ss:inst-region2}, and moreover, $T_3<T^*$. 

The starting time $T_1$ of exponential growth may be detected by another choice of centers, because of \eqref{est center change}: If for some $T\in[0,T^*)$ and $\ti z\in(\R^N)^K$
\EQ{
 \sN_0(\vec u(T),\ti z) \ll |\ap^+(\ti z)\vec u(T)| \ll 1, }
then $T_1\le T$, and also $|a^+-\ap^+(\ti z)\vec u(T)| \ll |\ap^+(\ti z)\vec u(T)|$. 

Notice the qualitative difference between the estimates on $[T_1,T_2)$ and on $[T_2,T_3)$, namely \eqref{grw a reg1} and \eqref{grw a reg2}. 
The latter also preserves the direction $a^+/|a^+|$ essentially, but the former does not. 
So we need more detailed information on the behavior of solutions for $t<T_2$, in particular, distinguishing two regimes depending whether the interaction between the solitons is dominant or not. 
The following argument applies partly to the region \eqref{region 2} as well, namely for $t>T_2$, but it is less important in the latter region, as the exponential growth is dominating everything else. 

\subsection{Radiation dominant dynamics} \label{ss:rad dom}
Consider a solution $u$ satisfying on a time interval $[0,T]$
\EQ{ \label{region rad}
 \de_r^{-1}\max(|a^+|,\eN(D_z)) \le \sN(\vec u,z)^{1/2} \le \de_r}
for some $\de_r\in(0,\de_r^*]$, where $\de_r^*\in(0,\de_E^\star]$ is a small constant to be determined depending only on $N,\al,p,K,C_V$, while $\de_E^\star$ is the constant in \eqref{bd CE}.  This condition roughly means 
\EQ{
 |a^+| + \eN(D_z) \ll \|v\|_\cH \ll 1.}
In particular it implies 
\EQ{ \label{ga dom}
 \|v\|_\cH \le \|\ga\|_\cH + C(1+O(\eN(D_z))|a^+| \le  2\|\ga\|_\cH,}
if $\de_r^*>0$ is chosen small enough. 
Then we have from \eqref{est eqs} 
\EQ{
 |\dot z| + \|N(v,z)\|_\cH \lec \sN(\vec u,z) \lec \de_r\|\ga\|_\cH}
Applying Lemma \ref{lem:linene damp} to $\ga$, and using \eqref{region rad}, we obtain via a bootstrapping argument 
\EQ{ \label{ga radecay}
  \|v\|_{\cH\cap S(t)} \sim \|\ga\|_{\cH \cap S(t)} \lec e^{-\frac{\mu}{2} t}\|\ga(0)\|_\cH.}
The detail of the bootstrapping is as follows. Let $T'\in(0,T]$ be the maximal time such that for $0\le t\le T'$ we have $\|\ga\|\le Be^{-\frac{\mu}{2}t}\|\ga(0)\|_\cH$ for some large constant $B\in(1,\I)$. Then the lemma yields 
\EQ{
 \|\ga\|\lec e^{-\mu t}\|\ga(0)\|_\cH + \int_0^t e^{\mu(s-t)}[\|J\vec N^0\|_\cH+\|N(v,z)\|_\cH](s) ds,}
for $0\le t\le T'$, and the last integral is bounded by 
\EQ{
 \int_0^t e^{\mu(s-t)}\sN(\vec u,z)(s)ds 
 \lec \int_0^t \de_r Be^{\mu(s-t)-\frac{\mu}{2}s}\|\ga(0)\|_\cH ds
 \lec \de_r B e^{-\frac{\mu}{2}t}\|\ga(0)\|_\cH.}
Hence if $B>1$ is large enough and $\de_r\le\de_r^*$ is small enough, then we obtain 
$\|\ga\|\le \tf{B}{2}e^{-\frac{\mu}{2}t}\|\ga(0)\|_\cH$ for $0\le t\le T'$, which implies $T'=T$ because of the maximality of $T'$. Thus we obtain \eqref{ga radecay} on $[0,T]$. 
Then integrating $\dot z$ with the above estimates yields
\EQ{
  |z-z(0)| \lec \de_r\|\ga(0)\|_\cH \lec \de_r^2,}
which is small by choice of $\de_r$. 
Thus $\|v\|_\cH$ is exponentially decaying, while $z$ is not essentially moving. 
Hence the upper bound on $\eN(D_z)$ in \eqref{region rad} has to break down in finite time. Indeed, we obtain an upper bound on $T$, depending on $\|v(0)\|_\cH/\eN(D_z(0))$: 
\EQ{
 \eN(D_z(0)) \lec \eN(D_z(T)) \lec \|v(T)\|_\cH \lec \|v(0)\|_\cH e^{-\frac{\mu}{2} T}.}
However \eqref{region rad} may well break down from the upper bound on $|a^+|$ before this bound.  
In the region \eqref{region 1} with $\de_rB_1\ge 1$, the condition \eqref{region rad} can break down only by $\eN(D_z)\le \de_r\sN(\vec u,z)$. 

\subsection{Soliton dominant dynamics} \label{ss:soliton dynamics}
Next consider a solution $u$ satisfying on a time interval $[0,T]$
\EQ{ \label{region sol}
 \|v\|_{\cH_z}^2 < V_\s(z), \pq \sN(\vec u,z) \le \de_s^2}
for some small $\de_s\in(0,\de_E^\star]$. The smallness requirement is to be determined in this subsection. 
Since $\de_s\le\de_E^\star$, \eqref{bd CE} implies 
\EQ{
 \tf12\sN(\vec u,z) \le V_\s(z) \le \tf54\sN(\vec u,z).}
On the other hand, \eqref{eq z asy} implies 
\EQ{ \label{eq V}
 \p_t C^1_\om V_\s(z) = \na_z V_\s(z)\cdot\dot z = -|\na_z V_\s(z)|^2 + O(\eN(D_z)(\eN(D_z)^{p_1}+\|v\|_\cH^2)),}
so, using \eqref{sol repul},  
\EQ{ 
 |\p_t V_\s(z)| \lec \eN(D_z)^2 + \eN(D_z)\|v\|_\cH^2 \lec \de_s^2 V_\s(z).}
In particular, choosing $\de_s>0$ small enough ensures 
\EQ{ \label{slow decay V}
 \de:=\min(\nu^+,\mu)/2 \implies V_\s(z(s)) \le e^{\de|s-t|}V_\s(z(t))}
for $s,t\in [0,T]$. Using this with \eqref{est eqs} in the Duhamel formula for $a^+$ from $t=T$, we obtain
\EQ{
 |a^+| \pt\lec e^{\nu^+(t-T)}|a^+(T)| + \int_t^T e^{\nu^+(t-s)}e^{\de(s-t)}V_\s(z(t))ds
   \pr\lec e^{\nu^+(t-T)}|a^+(T)| +  V_\s(z).}
Similarly, using Lemma \ref{lem:linene damp} as well, we obtain
\EQ{
 \|\ga\|_{\cH\cap S(t)} \pt\lec e^{-\mu t}\|\ga(0)\|_\cH + \int_0^t e^{-\mu(t-s)} e^{\de(t-s)}V_\s(z(t))ds
  \pr\lec  e^{-\mu t}\|\ga(0)\|_\cH +  V_\s(z).}
Adding the above two estimates, we obtain 
\EQ{ \label{bd v sol}
 \|v\|_{\cH\cap S(t)} \lec e^{\nu^+(t-T)}|a^+(T)| +  e^{-\mu t}\|\ga(0)\|_\cH +  V_\s(z),}
which means that $v$ can essentially be grown only by $|a^+|$. 

In particular, if we have for some constant $B_s\in[1,\I)$
\EQ{
 \|v(0)\|_\cH \le B_sV_\s(z(0)), }
as well as $|a^+|\ll\|v\|_\cH$ such that \eqref{ga dom} works on $[0,T]$, then, using \eqref{slow decay V} as well, we obtain 
\EQ{
 \|v\|_{\cH_z} \lec \|\ga\|_\cH \lec e^{-\mu t}\|\ga(0)\|_\cH + V_\s(z) \lec B_s  V_\s(z),}
so that we can conclude $\|v\|_{\cH_z}^2<V_\s(z)$ on $[0,T]$ provided that
$B_s^2 \de_s>0$ is small enough. 

Using the above estimates, we may now combine those two dynamics in the regions \eqref{region 0} and \eqref{region 1}. Let $u$ be a solution on an interval $[0,T]$ satisfying 
\EQ{
 B_1|a^+|^2 \le \sN(\vec u,z) \le \de_s^2,}
with the parameters $B_1,\de_s,\de_r$ satisfying 
\EQ{ \label{cond param 2}
 0<\de_r\le\de_r^*, \pq \max(2\mm,\de_r^{-1})\le B_1, \pq 0<\de_s \le \de_s^* \de_r^2,}
for some small constant $\de_s^*\in(0,1)$ to be determined. 
Define $T_s\in[0,T]$ by 
\EQ{
 T_s:=\inf\{t\in[0,T]\mid \eN(D_z)\ge \de_r\sN(\vec u,z)^{1/2}\}.} 
If the set is empty, then let $T_s:=T$. 
For $0\le t<T_s$, we have $\eN(D_z)<\de_r\sN(\vec u,z)^{1/2}$, so the radiation dominant dynamics \eqref{region rad} holds. 
If $T_s<T$, then we have at $t=T_s$
\EQ{
 V_\s(z) \gec \eN(D_z) \ge \de_r\sN(\vec u,z)^{1/2} \gec \de_r\|v\|_\cH,}
so that we can find some $B_s\lec \de_r^{-1}$ such that 
\EQ{
 \|v(T_s)\|_\cH \le B_s V_\s(z(T_s)).}
Then by the above argument, if $B_s^2\de_s>0$ is small enough then we have the soliton dominant dynamics \eqref{region sol} on $[T_s,T]$. The smallness condition amounts to choosing $\de_s^*>0$ small enough.

\subsection{Summary of the dynamics near repulsive $K$-solitons}
Gathering the above information, we obtain a full description of dynamics in a small neighborhood of any $K$-soliton of sufficient distance, under the soliton repulsivity condition \eqref{sol repul}. 

Now we fix the parameters except $\de_0,\de_2$, first by putting
\EQ{
 \de_1:=\de_1^*, \pq  \de_r=\de_r^*, \pq \de_s:=\sqrt{2}\de_0, \pq B_2:=B_1:=\max(2\mm,1/\de_r^*)=:B_1^*,}
where small positive constants $\de_1^*,\de_r^*,\mm$ were chosen at Section \ref{ss:inst order2}, Section \ref{ss:rad dom}, and \eqref{def mm}, respectively. 
Then the conditions from \eqref{cond param 1} and \eqref{cond param 2} are reduced to 
\EQ{ \label{param cond all}
 \pt \de_0\le \min(\de_E^\star,\tf12\mm^{-1/2}\de_1^*,\tf{1}{\sqrt{2}}\de_s^*(\de_r^*)^2), \pq \de_2\le \de_2^*/B_1^*,}
where small positive constants $\de_E^\star,\de_s^*,\de_2^*$ were chosen at \eqref{bd CE}, \eqref{cond param 2} and Section \ref{ss:inst-region2}, respectively. 
Let $\de_2^\star$ be the minimum of the right sides in \eqref{param cond all}.

\begin{theorem}  \label{thm:dyn near Ksol}
For any $N\in\N$, $\al\in(0,\I)$, $p\in(2,p^\star(N))$, $K\in\N$ and $C_V\in[1,\I)$,  
there exist $B_1^*\in[2\mm,\I)$, $\de_1^*,\de_2^\star,\de_r^*,c_E,c_X\in (0,1)$ such that the following hold. 
Let $\s\in\{\pm 1\}^K$ and let $u$ be a solution of \eqref{NLKG} in $\cH$ satisfying initially 
$\U\sN_0(\vec u(0)) \le |\de_2^\star|^2$. 
Let $T^*\in(0,\I]$ be the maximal existence time of $u$.  
Let $\de_2\in(0,\de_2^\star]$ and assume that the soliton repulsivity condition \eqref{sol repul} holds for the solution $u$ as long as $\U\sN_0(\vec u)\le B_1^*|\de_2|^2$. 
 
Then there are $T_1,T_2,T_3,T_s\in[0,T^*]$ and $z\in C^1([0,T_3);(\R^N)^K)$ such that $v:=\vec u-\vec Q_\Si\in\Y^\xc_\perp(z)$ and $\sN_0(\vec u,z)\sim \U\sN_0(\vec u)$ for $0\le t< T_3$, and the following properties. Let $a^+_k:=p^+(z_k)v$ and $\ga:=P^\xcu_\perp(z)v$. 
\begin{enumerate}
\item $0\le T_1 \le T_2 \le T_3 \le T^*$ and $0\le T_s \le T_2$. $T_s=T_2$ for $K=1$ and $T_s<\I$ for $K\ge 2$. 
Either $T_1=T_2=T_3=T^*=\I$ or $T_3<T^*$. $T_3$ is uniquely defined by 
\EQ{
 T_3 = \inf\{t\ge 0 \mid |a^+(t)|\ge\de_2\},}
where $T_3:=\I$ if the set is empty. If $T_3<\I$ then 
\EQ{ \label{energy dec}
 \U\sN_0(\vec u(0))\le c_E|\de_2|^2 \implies E(\vec u(T_3)) \le KE(\vec Q) - \tf{\al\nu^+}{4}|\de_2|^2.}
\item If $\sN_0(\vec u(t),\ti z)\le c_X|\ap^+(\ti z)\vec u(t)| \le c_X^2$ for some $t\in[0,T^*)$ and $\ti z\in(\R^N)^K$ then $T_1\le t$ and $|a^+(t)-\ap^+(\ti z)\vec u(t)| \le \tf14|\ap^+(\ti z)\vec u(t)|$. \label{est T1} 
\item For $0\le t<T_1$, we have $\de_1^*|a^+| \le \sN(\vec u,z)\le 2\sN(\vec u(0),z(0))$. \label{first stage} 
\item For $T_1\le t<T_2$, we have $B_1^*|a^+|^2 \le \sN(\vec u,z)\le\de_1^*|a^+|$ and \label{second stage}
\EQ{
 |(\p_t-\nu^+)a^+|\le\tf{\nu^+}{2}|a^+| \le \p_t|a^+|, \pq \sN(\vec u,z) \le 2\sN(\vec u(0),z(0)).}
\item For $T_2\le t<T_3$, we have $\sN(\vec u,z)/B_1^* \le |a^+|^2 \le \de_2^2$ and \label{last stage}
\EQ{
 \pt a^+_*:=e^{\nu^+(t-T_2)}a^+(T_2) \implies |a^+-a^+_*|+|\dot a^+-\dot a^+_*|/\nu^+ \le \tf13|a^+_*|, 
 \pr |z-z(T_2)| \lec B_1^*|a^+|^2, 
 \pq \|\ga\|_{\cH\cap S(t)} \lec e^{-\mu(t-T_2)}\|\ga(T_2)\|_\cH+B_1^*|a^+|^2.
}
\item For $0<t<T_s$, we have $\max(|a^+|,\eN(D_z)) \le \de_r^*\sN(\vec u,z)^{1/2}$ and \label{rad stage}
\EQ{
 \|v\|_{\cH\cap S(t)}\sim\|\ga\|_{\cH\cap S(t)} \lec e^{-\frac{\mu}{2} t}\|\ga(0)\|_\cH, \pq |z-z(0)| \lec \de_r^*\|\ga(0)\|_\cH.
}
\item For $T_s<t<T_2$, we have $\|v\|_{\cH_z}^2<V_\s(z)$ and \label{sol stage}
\EQ{
 \pt |a^+| \lec e^{\nu^+(t-T_2)}|a^+(T_2)| + \eN(D_z),
  \pr \|v\|_{\cH\cap S(t)} \sim \|\ga\|_{\cH\cap S(t)} \lec e^{\mu(T_s-t)}\|\ga(T_s)\|_\cH + \eN(D_z).
}
\end{enumerate}
\end{theorem}
The order between $T_1$ and $T_s$ depends on the solution $u$. 
Also note that some of the time intervals may well be empty, depending on $u$. 
The energy decrease \eqref{energy dec} follows from $|a^+(T_3)|\ge \de_2$ and \eqref{energy eject} (including the case of $T_3=0$, which is easier).

\section{Decay of the even part around $2$-solitons} \label{sect:even decay}
In the soliton dominant dynamics (Section \ref{ss:soliton dynamics}), the slow decay in time, namely the last term in \eqref{bd v sol}, can not be improved in general for $K\ge 2$. 
In order to construct the manifolds of asymptotic $1$-solitons around the $2$-soliton manifold, 
 we need a more refined estimate, because $\eN(D_z)$ is not integrable in time. 
The point is that in the case of 2-solitons with the opposite sign, the soliton interaction is essentially odd in space, so that we have a better decay on the even symmetric part. 

In the setting of Section \ref{ss:soliton dynamics}, let $K=2$ with $\s_1\s_2=-1$. Then we have not only $V_\s(z) = \eN_0(|z_1-z_2|) \sim \eN(D_z)$ 
but also 
\EQ{
 |\na_zV_\s(z)| = \sqrt{2}|\eN_0'(|z_1-z_2|)| \sim \eN(D_z)}
for $D_z=|z_1-z_2|\gg 1$. Then the equation \eqref{eq V} on $V_\s(z)$ is simplified 
\EQ{
 \p_t V_\s(z) \sim -|\na_zV_\s(z)|^2 \sim - V_\s(z)^2,}
so that we can estimate the evolution of distance as 
\EQ{ \label{dyn V}
 \eN(D_z) \sim V_\s(z) \sim [\eN(D_z(0))^{-1}+t]^{-1}.}
In the setting of Theorem \ref{thm:dyn near Ksol}, we obtain 
\EQ{ \label{Dz 2sol}
 T_s<t<T_2 \implies \eN(D_z) \sim [\eN(D_z(0))^{-1}+t-T_s]^{-1} \lec [\U\sN_0(\vec u(0))^{-1}+t-T_s]^{-1}.}

On the other hand, the barycenter of solitons 
\EQ{
 \ba{z}:=\tf 1K\sum_{k=1}^K z_k}
is not affected by the potential force of $V_\s$. Indeed, we have 
\EQ{
 \pt C^1_\om \sum_{k=1}^K \s_k p^j(z_k)J\vec N^0(v) = \sum_{k=1}^K \LR{f(Q_\Si)-f_\Si|\p_jQ_k}
  \pr=\LR{f(Q_\Si)|\p_jQ_\Si}-\sum_{k,l}\s_k\s_l\LR{f(Q)|\p_jQ(x-z_k+z_l)}=0,}
where the second term vanishes by the odd symmetry for $z_l-z_k$, cf.~\eqref{f Q na}. 
Hence \eqref{eq z} and \eqref{bd v sol} imply
\EQ{
 |\p_t \ba{z}| \pt\lec \eN(D_z)^2 + \|v\|_\cH^2
 \pr\lec \eN(D_z)^2 + e^{2\nu^+(t-T)}|a^+(T)|^2 +  e^{-2\mu t}\|\ga(0)\|_\cH^2.} 
Integrating it from $t=T$, using the above estimate \eqref{dyn V}, we obtain 
\EQ{ \label{sym z}
 |\ba{z}(t)-\ba{z}(T)| \lec \eN(D_z) + |a^+(T)|^2 + e^{-2\mu t}\|\ga(0)\|_\cH^2 \lec \eN(D_z),}
since $|a^+|^2+\|\ga\|_\cH^2\lec\|v\|_\cH^2$ for $0<t<T$, and $\eN(D_z)$ is decaying slower than $e^{-2\mu t}$. 
In the setting of Theorem \ref{thm:dyn near Ksol} with $K=2$ and $\s_1\s_2=-1$, the above estimate becomes 
\EQ{ \label{sym z12}
 T_s<t<T_2 \implies |z_1(t)+z_2(t)-z_1(T_2)-z_2(T_2)| \lec \eN(D_z).}

For the unstable modes, the leading term is extracted from the equation \eqref{eq a+ est}
\EQ{
 a^{+,*}_k := (C^+_\om)^{-1}\sum_{l\not=k}\int_T^t e^{\nu^+(t-s)}\s_l\eN_+(|z_l-z_k|)ds,}
so that we can bound the remainder as
\EQ{ \label{a+ prof}
 \pt a^+_k - a^{+,*}_k = e^{\nu^+(t-T)}a^+_k(T) 
 + \int_T^t e^{\nu^+(t-s)}O(\eN(D_z)^{p_1}+\|v\|_\cH^2)ds,
 \pr |a^+-a^{+,*}| \lec e^{\nu^+(t-T)}|a^+(T)| + e^{-2\mu t}\|\ga(0)\|_\cH^2+\eN(D_z)^{p_1},}
which is small in $(L^1\cap L^\I)_t(0,T)$. 

For the above estimates, it is enough to assume the stronger repulsivitiy condition
\EQ{ \label{repul strong}
 V_\s(z) \sim |\na_z V_\s(z)| \sim \eN(D_z),}
for general $K\ge 2$ and $\s$, 
but the following argument is special for the repulsive 2-solitons $K=2$ and $\s_1\s_2=-1$. 
First observe by the symmetry 
\EQ{ \label{a+ odd}
 a^{+,*}_1 = - a^{+,*}_2.}
Next for the radiation part $\ga$, let $\sI$ be the spatial inversion operator with respect to the two soliton positions at $t=T$, namely
\EQ{
 \sI\fy(x) := \fy(-x+z_1(T)+z_2(T)).}
Then using $\s_1\s_2=-1$ and $Q(x)=Q(-x)$, we obtain 
\EQ{
 \sI Q_\Si \pt= -\sum_{k=1}^2 \s_k Q(x-z_k+z_1+z_2-z_1(T)-z_2(T))
 \pr= -Q_\Si + O(\eN(D_z)|Q|_\Si), }
where $|Q|_\Si:=\sum_k|Q_k|$, and the last step uses \eqref{sym z12} and $|\na Q|\lec Q$. Similarly we have 
\EQ{
 \pt \sI N^0 = f(\sI Q_\Si)-\sum_{k=1}^2 f(\sI Q_k) = -N^0 + O(\eN(D_z)|Q|_\Si),
 \pr \sI f'(Q_\Si) = f'(\sI Q_\Si) = f'(Q_\Si) + O(\eN(D_z)|Q|_\Si),
 \pr \|[\sI ,P_\perp^\xcu(z)]\|_{\op(\cH)} \lec |z_1+z_2-z_1(T)-z_2(T)| \lec \eN(D_z).}
Hence the equation \eqref{eq ga} of $\ga$ is inverted to 
\EQ{
 \pt (\p_t-J\sL_\Si^\al)\sI \ga = \sI (\p_t-J\sL_\Si^\al)\ga + [\sI,\smat{0 & 0 \\ f'(Q_\Si) & 0}]\ga
 \pr= P_\perp^\xcu(z)J\vec N^0 - \sI P_\perp^\xcu(z)J\vec N^1_\Si(v) + O(\eN(D_z)^2+\|v\|_\cH^2), }
where $O(\cdot)$ terms are bounded in $\cH$. 
In other words, the inversion flips the sign of the leading term. Hence the symmetric part satisfies the better equation 
\EQ{
 (\p_t-J\sL_\Si^\al)(\ga+\sI \ga)
 = -(I+\sI )P_\perp^\xcu(z)J\vec N^1_\Si(v) + O(\eN(D_z)^2+\|v\|_\cH^2)}
in $\cH$. Then by Lemma \ref{lem:linene damp} (the linear version, using the Strichartz bound on $\ga$), and using \eqref{bd v sol}, we obtain 
\EQ{
 \|\ga+\sI \ga\|_\cH \pt\lec e^{-\mu t}\|\ga(0)+\sI \ga(0)\|_\cH 
 \prq+ \int_0^t e^{-\mu(t-s)}(\eN(D_z)^2+e^{2\nu^+(s-T)}|a^+(T)|^2+e^{-2\mu s}\|\ga(0)\|_\cH^2)ds
 \pr \lec e^{-\mu t}\|\ga(0)\|_\cH + \eN(D_z)^2 + e^{2\nu^+(t-T)}|a^+(T)|^2,} 
which is small in $(L^1_t\cap L^\I_t)(0,T)$, since $\eN(D_z)\sim(\eN(D_{z(0)})^{-1}+t)^{-1}$ and $\eN(D_{z(0)})\lec\de_s^2$. 

For the whole remainder, the radial symmetry of $Y^+$, \eqref{a+ prof} and \eqref{a+ odd} imply 
\EQ{
 v+\sI v \pt= \ga+\sI\ga + (a^+_1+a^+_2)(Y^+_1+Y^+_2) + O(|a^+|\eN(D_z))
 \pr=\ga+\sI\ga+ O(e^{\nu^+(t-T)}|a^+(T)|+e^{-2\mu t}\|\ga(0)\|_\cH^2+\eN(D_z)^{p_1}),}
where the error terms are bounded in $\cH$. Hence 
\EQ{
 \|v+\sI v\|_\cH \lec e^{\nu^+(t-T)}|a^+(T)| + e^{-\mu t}\|\ga(0)\|_\cH + \eN(D_z)^{p_1},}
which is also small in $(L^1_t\cap L^\I_t)(0,T)$, since $p_1>1$. 

Under the assumption of Theorem \ref{thm:dyn near Ksol} (with $K=2$ and $\s_1\s_2=-1$), the above estimates become
\EQ{ \label{even v bd}
 \pt\|\ga+\sI\ga\|_\cH \lec e^{\mu(T_s-t)}\|\ga(T_s)\|_\cH + \eN(D_z)^2 + e^{2\nu^+(t-T_2)}|a^+(T_2)|^2,
 \pr \|v+\sI v\|_\cH \lec e^{\mu(T_s-t)}\|\ga(0)\|_\cH  + \eN(D_z)^{p_1} + e^{\nu^+(t-T_2)}|a^+(T_2)|,}
for $T_s<t<T_2$, which are uniformly small in $(L^1\cap L^\I)_t(T_s,T_2)$, while 
\EQ{
 0<t<T_s \implies \eN(D_z) \lec \|\ga\|_\cH \sim \|v\|_\cH \lec e^{-\frac{\mu}{2}t}\|\ga(0)\|_\cH}
is also small in $(L^1\cap L^\I)_t(0,T_s)$. 

\section{Convergence to multi-soliton} \label{ss:conv Ksol}
In the conclusion of Theorem \ref{thm:dyn near Ksol}-(1), we have either $T_1=T_2=T_3=T^*=\I$ or $T_3<T^*$. 
In the latter case, if $\U\sN_0(\vec u(0))\le c_E|\de_2^\star|^2$ then  \eqref{energy dec} implies that $u$ cannot be an asymptotic $K$-soliton. 
In the former case, we have $\sN_0(\vec u,z)\le\U\sN_0(\vec u(0))\ll 1$ for all $t\ge 0$. 
Moreover in Section \ref{sect:even decay}, we have seen that the strong repulsivity condition \eqref{repul strong} implies that $D_z\to \I$ as $t\to\I$. 
Now we prove that such a solution is indeed an asymptotic $K$-soliton.
\begin{lemma} \label{lem:conv Ksol}
For any $N\in\N$, $\al\in(0,\I)$, $p\in(1,p^\star(N))$, there is $\de_A\in(0,1)$ such that if $u$ is a global solution of \eqref{NLKG} in $\cH$ satisfying 
\EQ{
 \limsup_{t\to\I}\sN_0(\vec u(t),z(t))\le\de_A, \pq \lim_{t\to\I} D_{z(t)}=\I,}
for some $K\in\N$, $\s\in\{\pm 1\}^K$ and $z\in C^0([0,\I);(\R^N)^K)$, then $\sN_0(\vec u(t),z(t))\to 0$ as $t\to\I$. 
\end{lemma}
\begin{remark}
The assumption $D_z\to\I$ is not needed in the one-dimensional case $N=1$. 
For $N\ge 2$, however, the conclusion becomes false in general without that condition, because by \cite{MPW} there is a sequence of stationary solutions of \eqref{NLKG} close to $K$-solitons with $K\to\I$, $D_z\to\I$ and $\sN_0(\vec u,z)\to 0$. 
\end{remark}
\begin{proof}
For each $k=1\etc K$ and any sequence $t_n\to\I$, let 
\EQ{
 u_n(t,x) := u(t+t_n,x+z_k(t_n)).} 
Then $\{u_n\}$ is a sequence of global solutions of \eqref{NLKG} that is uniformly bounded in $\cH$. 
Using the energy-subcriticality, we may extract a subsequence, still denoted by $u_n$, convergent in $C([0,\I);\weak{\cH})$. Let $u_\I\in C([0,\I);\weak{\cH})$ be its limit.  Since $E(\vec u)$ is bounded below, \eqref{Edecay} implies that $\dot u\in L^2_{t,x}((0,\I)\times\R^N)$, hence 
$\dot u_n\to 0 \IN{L^2_{t,x}((0,\I)\times\R^N)}$. 
Hence $u_\I(t,x)=u_\I(x)\in H^1(\R^N)$ is a stationary solution of \eqref{NLKG}. 

Since $\|v\|_\cH\lec\de_A<1$ and $D_z\to\I$, the weak convergence implies 
\EQ{ \label{limit closeness}
 \|u_\I - \s_k Q\|_{H^1(\R^N)} \lec \de_A.}
Since $Q$ is an isolated stationary solution of \eqref{NLKG}, it implies that $u_\I=\s_k Q$ if $\de_A>0$ is small enough.  
In fact, if $u_\I$ is not a translation of the ground state, then its variational characterization implies that $E(u_\I) = \sum_{\pm} E(\max(\pm u_\I,0)) > 2E(Q)$, 
contradicting \eqref{limit closeness}. Hence $u_\I$ is also a ground state. Then \eqref{limit closeness} implies $u_\I=\s_k Q(x-c)$ for some $c\in\R^N$ with $|c|\lec \de_A$. 
The limit of the orthogonality implies 
\EQ{
 0=P^\xc(z_k)v(t_n) \to P^\xc(0)(\vec u_\I-\s_k \vec Q),}
hence $c=0$. Thus we obtain $\vec u_n \to \s_k \vec Q$ in $C([0,\I);\weak{\cH})$. 
Since the limit is unique, this convergence holds without the restriction to a subsequence. In other words, 
\EQ{ \label{weak conv at zk}
 \vec u(t,x+z_k(t)) \to \s_k\vec Q \IN{\weak{\cH}} \pq (t\to\I)}
for each $k=1\etc K$. 
Since this weak convergence eliminates all the eigenmodes, we deduce that 
\EQ{
 \ka:=\lim_{t\to\I} E(\vec u(t))-KE(\vec Q) = \lim_{t\to\I}\tf12\|\ga(t)\|_{\cH_{z(t)}}^2=\lim_{t\to\I}\tf12\|\ga(t)\|_\cH^2,}
where the last equality follows from the weak convergence around each $x=z_k$. 

Now suppose for contradiction that $\ka>0$ and let $\ga^k(t):=\ga_1(t,x+z_k(t))$. Then $\ga^k(t)\to 0$ in $\weak{H^1(\R^N)}$, and the Taylor expansion of the functional $K_0$ around $K_0(Q)=0$ yields, as $t\to\I$, 
\EQ{
 K_0(u(t)) \pt=K_0((Q_\Si + \ga_1)(t)) + o(1) 
 \pr= \sum_{k=1}^N \LR{K_0'(\s_k Q)|\ga^k(t)}+\|\ga_1(t)\|_{H^1}^2+O(\|\ga_1\|_{H^1}^{\min(p,3)})+o(1)
 \pr=\|\ga_1(t)\|_{H^1}^2+O(\|\ga_1\|_{H^1}^{\min(p,3)})+o(1),}
where $K_0'(Q)=-2\De Q+2Q-pQ^p$. Since $\|\dot u\|_{L^2_{t,x}(t>T)}\to 0$ as $T\to\I$, we deduce 
\EQ{
 [\cP(\vec u)]_{t}^{t+1} = \int_{t}^{t+1}[\|\dot u\|_{L^2}^2-K_0(u)]dt \sim -\ka,}
for $t\gg 1$, but it contradicts the boundedness of $\vec u$ in $\cH$. 
Hence $\ka=0$ and $\ga(t)\to 0$ strongly in $\cH$. 
Since all the eigenmodes are decaying because of \eqref{weak conv at zk}, we conclude that $v(t)\to 0$ in $\cH$. 
\end{proof}

\section{Difference estimate around multi-solitons} \label{sect:diff Ksol}
In order to investigate structure of solution sets around the $2$-solitons, we need to estimate the difference of two solutions both starting near the $2$-soliton but staying close only for finite time. 
In particular, we need growth estimate on the linearized equation around the solutions starting near $2$-solitons and converging to $1$-solitons. 
Then the main issue is whether the vector of unstable modes $a^+\in\R^2$ for the difference of solutions stays essentially in the same direction. Otherwise, the instability around one of the two solitons may be transferred to the other, making it hard to follow the whole instability mechanism. 

Since the soliton interactions are not integrable in time, namely $O(1/t)$, such a phenomenon of energy transfer is not immediately precluded. However, the cancellation \eqref{even v bd} in the even part implies that the $O(1/t)$ interactions do not cause such a transfer, even though they change the growth rate of the unstable modes. 

\subsection{Difference estimate near $K$-solitons}
Let $u$ be a solution of \eqref{NLKG} in $\cH$ close to the $K$-soliton $Q_\Si$ decomposed as in \eqref{def v} with the orthogonality and initial smallness
\EQ{
 v=\vec u-\vec Q_\Si \in \Y^\xc_\perp(z), \pq \|v(0)\|_\cH \le \de_0, \pq \eN(D_{z(0)})\le\de_D}
for some (unique) $z\in C^1([0,T];(\R^K)^N)$ on some time interval $[0,T]$, and $\de_0,\de_D>0$ satisfying $\de_0^2+\de_D\ll\de_E^\star$. 

Let $u^\ci$ be another solution starting near $u$
\EQ{
 \|\vec u^\ci(0)-\vec u(0)\|_\cH \ll \de_E^\star,}
and apply the same decomposition as $u$ to $u^\ci$, namely
\EQ{
 v^\ci:=\vec u^\ci-\vec Q_\Si, \pq a^{\x\ci}_k:=p^\xi(z_k)v^\ci, \pq \ga^\ci:=P^\xcu_\perp(z)v^\ci}
for $\x\in\xcu$, $k=1\etc K$ and $t\in[0,T]$. 
The difference is decomposed in the same way
\EQ{
 v':=v^\ci-v=u^\ci-u, \pq a^{\x '}_k:=a^{\x\ci}_k-a^\x_k, \pq \ga':=\ga^\ci-\ga.}
Note that $P^\xc(z_k)v^\ci\not=0$ in general. 
The equation for the difference is decomposed
\EQ{
 \pt(\p_t-J\sL_\Si^\al)v' = J(\vec N_\Si(v)-\vec N_\Si(v^\ci)),
 \pr(\p_t-\nu^\x)a^{\x'}_k = p^\x(z_k)\BR{J(\vec N_k(v)-\vec N_k(v^\ci))+\dot z_k\cdot\na v'},
 \pr(\p_t-J\sL_\Si^\al)\ga' = P^\xcu_\perp(z)\BR{J(\vec N_\Si(v)-\vec N_\Si(v^\ci))+R^\xcu_\perp(z)v'}.}
The nonlinear terms may be expanded as
\EQ{ \label{N exp diff}
 \pt N_\Si(v_1^\ci)-N_\Si(v_1) = [f'(Q_\Si+v_1)-f'(Q_\Si)]v_1' + O(W^{p-2}(v_1')^2),
 \pr N_k(v_1^\ci)-N_k(v_1) = [f'(Q_\Si+v_1)-f'(Q_k)]v_1' + O(W^{p-2}(v_1')^2),
 \pr f'(Q_\Si+v_1)-f'(Q_\Si) = f''(Q_\Si)v_1 + O(|Q|_\Si^{p-p_1-1}|v_1|^{p_1}),
 \pr f'(Q_\Si)-f'(Q_k) = O(|Q|_\Si^{p-2}\eN(D_z)),}
where $W:=|Q|_\Si+|v_1|+|v_1'|$. 
Then we may extract the linear part in the form 
\EQ{
 \pt (\p_t-\nu^\x)a^{\x'}_k = \sum_{l=1}^K \sum_{\y\in\xcu} \M_{k,l}^{\x,\y} a^{\y'}_l + \M_{k,0}^\x \ga' + \cQ_k^\x,
 \pr (\p_t-J\sL_\Si^\al)\ga' = \sum_{l=1}^K \sum_{\y\in\xcu} \M_{0,l}^\y a^{\y'}_l + \M_{0,0}\ga' + \cQ_0,}
where $\M_{k,l}^{\x,\y}$ are linear operators $\R\to\R$, $\cH\to\R$, $\R\to\cH$ or $\cH\to\cH$, depending on $t$ and $u$, whereas $\cQ_k^\x,\cQ_0$ gather the quadratic or higher terms in $v'$ or $(a',\ga')$. 

The above estimates \eqref{N exp diff} imply for $k\not=0$
\EQ{
 \|\M_{k,l}^{\x,\y}\|_{\op}+\|\M_{0,k}^\y\|_{\cH_*(t)} \lec \|v\|_\cH + \eN(D_z),}
where $\cH_*(t)$ is defined by 
\EQ{
 \cH_*(\ta):=L^1(\max(0,\ta-1),\ta;\op(\R\to\cH)),}
while $\M_{0,0}$ may be decomposed for some space-time function $h_0$ such that
\EQ{
 \pt \M_{0,0}\fy=\M_0^E\fy + \smat{0 \\ h_0\fy_1},
 \pq \|\M_0^E\|_{\op(\cH)}+\|h_0\|_{S^*(t)} \lec \|v\|_{\cH\cap S(t)}+\eN(D_z). }
In the simpler case $p\le\tf{N}{N-2}$, all the components are bounded in the operator norm as $\M_{k,l}^{\x,\y}$. 
For brevity, denote the time-dependent operator norm of $\M$ by  
\EQ{
 \dbr{\M}(t) := \sum_{k,l,\x,\y}\|\M_{k,l}^{\x,\y}\|_{\op} + \sum_{k,\x} \|\M_{0,k}^\x\|_{\cH_*(t)} + \|\M_0^E\|_{\op(\cH)} + \|h_0\|_{S^*(t)}.}
The nonlinear part is estimated similarly in the form 
\EQ{
 \pt |\cQ_k^\x| \lec \|v'\|_\cH^2, 
 \pq \cQ_0=\cQ_0^E+\smat{0 \\ g_0}, 
 \pq \|\cQ_0^E\|_{\cH} \lec \|v'\|_\cH^2,
 \pq |g_0|\lec |\ga'_1|^2+|\ga'_1|^p.}

If Theorem \ref{thm:dyn near Ksol} is applicable to $u$ for some $\de_2\in(\de_0,\de_2^\star]$, then the above estimate together with those by the theorem implies
\EQ{ \label{est M T23}
 \|\eN(D_z)\|_{(L^1\cap L^\I)_t(T_2,T_3)} + \|\dbr{\M}\|_{(L^1\cap L^\I)_t(T_2,T_3)}\lec \de_2,}
which becomes trivial if $T_2=T_3=T^*$. 
If in addition the stronger repulsivity condition \eqref{repul strong} holds, then 
\EQ{ \label{est M T02}
 \pt\|\eN(D_z)\|_{(L^2\cap L^\I)_t(0,T_2)}+ \|\dbr{\M}\|_{(L^2\cap L^\I)_t(0,T_2)} 
 \lec \de_D+\de_0.}
In the simple case $K=1$, we have no soliton interaction ($T_s=T_2$), so
\EQ{ \label{K=1 M bd}
 \|\dbr{\M}\|_{(L^1\cap L^\I)_t(0,T_2)} \lec \de_0.}

For the interaction of unstable modes, we have 
\EQ{
 \M^{+,+}_{k,l} \pt= \LR{f'(Q_\Si+v_1)-f'(Q_k)|\phi_k\phi_l}/C^+_\om +\LR{\dot z_k\cdot\na\phi_l|\phi_k}
 \prQ +O(\eN(D_z)(\eN(D_z)+\|v\|_\cH+|\dot z|)),}
where $\phi_k:=\phi(x-z_k)$, and the second line comes from the error term in Lemma \ref{lem:spec coord}. Hence
\EQ{ 
 \pt \M^{+,+}_{k,k} = \LR{f''(Q_k)(v_1+Q_\Si-Q_k)|\phi_k^2}/C^+_\om + O([\|v\|_\cH+\eN(D_z)]^{\min(p-1,2)}),}
and for $k\not=l$, using the better decay of $\phi$, 
\EQ{
 \M^{+,+}_{k,l} \pt= \LR{f'(Q_l)|\phi_k\phi_l}/C^+_\om + O((\|v\|_\cH+\eN(D_z)^{\min(p-2,1)})\eN(D_z)^{\LR{\nu_0}}) \pr=O(\eN(D_z)^{\LR{\nu_0}}). }
Hence, if Theorem \ref{thm:dyn near Ksol} is applicable as well as the strong repulsivity \eqref{repul strong}, then 
\EQ{ \label{M++ rem}
 \pt \|\M^{+,+}_{k,k}-\M^*_k\|_{(L^1\cap L^\I)_t(0,T_2)}
 \lec \de_0 + \de_D^{\min(1,p-2)},
 \pr \|\M^{+,+}_{k,l}\|_{(L^1\cap L^\I)_t(0,T_2)} \lec \de_D^{\LR{\nu_0}-1},}
where 
\EQ{ \label{def p3}
 \pt \M^*_k := \LR{f''(Q_k)(v_1+Q_\Si-Q_k)|\phi_k^2}/C^+_\om.}

\subsection{Difference estimate near $2$-solitons} \label{ss:diff 2sol}
Now let us specialize to the case of $K=2$ with $\s_1\s_2=-1$. Then, using the radial symmetry of $Q$ and $\phi$, as well as $f''(-u)=-f''(u)$, we obtain  
\EQ{
 C^+_\om(\M^*_1-\M^*_2) \pt= \LR{f''(Q)\phi^2|v_1(x+z_1)+v_1(x+z_2)}
 \pr=\LR{[f''(Q)\phi^2](x-z_1)|v_1+v_1(-x+z_1+z_2)},}
hence for $0<t<T_2$, using \eqref{sym z12} and \eqref{even v bd}
\EQ{
 |\M^*_1-\M^*_2|
 \pt\lec \|v+\sI v\|_\cH + |z_1+z_2-z_1(T_2)-z_2(T_2)|\|v\|_\cH 
 \pr\lec e^{-\frac{\mu}{2}t}\|\ga(0)\|_\cH  + \eN(D_z)^{p_1} + e^{\nu^+(t-T_2)}|a^+(T_2)|.}
Adding the other terms from \eqref{M++ rem}, we obtain 
\EQ{
 \|\M^{+,+}_{1,1}-\M^{+,+}_{2,2}\|_{(L^1\cap L^\I)_t(0,T_2)} \lec \de_0 + \de_D^{p_1-1}.}
Thus we obtain a uniform $L^1$ bound on 
\EQ{
 \pt \M^{(1)}(t):=\sum_{k\not=l}|\M^{+,+}_{k,k}-\M^{+,+}_{1,1}|+|\M^{+,+}_{k,l}|,}
namely (for $K=2$ with $\s_1\s_2=-1$)
\EQ{ \label{est M 2sol}
 \|\eN(D_z)\|_{(L^2\cap L^\I)(0,T_3)} + \|\M^{(1)}\|_{L^1(0,T_3)} + \|\dbr{\M}\|_{L^2_t(0,T_3)} \lec \de_D^{p_3}+\de_2,}
using $\de_0\le\de_2$, where the constant $p_3\in(0,1]$ is defined by 
\EQ{
 p_3:=\min(p_1-1,\LR{\nu_0}-1).}

The smallness of the $\ti M$ norm implies that the direction of $a^{+'}\in\R^2$ is essentially unchanged. 
To see that, let 
\EQ{ \label{def m}
 m(t):=\int_0^t [\nu^++\M^{+,+}_{1,1}(s)]ds, \pq \hat\de:=\de_D + \de_2 + \|\M^{(1)}\|_{L^1_t(0,T)} + \dbr{\M}_{L^\I_t(0,T)}.}
Then the equation becomes for $\x=+$ 
\EQ{
 e^m \p_t e^{-m}a^{+'}_k \pt= (\M^{+,+}_{k,k}-\M^{+,+}_{1,1})a^{+'}_k  
 \pn+\sum_{(l,\y)\not=(k,+)}\M_{k,l}^{+,\y}a^{\y'}_l + \M_{k,0}^+ \ga' + \cQ^+_k,}
where the nonlinear term is estimated for any $(k,\x)$ by 
\EQ{
 |\cQ^\x_k| \lec |a^{+'}_k|^2+(\de_D+\de_2)a^\neg_k+b,}
where $a^\neg_k$ and $b$ gather the other components as
\EQ{ \label{def B}
 a^\neg_k := \sum_{l\not=k}|a^{+'}_l|, \pq b:=\sum_k\sum_{\x\in\xc}|a^{\x'}_k|+\|\ga'\|_\cH.}
Thus we obtain 
\EQ{ \label{int est a+'}
 \pt |e^{-m}a^{+'}_k-a^{+'}_k(0)| 
 \pr\lec \int_0^t e^{-m}\BR{\M^{(1)}|a^{+}| + \dbr{\M} b+ |a^{+'}_k|^2+\hat\de |a^\neg_k|^2 + b^2 }(s)ds. }
On the center subspaces, just integrating the equation for $a^{\x'}_k$ yields 
\EQ{
 |a^{\xc'}| \lec |a^{\xc'}(0)|+\int_0^t \BR{\dbr{\M}\|v'\|_\cH+\|v'\|_\cH^2}(s)ds.}
For the radiation $\ga'$, Lemma \ref{lem:linene damp} yields 
\EQ{ 
 \|\ga'\|_{\cH} \pt\lec e^{-\mu t}\|\ga'(0)\|_\cH
  \pn+\int_0^t e^{-\mu(t-s)} \BR{\dbr{\M}\|v'\|_\cH+\|v'\|_\cH^2}(s)ds.}

Now we assume that for some $\ti\de\in(0,1)$ we have 
\EQ{
 \hat \de + \sum_k|a^{+'}|_k + b \le \ti\de}
on $t\in[0,T]$. Then the above two estimates may be combined into 
\EQ{ \label{est B}
 b \lec b(0) + \int_0^t [\ti\de b+\hat\de|a^{+'}|+|a^{+'}|^2](s)ds.}
For any $\ka>0$, let $b_\ka:=e^{-\ka m}b$. Then 
\EQ{ \label{Duh bka}
 b_\ka \lec e^{-\ka m}b(0) \pt+ \ti\de\int_0^t e^{\ka(m(s)-m(t))}b_\ka(s)ds
 \pn + e^{-\ka m}\int_0^t [\hat\de|a^{+'}|+|a^{+'}|^2](s)ds.}
We introduce further abbreviation 
\EQ{
 \pt \ck A_k(T):=\|e^{-m} a^{+'}_k\|_{L^\I_t(0,T)}, \pq A_k(T):=e^{m(T)}\ck A_k(T).}
The key point of the following argument is to show that $A_k\sim\|a^{+'}_k\|_{L^\I_t}$ under some initial assumption leading to exponential growth of $a^{+'}$ (but otherwise $A_k$ may be essentially bigger than $\|a^{+'}_k\|_{L^\I_t}$).

Since $\dot m\sim\nu^+$, using the Young inequality for convolution in the last inequality, we obtain for any $p\in[1,\I]$ 
\EQ{
 \|b_\ka\|_{L^p_t} \pt\lec \ka^{-1/p} b(0) + \ka^{-1}\ti\de\|b_\ka\|_{L^p_t}
 \pn+ \ka^{-1/p}\BR{\hat\de e^{(1-\ka)m}\ck A + e^{(2-\ka)m}\ck A^2}.}
If $\ka^{-1}\ti\de>0$ is small enough, then we obtain  
\EQ{ \label{est bka}
  \|b_\ka\|_{L^p_t} \pt\lec \ka^{-1/p}b(0)+ \ka^{-1/p}(\hat\de+A)e^{(1-\ka)m}\ck A.}
Using this in \eqref{int est a+'}, we obtain 
\EQ{
 \pt|e^{-m} a^{+'}_k-a^{+'}_k(0)| 
 \pr\lec \|\M^{(1)}\|_{L^1_t(0,T)}\ck A + A_k\ck A_k + \hat\de A \ck A + \|\dbr{\M}\|_{L^\I_t} \|b_1\|_{L^1_t} + \|b_{1/2}\|_{L^2_t}^2
 \pr\lec \hat\de\ck A + A_k\ck A_k + \hat\de A \ck A
 \pn + \hat\de [b(0) +\hat\de \ck A + e^m \ck A^2]
 \pn + [b(0)+ \hat\de e^{\frac 12m}\ck A + e^{\frac32 m}\ck A^2]^2
 \pr\lec \hat\de\ck A + A_k\ck A_k + \hat\de A \ck A + A^3 \ck A + \hat\de b(0) + b(0)^2.}
Hence, as long as $\hat\de+A^3$ remains small, we have 
\EQ{ \label{est ckA}
 \ck A \lec |a^{+'}(0)| + \hat\de b(0) + b(0)^2,}
and plugging this into the above yields 
\EQ{
  |e^{-m} a^{+'}_k-a^{+'}_k(0)|  \lec (\hat\de+A^3)|a^{+'}(0)| + A_k\ck A_k  + \hat\de b(0) + b(0)^2,}
and similarly for $\ck A_k$, 
\EQ{
 \pt \ck A_k \lec |a^{+'}_k(0)| + (\hat\de+A^3)|a^{+'}(0)| + \hat\de b(0) + b(0)^2, 
 \pr |e^{-m} a^{+'}_k-a^{+'}_k(0)|  \lec (\hat\de+A^3)|a^{+'}(0)| + A_k|a^{+'}_k(0)| + \hat\de b(0) + b(0)^2.}

Hence if there is some $\de\in(\hat\de,1)$ such that for a sufficiently small constant $\de^*>0$ 
\EQ{ \label{init dom ak-0}
 \de|a^{+'}(0)| + \hat\de b(0) + b(0)^2 \le \de^* |a^{+'}_k(0)|, }
then as long as $\hat\de+A^3\le\de\ll 1$, we have 
\EQ{ \label{ak' exp}
 \pt |e^{-m} a^{+'}_k - a^{+'}_k(0)| \lec \de^*|a^{+'}_k(0)|, \pq A_k \sim e^m|a^{+'}_k(0)| \sim |a^{+'}_k|,
 \pr |e^{-m} a^+-a^{+'}(0)| \lec  \de^*|a^{+'}(0)|, \pq A\sim e^m|a^{+'}(0)| \sim |a^{+'}|.}
Plugging this into \eqref{est bka} yields for $\ka\gec\ti\de$,
\EQ{ \label{bd on b}
 b = e^{\ka m}b_\ka \lec e^{\ka m}b(0) + (\hat\de+|a^{+'}|)|a^{+'}| }
for $\ka\gec\ti\de$. Using a fixed $\ka$ (e.g., $\ka=1/2$), we obtain 
\EQ{
 \de|a^{+'}| + \hat\de b + b^2 \lec \de^*|a^{+'}_k|.}
In other words, \eqref{init dom ak-0} is essentially preserved as long as \eqref{ak' exp} holds. 
The center-stable part is also bounded by using Lemma \ref{lem:spec coord}, 
\EQ{
 \|P^\xu_\perp(z)u'\|_\cH \lec \de_D|a^{+'}| + b \lec e^{\ka m}b(0) + (\hat\de+|a^{+'}|)|a^{+'}|,}
which is the same as the bound on $b$ in \eqref{bd on b}. 
The above condition \eqref{init dom ak-0} may be simplified to 
\EQ{ \label{init dom ak-1}
 \de\|\vec u'(0)\|_\cH + \|\vec u'(0)\|_\cH^2 \lec \de^*|a^{+'}_k(0)|.}
Moreover, this is not essentially affected by the modulation of centers in Lemma \ref{lem:center}, even if we do not have the orthogonality at $t=0$. In fact, if $\ti z\in(\R^N)^K$ is given by the lemma, then \eqref{est center change} yields 
\EQ{
 |a^{+'}_k(0)-p^+(\ti z_k)\vec u'(0)| \lec |z-\ti z|\|\vec u'(0)\|_\cH \lec \de\|\vec u'(0)\|_\cH.}
This difference on the right side of \eqref{init dom ak-1} is absorbed by the left if $\de^*$ is small enough.  

Combining this conclusion with the estimate on $\M$ in \eqref{est M T23},\eqref{est M T02} and \eqref{est M 2sol}, we obtain
\begin{lemma} \label{lem:diff 2sol}
For any $N\in\N$, $\al\in(0,\I)$ and $p\in(2,p^\star(N))$, 
there exist $\de_m^\star\in(0,1)$ and $C_m\in[1,\I)$ such that the following hold. 
Let $u,u^\ci$ be two solutions in $\cH$ of \eqref{NLKG} on a time interval $[0,T)$ with the difference $u':=u^\ci-u$, satisfying 
\EQ{
 \pt \sN_0(\vec u(0),\U{z}) \le \de, \pq \eN(D_{\U z})^{p_3} \le \de, 
 \pr \de\|\vec u'(0)\|_\cH + \|\vec u'(0)\|_\cH^2 \le 2\de_m^\star|p^+(\U z_k)\vec u'(0)|, \pq \|\vec u'(0)\|_\cH^3 <\de, }
for some $\s\in\{\pm 1\}^2$, $\U{z}\in(\R^N)^2$, $k\in\{1,2\}$ and $\de\in(0,\de_m^\star]$. 
Then for $0\le t<T$ and as long as 
\EQ{ \label{a priori cond u'}
 \sup_{0\le s\le t}\|\vec u'(s)\|_\cH^3 \le \de,}
we have
\EQ{ \label{diff grow}
 \pt |e^{-m}a^{+'}_k-a^{+'}_k(0)| \le \tf14|a^{+'}_k(0)|, \pq |e^{-m}a^{+'}-a^{+'}(0)| \le \tf14|a^{+'}(0)|, 
 \pr \de\|\vec u'\|_\cH+\|\vec u'\|_\cH^2 \le C_m\de_m^\star|a^{+'}_k|,
 \pr \|P^\xu_\perp(z)\vec u'\|_\cH \le C_me^{C_m \de^{1/3} m}\|\vec u'(0)\|_\cH + C_m(\de+|a^{+'}|)|a^{+'}|,}
where $z\in C^1([0,T);(\R^N)^2)$ is chosen to satisfy the orthogonality $\vec u-\vec Q_\Si\in\Y^\xc_\perp(z)$, $a^{+'}:=p^+(z)\vec u'$, and  
$m$ is defined by \eqref{def m}, satisfying $m(0)=0$ and $\dot m\sim\nu^+$. 
In the case of $\|\vec u'(0)\|_\cH\le\de$, the condition \eqref{a priori cond u'} may be replaced with 
\EQ{ \label{a priori cond a'}
 \sup_{0\le s\le t}|a^{+'}(s)|^3 \le \de.}
\end{lemma}
In short, if the initial difference is essentially in the unstable direction, then it grows mainly in the unstable components, and its direction is not essentially changed in $a^{+'}\in\R^2$, as long as the difference remains small enough. 

The reduction of the condition \eqref{a priori cond u'} to \eqref{a priori cond a'} follows from the estimates in \eqref{diff grow}: if $\|\vec u'(0)\|_\cH\le\de\ll\de^{1/3}\sim\|\vec u'(t)\|_\cH$ then the upper bound on $P^\xu_\perp(z)\vec u'$ and the lower bound on $|a^{+'}|$ implies that $\|P^\xu_\perp(z)\vec u'(t)\|_\cH\ll|a^{+'}|\sim\de^{1/3}$.

\subsection{Difference estimate near 1-solitons} \label{ss:diff 1sol}
The case of $K=1$ is essentially known. 
Note that the above argument applies also to $K=1$, but it is much simpler in several places. In particular, the norm $\ti M$ does not contain the extra components of $\M^{+,+}$. Second, there is no soliton interaction or $T_s=T_2$, so each component of $\M$ is uniformly integrable in time. 
In particular, we may replace $e^{m}$ with $e^{\nu^+t}$. For later use, we state it in a simpler form. 
\begin{lemma} \label{lem:diff 1sol}
For any $N\in\N$, $\al\in(0,\I)$ and $p\in(2,p^\star(N))$, there exist $\de_m^\star\in(0,1)$ and $C_m\in[1,\I)$ such that the following hold. Let $u,u^\ci$ be two solutions in $\cH$ of \eqref{NLKG} on a time interval $[0,T)$ with the difference $u':=u^\ci-u$, satisfying 
\EQ{
 \pt \sN_0(\vec u(0),\U{z}) \le \de, \pq \de\|\vec u'(0)\|_\cH \le \de_m^\star|p^+(\U z)\vec u'(0)|, \pq \|\vec u'(0)\|_\cH \le\de, }
for some $\s\in\{\pm 1\}$, $\U{z}\in \R^N$ and $\de\in(0,\de_m^\star]$. 
Then for $0\le t<T$ and as long as \eqref{a priori cond a'} holds,
\EQ{
 \pt |e^{-\nu^+t}a^{+'}-a^{+'}(0)| \le \tf14|a^{+'}(0)|, 
 \pq \de\|\vec u'\|_\cH \le C_m\de_m^\star|a^{+'}|, 
 \pr \|P^\xu_\perp(z)\vec u'\|_\cH \le C_me^{C_m \de^{1/3} \nu^+ t}\|\vec u'(0)\|_\cH + C_m(\de+|a^{+'}|)|a^{+'}|,}
where $z\in C^1([0,T);\R^N)$ is chosen to satisfy the orthogonality $\vec u-\s\vec Q(x-z)\in\Y^\xc_\perp(z)$, and $a^{+'}:=p^+(z)\vec u'$. 
\end{lemma}

\section{Collapse of multi-solitons} \label{sect:collapse}
In this section, we investigate the behavior of solutions getting away from the small neighborhood of $K$-solitons with sufficient distance. Using the finite speed of propagation, together with uniform bounds on time needed for the decay and for the blow-up criterion, those solutions are investigated by superposition of the case of $1$-solitons. 
In particular, if one of the $K$-solitons collapses by the instability into the blow-up region, then the original solution also blows up. In other words, those solutions can survive the collapse only if the instability or each $a^+_k$ moves into the decay region or stays small. 

\subsection{Collapse of $1$-solitons} \label{ss:col-1}
We start with the $1$-soliton case, as it is the building block to describe the collapse of $K$-solitons. 

Let $\s\in\{\pm 1\}$ and let $u$ be a solution of \eqref{NLKG} on the maximal interval $[0,T^*)$ with 
\EQ{
 \U\sN_0(\vec u(0))^{1/2} \le \de \ll \min(\de_0^\star,\de_2^\star)}
so that we can apply Theorem \ref{thm:dyn near Ksol} with $K=1$ and $\de_2=\de_2^\star$, as well as \eqref{energy dec}. 

If $T_3=\I$ then Lemma \ref{lem:conv Ksol} implies that $u$ is an asymptotic $1$-soliton. 
Otherwise, $T_3<\I$ and $E(\vec u(T_3))<E(\vec Q)$ by \eqref{energy dec}, 
so the solution $u$ is below the ground state, with
\EQ{
 K_0(u) \pt= K_0(Q) + \s a^+\LR{K_0'(Q(x-z))|v} + o(a^+) 
  \pr= -\s (p-2) a^+\LR{Q^p|\phi} + o(a^+) \sim -\s a^+.}
If it is positive then $u$ is global and decaying, and if it is negative then $u$ blows up. 

For $0\le t<T_1$, we have $|a^+|\lec\de^2$. For $T_1\le t<T_3$, we have 
\EQ{ \label{inst grw 1sol} 
 \p_t a^+ \sim a^+, \pq e^{\frac12\nu^+(t-T_1)} \le a^+/a^+(T_1) \le e^{\frac32\nu^+(t-T_1)},}
and $\sign a^+$ is preserved. For $0\le t<T_2=T_s$, we have 
\EQ{
 \|v\|_\cH \lec e^{-\mu t/2}\|\ga(0)\|_\cH, \pq |z-z(0)| \lec \|\ga(0)\|_\cH,}
and for $T_2\le t<T_3$, 
\EQ{
 \|\ga\| \lec e^{-\mu(t-T_2)}\|\ga(T_2)\|_\cH + |a^+|^2, \pq |z-z(T_2)| \lec |a^+|^2.}
If $T_1>0$, then there is $t_0\in[T_1,T_3)$ such that
\EQ{
 |a^+(t_0)| \sim \de^2, \pq |t_0-T_3|\lec|\log(\de_2^\star/\de_0)| \lec |\log\de|.} 
Otherwise, we have $|a^+(0)|^2\gec\de^2$ and $T_3\lec|\log\de|$. 
On the other hand, Lemmas \ref{lemma2.2} and \ref{lem:bup destiny}, together with the uniform bound on $\vec u$ on $[0,T_3]$, imply
\EQ{
 \CAS{\s a^+(T_1)<0 \implies \|\vec u(t)\|_\cH \le Ce^{-c(t-T_3)_+} &(0\le t<T^*=\I),\\ 
  \s a^+(T_1)>0 \implies P(\vec u(t)) \ge e^{c(t-T_3)_+}-C &(0\le t<T^*<\I),} }
for some constants $c\in(0,1)$ and $C\in(1,\I)$ depending only on $N,\al,p$. 
Hence in both cases, we may find some $t_1\in(T_3,T^*)$ such that 
\EQ{
 |T_3-t_1|\lec |\log\de|, \pq \CAS{\s a^+(T_1)<0 \implies \|\vec u(t)\|_\cH \le \de^5 &(t_1\le t<\I),\\ 
  \s a^+(T_1)>0 \implies P(\vec u(t)) \ge \de^{-5} &(t_1\le t<T^*),} }
where the number $5$ may be replaced with any number bigger than $4$. 
Thus we obtain from Theorem \ref{thm:dyn near Ksol}: 
\begin{lemma} \label{lem:col 1sol}
For any $N\in\N$, $\al\in(0,\I)$ and $p\in(2,p^\star(N))$, 
there exist $\de_x^\star\in(0,1)$ and $C_x\in(1,\I)$ such that for any solution $u$ of \eqref{NLKG} satisfying for some $\s\in\{\pm 1\}$ and $\de\in(0,\de_x^\star]$
\EQ{
  \U\sN_0(\vec u(0)) \le \de^2,}
there exist $0\le t_0\le t_1\le\I$ and $z\in C^1([0,t_0);\R^N)$ with the following properties. 
Let $T^*\in(0,\I]$ be the maximal existence time of $u$. Then $t_1\le T^*$. 
For $0\le t<t_0$ we have 
\EQ{ \label{before 1sol inst}
 \pt v:=\vec u-\s\vec Q(x-z) \in \Y^\xc_\perp(z), \pq |z-z(0)| \le C_x\de,
 \pr |p^+(z)v| \le C_x\de^2, \pq \|P^\xcu_\perp(z)v\|_\cH \le C_x(\de e^{-c_*t}+\de^4e^{-c_*(t_0-t)}).} 
$u$ is an asymptotic $1$ soliton if and only if $t_0=\I$. 
Otherwise, $t_0<t_1<\I$ and 
\EQ{
 |t_0-t_1| \le C_x|\log\de|.}
For $t_1\le t<T^*$, we have either one of
\begin{enumerate}
\item $\|\vec u(t)\|_\cH \le \de^5$ \label{decay case}
\item $\cP(\vec u(t)) \ge \de^{-5}$. \label{bup case}
\end{enumerate}
In the former case, $u$ is global and decaying, while in the latter case, $u$ blows up. 
The three cases may be distinguished by the following. 
For any $t\in[0,T^*)$ and $\ti z\in\R^N$ satisfying ($c_X$ is from Theorem \ref{thm:dyn near Ksol})
\EQ{
 \sN_0(\vec u(t),\ti z)\le c_X|\ap^+(\ti z)\vec u(t)| \le c_X^2,} 
the sign $\ta:=\s\sign\ap^+(\ti z)\vec u(t)\in\{\pm 1\}$ is independent of $(t,\ti z)$. 
$\ta=-1$ in the decaying case \eqref{decay case}, and $\ta=+1$ in the blow-up case \eqref{bup case}. 
If $t_0=\I$ then $|\ap^+(\ti z)\vec u(t)|\le C_x\de^2$ for all $t\ge 0$. 
Otherwise, there is some $t\in(t_0,t_1)$ and $\ti z\in\R^N$ such that $|\ap^+(\ti z)\vec u(t)|>\de_2^\star>C_x\de_x^\star$. 
\end{lemma}
In particular, the criteria for the decaying case \eqref{decay case} and for the blow-up case \eqref{bup case} imply that those two cases are stable with respect to initial perturbation. The stability of decaying solutions follows generally from the definition, but the stability of blow-up in general is unclear for large $N,p$. 

\subsection{Collapse of $K$-solitons}
Next we consider the general case with $K$-solitons. 
Let $u$ be any solution of \eqref{NLKG} in $\cH$ on the maximal existence interval $[0,T^*)$. 
For any $z\in(\R^N)^K$ with $D_z\gg 1$, we 
consider localized solutions $u_k$ around $z_k\in\R^N$, defined as follows.  
For $k=1\etc K$ and $l=0\etc K$, define cut-off functions by
\EQ{
 \pt \chi_k(x):=\chi(|x-z_k|-\tf23 D_z),
 \pr \chi_0(x):=1-\sum_{k=1}^K \chi(|x-z_k|-\tf 13D_z), \pq \chi_l^C:=1-\chi_l,}
where $\chi\in C^\I(\R)$ has been fixed in \eqref{def chi}. 
Since $D_z\gg 1$, we have $\chi_k\in C_0^\I(\R^N)$ and $\chi_0\in C^\I(\R^N)$. 
For each $k=0\etc K$, let $u_k$ be the solution of \eqref{NLKG} with the initial data 
\EQ{ \label{def uk}
 \vec u_k(0) = \chi_k \vec u(0).}
Let $T^*_k\in(0,\I]$ be the maximal existence time of $u_k$, and
\EQ{
 \U T^* := \min(T^*,\min_{k=0\etc K}T_k^*)\in(0,\I].} 
For $k=1\etc K$, let 
\EQ{
 \pt \Om_k(t):=\{x\in\R^N\mid |x-z_k(0)|< \tf23D_z-t\},
 \pr \Om_k'(t):=\{x\in\R^N \mid |x-z_k(0)|< \tf13D_z+1+t\},
 \pq \Om_0(t):=\Ca_{k=1}^K (\R^N\setminus \ba{\Om_k'(t)}).}
Then we have for $k=0\etc K$
\EQ{
 \pt x\in\Om_k(0) \implies \vec u_k(0,x)=\vec u(0,x),}
so the finite propagation property of \eqref{NLKG} implies for $t\in[0,\U T^*)$ 
\EQ{ \label{uk match}
 \pt x\in\Om_k(t) \implies \vec u_k(t,x)=\vec u(t,x).}
Since 
\EQ{
 \pt 0\le t<\tf 16D_z-\tf12=:T_z^\star 
 \pr\implies \Om_k'(0) \subset \Om_k'(t) \subset \Om_k(t) \subset \Om_k(0)\pq(k=1\etc K)
  \pr\implies \Cu_{k=0}^K \Om_k(t)=\R^N,} 
we may recover the original solution $u$ completely from $\{u_k\}_{k=0\etc K}$ for $0\le t<T_z^\star$. 

Now we consider the special case where $u$ is initially 
close to a $K$-soliton $Q_\Si$ for some $\s\in\{\pm 1\}^K$ as before, namely
\EQ{
 \pt v:=\vec u-Q_\Si, \pq  \eN(D_z)=:\de_D \ll 1, \pq \|v(0)\|_\cH=:\de \ll 1.}
Then for the background part $u_0$, the small data theory for \eqref{NLKG} with
\EQ{
  \|\vec u_0(0)\|_{\cH} \le \|\vec u(0)\|_{\cH(\Om_0(0))} \lec \eN(\tf13D_z) + \|v(0)\|_\cH \lec \de_D^{1/3} + \de \ll 1}
implies that $u_0$ is a small global ($T_0^*=\I$) and decaying solution with 
\EQ{ \label{est u0}
 \|\vec u(t)\|_{\cH(\Om_0(t))} \le \|\vec u_0(t)\|_\cH \lec e^{-\mu t}\|\vec u_0(0)\|_{\cH} \lec e^{-\mu t}(\de_D^{1/3}+ \de) \ll 1,}
for some constant $\mu\in(0,1)$ depending only on $N,\al$. 
Similarly we obtain 
\EQ{ \label{ext dec uk}
 \|\vec u_k(t)\|_{\cH(\R^N\setminus\ba{\Om_k'(t)})} \lec e^{-\mu t}\|\vec u_k(0)\|_{\cH(\Om_k(-1)\setminus\ba{\Om_k'(0)})} \lec e^{-\mu t}(\de_D^{1/3}+ \de).}
Hence, if any of $u_k$ blows up before $T_z^\star$, then for the earliest $T_k^*=\U T^*$, we have 
\EQ{
 \I=\lim_{t\to T_k^*-0}\|\vec u_k(t)\|_\cH \pt= \lim_{t\to T_k^*-0}\|\vec u_k(t)\|_{\cH(\Om_k'(t))}
 \pn=\lim_{t\to T_k^*-0}\|\vec u(t)\|_{\cH(\Om_k'(t))}, }
so $u$ also blows up at $T_k^*=\U T^*=T^*<T_z^\star$. 
In other words, if $T_z^\star<T^*$ then $T_z^\star<\U T^*$. 
Note that the decomposition into $u_k$ works for any solution $u$ in $\cH$, as well as the comparisons with the original solution including the blow-up time.

The exponential decay of the ground state and eigenfunctions implies for $0\le l\le K$ with $l\not=k$ and $\x\in\xscu$,  
\EQ{
 \pt \|\chi_l \vec Q_k\|_\cH + \|\chi_l \vec Y^\x_k\|_\cH \lec \eN(\tf13D_z) \lec \de_D^{1/3} 
 \pr \|\chi_k^C \vec Q_k\|_\cH + \|\chi_k^C \vec Y^\x_k\|_\cH \lec \eN(\tf23D_z) \lec \de_D^{2/3}.}
Hence for $k=1\etc K$, we have 
\EQ{
 \|\vec u_k(0)-\vec Q_k\|_{\cH} \pt\lec \|\chi_k^C\vec Q_k\|_\cH + \sum_{l\not=k}\BR{\|\chi_k\vec Q_l\|_\cH + \|\chi_k Y_l^+\|_\cH}+\|v(0)\|_\cH  \pr\lec \eN(\tf13D_z)+\|v(0)\|_\cH \lec \de_D^{1/3}+\de,}
and similarly for $\x\in\xscu$, 
\EQ{
 \pt|p^\x(z_k)(\vec u_k(0)-\vec Q_k)-p^\x(z_k)v| 
 \pr\lec \|\chi_k^C\vec Q_k\|_\cH + \sum_{l\not=k}\BR{|p^\x(z_k)\vec Q_l| + |p^\x(z_k)Y_l^+|}+ |p^\x(z_k)\chi_k^Cv| 
  \pn\lec \de_D^{2/3}.}
In particular, Lemma \ref{lem:center} yields $\z(0)\in(\R^N)^K$ such that 
\EQ{ \label{uk init center}
 |\z(0)-z| \lec \de_D^{2/3}, \pq \vec u_k(0)-\s_k\vec Q(x-\z_k(0)) \in \Y^\xc_\perp(\z_k).} 
Assuming
\EQ{
 \de_D^{1/3} \le \de \ll \de_x^\star,}
we may apply Lemma \ref{lem:col 1sol} to each $u_k$ around $Q_k$ for $k=1\etc K$.
Then there is an open set $I\subset\R$ (a finite union of intervals or empty), such that $|I|\lec K|\log\de|$ and for each $k=1\etc K$ and any $T\in[0,\I)\setminus I$, we have one of the following three exclusively: 
\begin{itemize}
\item[($0$)] For all $0\le t\le T$, $|b_k^+|\lec\de^2$ (soliton case),
\item[($-$)] For all $T\le t<T_k^*=\I$, $\|\vec u_k(t)\|_\cH\le\de^5$ (decaying case),
\item[($+$)] For all $T\le t<T_k^*<\I$, $\cP(\vec u_k(t))\ge \de^{-5}$ (blow-up case),
\end{itemize}
where in the soliton case, the lemma yields some $\z_k\in C^1([0,T];\R^N)$ such that 
\EQ{ \label{est vk}
 \pt v_k:=\vec u_k - \s_k\vec Q(x-\z_k) \in \Y^\xc_\perp(\z_k), \pq b^+_k:=p^+(\z_k)v_k, \pq \ga_k:=P^\xcu_\perp(\z_k)v_k, 
 \pr |b^+_k| \lec \de^2, \pq \|\ga_k\|_\cH \lec \de e^{-c_*t} + \de^4 e^{-c_*(T-t)}, \pq |\z_k-z_k| \lec \de.}
Moreover, we may include $[0,C|\log\de|]$ in $I$ for some constant $C>1$ such that for any $T\in[0,T_k^*)\setminus I$ we have 
\EQ{
 \|\ga_k(T)\|_\cH \lec \de^4.}
Correspondingly, the indices $1\etc K$ are decomposed into $\La_0(T),\La_-(T),\La_+(T)$ with
\EQ{ \label{def La}
 \La_\dia(T):=\{k\in\{1\etc K\} \mid \text{$(\dia)$ holds for $(k,T)$}\}. }
The set $\La_0(T)$ is non-increasing in $T$, while the other two are non-decreasing. 
Note that in the soliton case $\La_0(T)$, we have $T\le t_0$, and in the other cases $\La_\pm(T)$, we have $T\ge t_1$, where $t_0,t_1$ are given by the lemma depending on $k$.

Now the idea is to investigate the original solution $u$ avoiding the transition time $I$. 
We have no control on the distribution or position of $I$, but the total length is bounded. 
So we impose another smallness condition: 
\EQ{ \label{d0 small 2}
 |\log\de_D| \gg |\log \de|,}
or in other words, $\de_D \le \de^M$ for some big constant $M>1$. Then we have 
\EQ{ \label{sep for propa}
 T_z^\star = \tf16D_z-\tf12 \sim |\log\de_D| \gg K|\log \de| \gec |I|.}
 
In particular we have plenty of $T\in(0,T_z^\star)\setminus I$. 
Since $T<T_z^\star$, if there is any $T_k^*\le T$ then $T^*\le T$. 
Suppose that $\U T^*>T$. 
If there is any blow-up case $\La_+(T)\not=\empt$, then denoting the localized $\cP$ by 
\EQ{
 \cP_\Om(\fy) := \int_\Om [\fy_1(x)\fy_2(x)+\al|\fy_1(x)|^2]dx,}
we have at $t=T$, using \eqref{uk match}, 
\EQ{
 \cP(\vec u) \pt= \sum_{k=1}^K P_{\Om_k'}(\vec u_k) + P_{\Om_0}(\vec u_0)
 \pn\ge \sum_{k=1}^K \cP(\vec u_k) - CK\|\vec u_0\|_\cH^2.}
Since $\|\vec u_0\|_\cH\ll 1$ and all $\cP(\vec u_k)$ are bounded from below $\cP(\vec u_k)\gec-1$ uniformly, while at least one of them satisfies $\cP(\vec u_k(T))\ge \de^{-5}\gg 1$, we deduce that
\EQ{
 \cP(\vec u(T)) \ge \de^{-5}/2 \gg E(\vec u(0)) \ge E(\vec u(T)),}
so $u$ satisfies the blow-up criterion \eqref{bc-P} at $t=T$ and blows up.

Thus we conclude that if $u$ is global, then $\La_+(T)=\empt$ for all $T\in(0,T_z^\star)\setminus I$. 
Moreover, if $T\gec|\log\de|$ then using \eqref{uk match} and \eqref{est u0} together with the above estimates, we obtain 
\EQ{
 \pt \vec u(T) = \sum_{k\in \La_0(T)}[\s_k\vec Q+b^+_k(T) Y^+](x-\z_k(T)) + O(\de^4),
 \pr |b^+_k(T)| \lec \de^2, \pq \sum_{k\in\La_0(T)}|z_k-\z_k(T)| \lec \de,}
where the error term is bounded in $\cH$. 
In particular, if $\La_0(T)=\empt$ or all $k$ are in the decaying case, then $u$ is also small at $t=T$, so it is global and decaying. 
We may apply Lemma \ref{lem:center} to modulate $(\z_k(T))_{k\in\La_0(T)}$ to a unique $(z_k(T))_k$ with the orthogonality condition with the difference 
\EQ{ \label{uk last center}
 |\z(T)-z(T)| \lec \eN(D_{\z_k(T)}) + \de^4 \sim \de^4,}
which does not affect the above estimates. 
Thus we have obtained 
\begin{theorem} \label{thm:col Ksol}
For any $N\in\N$, $\al\in(0,\I)$, $p\in(2,p^\star(N))$ and $K\in\N$, 
there exist $\de_X^\star\in(0,1)$ and $C_X\in(1,\I)$ such that the following hold. Let $z\in(\R^N)^K$, $\s\in\{\pm 1\}^K$ and let $u$ be a solution of \eqref{NLKG} in $\cH$, satisfying 
\EQ{ \label{col away cond}
 \sN_0(\vec u(0),z) \le \de, \pq \eN(D_z) \le \de^{C_X}}
for some $\de\in(0,\de_X^\star]$. Let $u_k$ be the solution of \eqref{NLKG} defined by \eqref{def uk} for each $k=1\etc K$ with the maximal existence time $T_k^*\in(0,\I]$. Then there is an open set $I\subset\R$ with $|I|\le C_X|\log\de|$ such that for each $T\in[0,\I)\setminus I$ we have the decomposition \eqref{def La} of $\{1\etc K\}$ into the three cases. 
$\La_-(T)$ and $\La_+(T)$ are non-decreasing for $T$ and for small perturbation of the initial data $\vec u(0)$. If $a^+_k:=\ap^+_k(\ti z)\vec u(0)$ for some $\ti z\in(\R^N)^K$ satisfies $C_X^2\sN_0(\vec u(0),\ti z)\le C_X|a^+_k| \le 1$, then $k\in\La_{\s_k\sign a^+_k}(0)$. 
If $\La_+(T)\not=\empt$ for some $T<T_z^\star:=\tf16D_z-\tf12$, then $u$ blows up. 
For each $k\in\La_0(T)$, we have the estimates \eqref{est vk} around the $1$-soliton. 
If $T\in(0,T_z^\star)\setminus I$ and $\La_+(T)=\empt$, then there is $z(T)=(z_k(T))_{k\in\La_0(T)}\in(\R^N)^{\#\La_0(T)}$ such that 
\EQ{ \label{decop after col}
  \pt v:=\vec u(T)-\sum_{k\in\La_0(T)}\s_k\vec Q(x-z_k(T)) \in \Y^\xc_\perp(z(T)),
  \pr |p^+(z_k)v| \le C_X\de^2, \pq \|P^\xcu_\perp(z(T))v\|_\cH \le C_X\de^4, \pq |z_k(T)-z_k| \le C_X\de.}
\end{theorem}
In short, if we avoid the transition time $|I|\lec|\log\de|$, then either the solution settles down in a neighborhood of solitons (of a smaller number) or blows up. 

Therefore, if we start with sufficient distance of solitons $\eN(D_z)\ll 1$, then we may repeat applying Theorem \ref{thm:dyn near Ksol} and the above one, until we get $\La_+(T)\not=\empt$ (blow-up) or $\La_0(T)=\empt$ (decaying). 
Note that the solitons cannot get essentially closer, under the soliton repulsivity condition \eqref{sol repul}. 

Thus we conclude that under the soliton repulsivity condition, all solutions starting near $K$-solitons should be either decaying, blowing-up, or getting around some $J$-solitons with $J\le K$. 
Under the strong repulsivity condition, every $u$ in the last case is indeed an asymptotic $J$-soliton, by Lemma \ref{lem:conv Ksol}. 
The repulsivity conditions are valid if $K=2$ with $\s_1\s_2=-1$. 
In more general case, however, it is possible that solitons get closer, violating the condition \eqref{col away cond} after some time. 

Note also that we can not simply repeat the above theorem by itself to follow the full-time dynamics around $K$-solitons, because of the time limit $T<T_z^\star$.

\section{Difference estimate around collapsing solitons} \label{sect:diff col}
In this section, we consider the difference estimate for two solutions while some of the solitons is collapsing as above.
\subsection{Difference estimate during the collapse}
Under the assumption of Theorem \ref{thm:col Ksol} on the solution $u$, 
let $T\in(0,T_z^\star)\setminus I$ with $\La_+(T)=\empt$. 
Let $u^\ci$ be another solution starting close to $u$ at $t=0$, and define $u_k^\ci$ for $k=0\etc K$ in the same way as $u_k$, namely by \eqref{def uk}. Let $u'=u^\ci-u$ and $u_k':=u_k^\ci-u_k$. Then \eqref{uk match} applied to $u^\ci$ implies
\EQ{
 t\in[0,\min(T^{*\ci},T_k^{*\ci})),\ x\in\Om_k(t) \implies u_k^\ci(t,x)= u^\ci(t,x),}
where $T^{*\ci},T_k^{*\ci}$ denote the maximal existence times of $u^\ci$ and $u_k^\ci$, respectively. 
Hence we have the same identities for the differences $u_k'$, $u'$ as long as they exist. 

Assuming that $\|\vec u'(0)\|_\cH\le\de\ll\de_X^\star$, we may apply Theorem \ref{thm:col Ksol} also to $u^\ci$, which yields the transition time $I^\ci$ for $u^\ci$ with $|I^\ci|\lec|\log\de|$. 
For $T\in(0,T_z^\star)\setminus(I\cup I^\ci)$, we have the classification of the localized solutions $u_k$ by $\La_\dia(T)$ and of $u_k^\ci$ by $\La_\dia^\ci(T)$ for $\dia=0,\pm$. 
We want to estimate the difference $\vec u'$ under the assumption 
\EQ{
 \La_\dia(T)=\La_\dia^\ci(T)\ (\dia=0,\pm), \pq \La_+(T)=\La_+^\ci(T)=\empt,}
namely the same classification for $u_k$ and $u_k^\ci$ with no blow-up.

For each $k\in\La_-(T)=\La_-^\ci(T)$, $u_k$ and $u_k^\ci$ are uniformly bounded for $t\ge 0$ and small for $t\ge T$. 
On one hand, the standard argument for the Lipschitz dependence on the initial data yields on $t\le T$
\EQ{ \label{Gronwall}
 \|\vec u_k'(t)\|_\cH \lec e^{C_g t}\|\vec u_k'(0)\|_\cH,}
for some constant $C_g>0$ dependent only on $N,p$. 
On the other hand, since $\|\vec u_k\|_\cH\le\de^5$ is uniformly small for $t\ge T$, the equation for $u_k'$ becomes essentially free. More precisely, we have
\EQ{
 (\p_t^2-\De+1+2\al\p_t)u_k' = f(u_k^\ci)-f(u_k) = O(|u_k|^{p-1}|u_k'|+|u_k'|^p),}
and the same argument as for Lemma \ref{lem:linene damp} yields (actually this case is much easier)
\EQ{ \label{small decay}
 \sup_{t\ge T} e^{\mu(t-T)}\|\vec u_k'(t)\|_\cH \lec \|\vec u_k'(T)\|_\cH,}
for some constant $\mu>0$ depending only on $N,\al$. 
Then, combining the growth bound \eqref{Gronwall} and the decay \eqref{small decay}, we obtain
\EQ{
 \|\vec u_k'(t)\|_\cH \lec e^{\mu(T-t)+C_g T}\|\vec u_k'(0)\|_\cH\pq(t\ge T).}
Extending $I$ and $I^\ci$ after $T$ in length $O(|\log\de|)$, we may improve this bound to 
\EQ{ \label{col damp}
 t \ge T \implies \|\vec u_k'(t)\|_\cH \le \de^6\|\vec u_k'(0)\|_\cH.} 

For each $k\in\La_0(T)=\La_0^\ci(T)$, we may apply the difference estimate around $1$-soliton of Lemma \ref{lem:diff 1sol} to $u_k$ and $u_k^\ci$. Let $\z_k\in C^1([0,T];\R^N)$ be the soliton center for $u_k$ with the orthogonality 
$v_k:=\vec u_k-\s_k\vec Q(x-\z_k)\in\Y^\xc_\perp(\z_k)$. Let 
\EQ{ 
 \de_1:=\|\vec v_k\|_{L^\I_t(0,T;\cH)}+\de, \pq b^+_k:=p^+(\z_k)v_k, \pq b^{+'}_k:=p^+(\z_k)\vec u_k'.}
We have $\de_1\sim\de\ge\|\vec u'(0)\|_\cH$. Hence, if $\de \ll \de_m^\star$ and
\EQ{ \label{diff inst uk}
 \de_1\|\vec u_k'(0)\|_\cH \le \de_m^\star|b^{+'}_k(0)|,}
then Lemma \ref{lem:diff 1sol} implies that for $t\in[0,T]$, 
\EQ{ \label{ak diff grw}
 \pt |e^{-\nu^+t}b^{+'}_k-b^{+'}_k(0)| \le \tf14|b^{+'}_k(0)|, 
 \pq \de_1\|\vec u_k'\|_\cH \le C_m\de_m^\star|b^{+'}|. }
Note that the condition \eqref{a priori cond a'} in the lemma holds on $[0,T]$, since $k\in\La_0(T)=\La_0^\ci(T)$ implies $\|v_k\|_\cH+\|v_k^\ci\|_\cH\lec\de\ll\de^{1/3}$. 

We need to relate the decomposition in the different coordinates at $t=0$: 
\EQ{
 |b^{+'}_k(0)-p^+(z_k)\vec u'(0)|\pt= |(p^+(\z_k)-p^+(z_k))\vec u_k'(0)-p^+(z_k)\chi_k^C\vec u'(0)| 
  \pr\lec [|\z_k(0)-z_k|+\eN(\tf13D_z)]\|\vec u'(0)\|_\cH \lec \de_D^{1/3}\|\vec u'(0)\|_\cH,}
where \eqref{uk init center} is used. 
Hence the condition \eqref{diff inst uk} is satisfied (using $\de_D\ll\de^3$), if
\EQ{ \label{diff inst dom uk}
 C\de \|\vec u'(0)\|_\cH \le \de_m^\star|p^+(z_k)\vec u'(0)|,}
for some constant $C\in(1,\I)$.

Then we may continue the difference estimate for $t\ge T$ with
\EQ{ 
 \pt\|\vec u'(T)-\sum_{k\in\La_0(T)}\vec u_k'(T)\|_\cH 
 \pr\lec \sum_{l\in\La_-(T)}\|\vec u_l'(T)\|_{\cH(\Om_l'(T))}  + \sum_{k\in\La_0(T)} \|\vec u_k'(T)\|_{\cH(\R^N\setminus\Om_k'(T))}  + \|\vec u_0'(T)\|_{\cH}
  \pr\lec \sum_{0\le l\le K}\de^6\|\vec u_l'(0)\|_\cH \lec \de^6\|\vec u'(0)\|_\cH,}
where the first inequality uses \eqref{uk match} for $T<T_z^\star$, 
and the second inequality uses the damped decay estimate \eqref{col damp} for $l\in\La_-(T)$ as above, and similarly for the other terms by the argument of \eqref{est u0}--\eqref{ext dec uk} applied to the difference. 
Apply Lemma \ref{lem:center} to $\vec u(T)$ to get the centers $z(T)$ with the orthogonality. 
Let $v:=\vec u-\vec Q_\Si$, $a^+_k:=p^+(z_k)v$, $\be:=P^\xc_\perp(z)v$, and similarly for $u^\ci$ and $u'$. 
Then at $t=T$ and for $k\in\La_0(T)$, we have, using \eqref{uk last center} and that the radius of $\Om_k(T)>D_z/2$, 
\EQ{ \label{diff b' a'}
 |b^{+'}_k(T)-a^{+'}_k(T)| \pt= |(p^+(\z_k)-p^+(z_k))\vec u_k'-p^+(z_k)(\vec u'-\vec u_k')| 
  \pr\le |\z_k-z_k|\|\vec u_k'\|_\cH + \eN(\tf12D_z)[\|\vec u'\|_\cH+\|\vec u_k'\|_\cH].}
Combining the above estimates \eqref{diff inst dom uk}--\eqref{diff b' a'} with \eqref{ak diff grw}, we obtain  
\EQ{
 \de\|\vec u'(T)\|_\cH \lec \de_m^\star|a^{+'}(T)|,}
where $\vec u'(0)$ is absorbed by the last term, since we have \eqref{diff inst dom uk}.
Thus we have obtained 
\begin{lemma} \label{lem:diff col}
For any $N\in\N$, $\al\in(0,\I)$, $p\in(2,p^\star(N))$ and $K\in\N$, 
there exist $\de_Y^\star\in(0,\de_X^\star]$ and $C_Y\in[C_X,\I)$ such that the following hold. 
Let $u,u^\ci$ be two solutions of \eqref{NLKG} in $\cH$, both satisfying the assumption of Theorem \ref{thm:col Ksol} for the same $z\in(\R^N)^K$, $\s\in\{\pm 1\}^K$ and $\de\in(0,\de_Y^\star]$. 
Let $I,I^\ci\subset\R$ be the open sets, and $\La_\dia(T),\La^\ci_\dia(T)\subset\{1\etc K\}$ be the index sets, given by the theorem respectively for $u$ and $u^\ci$. 
Then there is an open set $\ti I\subset\R$ satisfying $I\cup I^\ci\subset \ti I$, $|\ti I|\le C_Y|\log\de|$ and the following properties. Let $T\in(0,T_z^\star)\setminus\ti I$ and suppose that 
$\La_\dia(T)=\La_\dia^\ci(T)$ for all $\dia=0,\pm$ with $\La_+(T)=\empt$. Let $u':=u^\ci-u$ and suppose that 
\EQ{
  a^{+'}_k(0) := p^+(z_k)\vec u'(0) \implies \de\|\vec u'(0)\|_\cH \le \de_Y^\star|a^{+'}_k(0)|}
for some $k\in\La_0(T)$. Then 
\EQ{
 |e^{-\nu^+T}a^{+'}_k(T)-a^{+'}_k(0)|\le \tf13|a^{+'}_k(0)|, \pq \de\|\vec u'(T)\|_\cH \le C_Y\de_Y^\star|a^{+'}_k(T)|,}
where $a^{+'}_k(T):=p^+(z_k(T))\vec u'(T)$ is defined for the decomposition in \eqref{decop after col}. 
\end{lemma}

In short, if two nearby solutions have the same behavior during the collapse, in terms of the localized solutions $u_k$ without the blow-up case, and if some of the unstable modes around the surviving solitons is initially dominant in the difference, then that component is exponentially growing and remains dominant after the collapse. 

\subsection{Difference estimate around 2-to-1 solitons} \label{sect:diff 2-to-1}
The dominant unstable modes grow exponentially in the above lemma in the same fashion as in Section \ref{sect:diff Ksol}. 
Hence we may combine those estimates to consider difference estimates around asymptotic $1$-solitons starting near repulsive $2$-solitons. 

Let $u,u^\ci$ be two global solutions of \eqref{NLKG} in $\cH$ starting near a $2$-soliton $Q_\Si:=\s_1 Q(x-z_1)+\s_2Q(x-z_2)$ with $K=2$ and $\s_1\s_2=-1$, 
\EQ{
 \sN_0(\vec u(0),z)+\sN_0(\vec u^\ci(0),z) \le |\de_0|^2 \ll 1, \pq \eN(D_z) \le \de_D\ll 1.}
Assuming $\de_0\ll\de_0^\star$, we may apply Theorem \ref{thm:dyn near Ksol} to both $u,u^\ci$ with some $\de_2\in(0,\de_2^\star]$ satisfying $\de_0\ll\de_2$. Let $z\in C^1([0,T_3);(\R^N)^2)$ be the unique centers for $u$ given by Lemma \ref{lem:center} for the orthogonality $v\in \Y^\xc_\perp(z)$, and $a_k^+:=p^+(z_k)v$ as before. 
We may assume without losing generality by symmetry that the collapsing time is earlier for $u$ than $u^\ci$ and for the $Q_2$ part than the $Q_1$ part, namely
\EQ{
 T_3\le T_3^\ci, \pq T_3<\I \implies |a^+_2(T_3)|\sim \de_2.}
Assuming $\de_D\le \de_2\ll\de_m^\star$, we may apply Lemma \ref{lem:diff 2sol} to the difference of $u':=u^\ci-u$ for any $\de_3\in(0,\de_m^\star]$ satisfying 
\EQ{ \label{init diff}
 \eN(D_{z(0)})+\|v\|_{L^\I_t(0,T_3;\cH)}+\|\vec u'(0)\|_\cH \le \de_3, \pq 
 \de_3\|\vec u'(0)\|_\cH  \le \de_m^\star|a^{+'}_1(0)|.}
We may choose $\de_3\sim\de_2$. Then $T_3\le T_3^\ci$ implies \eqref{a priori cond a'} for $t<T_3$, since otherwise $|a^{+\ci}|\ge|a^{+'}|-|a^{+}|\sim\de_3^{1/3}\gg\de_2$ and so $T_3^\ci<T_3$. Note that the orthogonality does not matter thanks to Theorem \ref{thm:dyn near Ksol}-\eqref{est T1}. 
Thus we obtain for $t<T_3$ from the lemma  
\EQ{
 \pt |e^{-m}a^{+'}_1-a^{+'}_1(0)| \le \tf14|a^{+'}_1(0)|, \pq |e^{-m}a^{+'}-a^{+'}(0)| \le \tf14|a^{+'}(0)|,  
 \pr \de_3\|\vec u'\|_\cH \le C_m\de_m^\star|a^{+'}_1|, 
 \pr \|P^\xu_\perp(z)\vec u'\|_\cH \le C_me^{C_m \de^{1/3} m}\|\vec u'(0)\|_\cH + C_m(\de_3+|a^{+'}|)|a^{+'}|.}
If $T_3=\I$ then $|a^{+'}_1|\sim |e^m a^{+'}_1(0)|\to\I$ as $t\to\I$ is contradicting that $\|v\|_\cH+\|v^\ci\|_\cH\ll 1$ for all $t<\I=T_3=T_3^\ci$. Hence $T_3<\I$. 

Starting from $t=T_3<\I$, we apply Lemma \ref{lem:diff col} to the difference $u'=u^\ci-u$ with some $\de_4\in(0,\de_Y^\star]$ satisfying 
\EQ{
 \pt \sN_0(\vec u(T_3),z(T_3))+\sN_0(\vec u^\ci(T_3),z(T_3))\le \de_4, \pq \eN(D_{z(T_3)}) \le \de_4^{C_X}, \pr \de_4\|\vec u'(T_3)\|_\cH \le \de_Y^\star|p^+(z_k)\vec u'(T_3)|.}
We may choose $\de_4\sim\de_2$ assuming
\EQ{
 \de_D \ll \de_2^{C_X}, \pq \de_2\ll\de_Y^\star.}
Note that $\eN(D_z)\lec\eN(D_{z(0)})$ for $t\le T_3$ by Theorem \ref{thm:dyn near Ksol} with \eqref{Dz 2sol}. 

For the open set $\ti I\subset\R$ given by the lemma, we have $|\ti I|\lec|\log\de_4|\sim|\log\de_2|$, 
and for each $T_4\in[T_3,T_4^*]\setminus \ti I$ with 
\EQ{
 T_4^*:=T_3+T_{z(T_3)}^\star,}
we have the decomposition $\{1,2\}=\La_0(T_4)\cup\La_-(T_4)=\La_0^\ci(T_4)\cup\La_-^\ci(T_4)$. $\La_+(T_4)=\La_+^\ci(T_4)=\empt$ since we assumed that both $u,u^\ci$ are global. 

Now suppose that both $u,u^\ci$ keep the $Q_1$ soliton, namely $1\in\La_0(T)\cap\La_0^\ci(T)$ for all $T\in[T_3,\I)\setminus\ti I$. Then $|a^+_1(T_3)|+|a^{+\ci}_1(T_3)|\lec\de_2^2$ and
\EQ{
 |a^{+'}_2(T_3)| \lec \de_3^{-1}|a^{+'}_1(T_3)| \lec \de_3^{-1}\de_2^2.}
Hence choosing $\de_3\ge C\de_2$ for a sufficiently large constant $C$, we may deduce that 
$|a_2^+(T_3)|\sim\de_2$ implies $|a_2^{+\ci}(T_3)|\sim\de_2$, 
so $2\in\La_-(T)\cap \La_-^\ci(T)$. 

Thus we obtain $\La_\dia(T_4)=\La_\dia^\ci(T_4)$ for all $\dia=0,\pm$, thereby the assumptions of Lemma \ref{lem:diff col} are fulfilled. Then the estimates of the lemma imply
\EQ{
 \pt a^{+'}_1(T_4) \sim e^{\nu^+(T_4-T_3)}a^{+'}_1(T_3), 
 \pq \de_4\|\vec u'(T_4)\|_\cH \lec |a^{+'}_1(T_4)| \lec \de_4^2.}
Moreover, since $\{1\}=\La_0(T)$ and $\{2\}=\La_-(T)$ for all $T\ge T_4$, we have a unique $z_1\in C^1([T_4,\I);\R^N)$ such that 
\EQ{
 \pt v:=\vec u-\s_1\vec Q(x-z_1) \in\Y^\xc_\perp(z_1),
 \pq |a^+(z_1)v| \lec \de_4^2, \pq \|P^\xcu_\perp(z_1)v\|_\cH\lec \de_4^4.}
Let $a_1^+:=p^+(z_1)v$, $a_1^{+\ci}:=p^+(z_1)v^\ci$ and $a_1^{+'}:=p^+(z_1)v'$ as before. 
Then we may apply Lemma \ref{lem:diff 1sol} to the difference $u'=u^\ci-u$ starting from $t=T_4$, with some $\de_5\in(0,\de_m^\star]$ such that
\EQ{
 \|v\|_{L^\I_t(T_4,\I;\cH)} \le \de_5, \pq \|\vec u'(T_4)\|_\cH \le \de_5,
 \pq \de_5\|\vec u'(T_4)\|_\cH \le \de_m^\star|a^{+'}_1(T_4)|.}
The above estimates allow us to choose $\de_5\sim\de_4\sim\de_2$. 
Then for $t\ge T_4$ and as long as $|a_1^{+'}|^3<\de_5$, we have
\EQ{
 \pt a^{+'}_1 \sim e^{\nu^+(t-T_4)}a^{+'}_1(T_4), \pq \de_5\|\vec u'\|_\cH \lec |a^{+'}_1|,
 \pr \|P^\xc_\perp(z)\vec u'\|_\cH \lec e^{C_m\de_5^{1/3}\nu^+(t-T_4)}\|\vec u'(T_4)\|_\cH + (\de_5+|a^{+'}_1|)|a^{+'}_1|.}
Then we have $|a^{+\ci}_1|\gg\|v^\ci\|_\cH^2$ after some time and before $|a_1^{+'}|$ reaches $\de_5^{1/3}$, so Theorem \ref{thm:dyn near Ksol}-\eqref{est T1} implies that $u^1$ exits the neighborhood of $Q_1$. 
Thus we conclude that it is impossible for both $u,u^\ci$ to keep the $Q_1$ part with the initial difference \eqref{init diff}, as well as all the size assumptions on the parameters. 

By contraposition, if both $u,u^\ci$ keep the $Q_1$ part, then under the closeness to the $2$-soliton and the distance condition, 
\EQ{
 \sN_0(\vec u(0),z)^{1/2}+\sN_0(\vec u^\ci(0),z)^{1/2} \le \de_2 \ll \min(\de_2^\star,\de_m^\star,\de_Y^\star), \pq \eN(D_z) \le \de_2^{C_X+1}}
we have, cf.~\eqref{init diff}, 
\EQ{
 |a^{+'}_1(0)| \lec \de_2\|\vec u'(0)\|_\cH.}
Thus we have obtained 
\begin{theorem} \label{thm:1sol Lip}
For any $N\in\N$, $\al\in(0,\I)$ and $p\in(2,p^\star(N))$, 
there exist $\de_L^\star\in(0,1)$ and $C_L\in(1,\I)$  such that the following hold. Let $u^0,u^1$ be two solutions of \eqref{NLKG} in $\cH$ satisfying initially for some $\s\in\{\pm 1\}^2$, $z\in(\R^N)^2$, and $\de\in(0,\de_L^\star]$, 
\EQ{
 \sN_0(\vec u^n(0),z) \le \de^2, \pq \eN(D_z) \le \de^{C_L}, \pq \s_1\s_2=-1,}
and asymptotic solitons with $Q_1$, namely 
\EQ{
 \lim_{t\to\I}\|\vec u^n(t)-\s_1\vec Q(x-z_1^n(t))-\te_n \s_2\vec Q(x-z_2^n(t))\|_\cH =0}
for some $z^n:[0,\I)\to(\R^N)^2$ and $\te_n\in\{0,1\}$ for each $n=0,1$. Then the initial difference $\vec u'(0):=\vec u^1(0)-\vec u^0(0)$ satisfies 
\EQ{ \label{Lip est}
 |p^+(z_1)\vec u'(0)| \le C_L\de \|\vec u'(0)\|_\cH.}
\end{theorem}
Note that the modulation of centers does not essentially affect the estimate \eqref{Lip est}, so we do not need to assume the orthogonality for either of $u^n$. 

When $C_L\de>0$ is small enough, the unstable component on the left side may be removed from the right side. In other words, 
\EQ{
 |p^+(z_1)\vec u'(0)| \lec C_L\de[|p^+(z_2)\vec u'(0)|+\|P^\xu_\perp(z)\vec u'(0)\|_\cH].}
This means that $p^+(z_1)\vec u(0)$ is uniquely determined by the other components: $p^+(z_2)\vec u(0)$ and $P^\xu_\perp(z)\vec u(0)$, among asymptotic solitons with $Q_1$ starting in a small neighborhood of $\vec Q_\Si$.  Moreover, the dependence is Lipschitz continuous with the smaller Lipschitz constant in the smaller neighborhood.

In the special case where both $u^n$ are asymptotic $2$-solitons, namely $\te_0=\te_1=1$, the above theorem applies as well after switching $Q_1$ and $Q_2$, thereby we obtain 
\EQ{
 |p^+(z)\vec u'(0)| \lec C_L\de\|P^\xu_\perp(z)\vec u'(0)\|_\cH.}
This is the uniqueness and Lipschitz dependence of the asymptotic $2$-solitons, which was proven in \cite{CMYZ}.

\section{Dynamical classification around $2$-solitons}
We are now ready to construct the manifolds of asymptotic $1$-solitons, near and approaching that of asymptotic $2$-solitons, and simultaneously to classify the full-time dynamics of all solutions in the small neighborhood of $2$-solitons. 

\subsection{Coordinate around $2$-solitons}
Let $K=2$, $z^*\in(\R^N)^2$, $\s\in\{\pm 1\}^2$ satisfy
\EQ{
 D_{z^*} \ge C_D, \pq \s_1\s_2=-1,}
for some constant $C_D\in(1,\I)$ to be chosen large. 
Let 
\EQ{
 \vec Q^*_k:=\s_k \vec Q(x-z^*_k), \pq \vec Q^*_\Si:=\sum_{k=1}^2 \vec Q^*_k,}
and consider the solution $u$ of \eqref{NLKG} with any initial data $\vec u(0)\in\cH$ satisfying 
\EQ{ \label{init u}
 \|\vec u(0)-\vec Q_\Si^*\|_\cH \ll 1.}
Decompose the initial data $\vec u(0)\in\cH$ as before
\EQ{
 \U v:=\vec u(0)-\vec Q_\Si^*, \pq \U a^+_k := p^+(z_k^*)\U v, \pq \U\be:= P^\xu_\perp(z^*)\U v,}
for $k=1\etc K$. Lemma \ref{lem:spec coord} allows us to use, choosing $C_D\ge D_\star$,
\EQ{
 \U v\in\cH \pt\leftrightarrow (\U a^+,\U\be) \in \R^K \oplus \Y^\xu_\perp(z^*),}
as coordinates for the initial data $\vec u(0)$. We consider initial data in the region 
\EQ{ \label{abe neigh}
 (\U a^+,\U\be)\in \R^2\oplus\Y^\xu_\perp(z), \pq |\U a^+|\le\de_\he, \pq \|\U\be\|_\cH\le\de_\he,}
for some constant $\de_\he\in(0,1)$ to be chosen small. 

We will impose several size conditions on $C_D$ and $\de_\he$. All of them are satisfied by choosing sufficiently small $\de_\he$, depending only on $N,\al,p$, and sufficiently large $C_D$, depending only on $N,\al,p,\de_\he$.

\subsection{Existence of $2$-solitons} \label{ss:exist 2sol}
First we recall the existence of $2$-solitons from \cite{CMYZ} in our notation. 
The proof is by a simple topological argument and had been worked out for other equations without damping as well. 

Fix any $z^*\in(\R^N)^2$ and $\U \be\in\Y^\xu_\perp(z^*)$ satisfying the above conditions. 
Imposing $\eN(C_D)^{1/2} \ll\de_\he \ll\de_2^\star$, we have 
\EQ{
 \U\sN_0(\vec u(0)) \le \sN_0(\vec u(0),z^*) \lec \de_\he^2 + \eN(C_D) \sim \de_\he^2 \ll |\de_2^\star|^2.}
Apply Theorem \ref{thm:dyn near Ksol} with some $\de_2>0$ satisfying $\de_\he\ll\de_2\le\de_2^\star$ to the solution $u$ with any initial data in the region \eqref{abe neigh}. 
Suppose for contradiction that the theorem yields $T_3<\I$ for every such $u$.

Let $z\in C^1([0,T_3];(\R^N)^2)$ be given by the theorem with the orthogonality $v=\vec u-\vec Q_\Si\in\Y^\xc_\perp(z)$. The theorem implies $\|v\|_\cH\lec\de_2$ and 
$\eN(D_z)\lec\de_\he^2$ for $t\le T_3$. 

Consider another decomposition $\U{v}(t)=\vec u(t)-\vec Q^*_\Si$ with $z^*(t)$ solving \eqref{modeq z} with the initial data $z^*(0)=z^*$. 
Then $p^j(z_k^*)\U{v}$ are conserved for all $j=1\etc N$ and $k=1\etc K$, and Lemma \ref{lem:center} implies, if $\de_\he^2 \ll \eN(1/\de_I)$ and $\de_2\ll\de_I$, then for $t\le T_3$
\EQ{
 \pt |z(t)-z^*(t)|\lec|\U{a}^\xc|\lec\de_\he, 
 \pr |\U a^+(t)-a^+(t)| \lec |z^*(t)-z(t)|[\|\U{v}(t)\|_\cH+\eN(D_z)]\ll\de_\he. }
Let $T\ge 0$ be the first time where $|\U a^+(T)|=\de_\he$. 
Theorem \ref{thm:dyn near Ksol} implies that $\de_\he\ll\|v(t)\|_\cH\ll\de_2$ at some $t\in[0,T_3)$, and also $\|v\|_\cH\sim|a^+|$ at any such $t$. 
Hence the above estimate implies that $T\in[0,T_3)$. 
Moreover, $\sN_0(\vec u(T),z(T))\lec\de_\he^2\ll\de_\he\sim|a^+(T)|$ implies $T\ge T_1$. 
Thus we obtain a mapping defined on the closed ball:
\EQ{
 \X: \U a^+ \in B(\de_\he):=\{a\in\R^2\mid |a|\le\de_\he\} \mapsto \U a^+(T) \in \p B(\de_\he).}
Note that $\U a^+$ also satisfies \eqref{eq a+ simp}, namely
\EQ{
 |(\p_t-\nu^+)\U a^+| \lec \eN(D_{z^*})+\|\U{v}\|_\cH^2,}
because $z^*$ solves the equation \eqref{modeq z}. Hence we have around $t=T$ 
\EQ{
 \p_t|\U a^+|^2 = 2\nu^+|\U a^+|^2 + O(\de_\he^3) \sim \de_\he^2>0.}
This uniform growth of $|\U a^+|$ around $t=T$, together with the local well-posednesss of the Cauchy problem, implies that the mapping $\X$ is continuous. 
Hence $\X$ is a retraction of $B(\de_\he)$ onto $\p B(\de_\he)$, but such a mapping does not exist, which is a contradiction.  

Hence there exists at least one $\U a^+\in\R^2$ such that Theorem \ref{thm:dyn near Ksol} yields $T_1=T_2=T_3=T^*=\I$, namely the solution $u$ staying forever near $Q_\Si$. 
Then the strong repulsivity \eqref{repul strong} for $K=2$ with $\s_1\s_2=-1$, or the estimate \eqref{dyn V} for $t<T_s=\I$, implies that $D_z\to\I$ as $t\to\I$. 
Then Lemma \ref{lem:conv Ksol} implies that $u$ is indeed an asymptotic $2$-soliton. 

The above argument may be applied to the more general $K$-soliton case under the soliton repulsivity condition \eqref{sol repul} to show existence of a solution $u$ staying forever near $Q_\Si$. If we have the strong repulsivity condition \eqref{repul strong}, then it is an asymptotic $K$-soliton by Lemma \ref{lem:conv Ksol}. 

\subsection{Dynamics around the $2$-soliton}
We have seen above that if $\de_\he>0$ is small enough and $C_D>1$ is large enough, then for any $z^*\in(\R^N)^2$ with $D_{z^*}\ge C_D$ and any $\U\be\in\Y^\xu_\perp(z^*)$ with $\|\U\be\|_\cH\le\de_\he$, there is at least one choice of $\U a^+\in \R^2$ with $|\U a^+|\le\de_\he$ such that the corresponding solution $u$ is an asymptotic $2$-soliton. 
Moreover, by the result of \cite{CMYZ} or Theorem \ref{thm:1sol Lip}, such $\U a^+$ is unique, provided that $\de_\he\ll\de_L^\star$ and $C_D\gg C_L|\log \de_L^\star|$. Moreover, it is given by some Lipschitz function $G$ in the form 
\EQ{ \label{graph 2sol}
 \U a^+ = G(\U \be).}
This has been established in \cite{CMYZ}, so the main task in this paper is to investigate the other solutions, where we need to distinguish the direction of instability $a^+$ in $\R^2$. 
Consider a region of initial data in the form
\EQ{ \label{rect region}
  (\U a^+,\U\be)\in \R^2\oplus\Y^\xu_\perp(z), \pq \max(|\U a^+_1|,|\U a^+_2|)\le\de_\he/\sqrt{2}, \pq \|\U\be\|_\cH\le\de_\he.}
Let $u^\ci$ be another solution with initial data in this region
satisfying for some $\ta\in\{\pm 1\}^2$
\EQ{
 \ta_k\s_k\U a^{+'}_k > \tf12 \|\U\be'\|_\cH \pq(k=1,2),}
where $u':=u^\ci-u$, $\U a^{+'}_k:=p^+(z^*)\vec u'(0)$ and $\U\be':=P^\xc_\perp(z^*)\vec u'(0)$. 
Denote the set of such initial data by $A_{\ta_1,\ta_2}$. 
We may apply Lemma \ref{lem:diff 2sol} with a small constant $\de\in(0,\de_m^\star]$ for both $k=1,2$, provided that $\de_\he+\eN(C_D)^{p_3}\ll\de_m^\star$. 
Since $u$ is an asymptotic 2-soliton, Theorem \ref{thm:dyn near Ksol} together with \eqref{Dz 2sol} implies
\EQ{
 |a^+| \lec \|v\|_\cH^2 \le \de_\he^2, \pq \eN(D_z)\lec\eN(C_D)}
for all $t\ge 0$, hence after sufficient time, the exponential growth given by the lemma implies 
\EQ{
 |e^{-m} a^{+\ci}_k-a^{+'}_k(0)| \le \tf14|a^{+'}_k(0)|, \pq \|P^\xu_\perp(z)\vec u^\ci\|_\cH \lec (\eN(C_D)^{p_3}+\de_\he)|a^{+'}|,}
for both $k=1,2$. So $a^{+\ci}_k$ has the same sign as $a^{+'}_k(0)$ with the dominant size. 
Then Theorem \ref{thm:col Ksol} implies 
\begin{enumerate}
\item If $+1\in\{\ta_1,\ta_2\}$ then $u^\ci$ blows up. 
\item If $\ta_1=\ta_2=-1$ then $u^\ci$ is global and decaying, 
\end{enumerate}
provided that $\de_\he^2\ll\de_X^\star$ and $C_D\gg C_X|\log\de_X^\star|$. 
If we fix $\U\be^\ci=\U\be$, then the above regions $A_{\ta_1,\ta_2}$ are reduced to four conical regions of $\U a^{+\ci}\in\R^2$ around the point of the asymptotic 2-soliton $\U a^+=G(\U\be)\in\R^2$, which 
are separated by the horizontal and vertical segments emanating from $G(\U\be)$. 

Take any horizontal or vertical segment away from $G(\U\be)$ that is connecting $A_{-,-}$ and either $A_{+,-}$ or $A_{-,+}$, and consider all the solutions $u^\ci$ starting from the initial data on the segment. 
By symmetry, it suffices to consider the case of $A_{-,-}\to A_{+,-}$. 

For all the initial data on the segment, Theorem \ref{thm:col Ksol} implies $2\in\La_-^\ci(T)$ for some $T>0$. If $1\in\La_+^\ci(T)$ then $u^\ci$ blows up, which is the case around the endpoint in $A_{+,-}$. If $1\in\La_-^\ci(T)$ then $u^\ci$ is global and decaying, which is the case around the endpoint in $A_{-,-}$. Since both cases are stable for initial perturbation by the theorem, there must be another case, that is when $u^\ci$ is an asymptotic $1$-soliton with the $Q_1$ side. 

If there are more than one point of the asymptotic soliton case on the segment, then their initial difference is dominated by the $\U a^+_1$ component, which is contradicting Theorem \ref{thm:1sol Lip}, provided that $\de_\he\ll\de_L^\star$ and $C_D\gg C_L|\log\de_L^\star|$. 
Hence there is a unique point on the segment corresponding to the $Q_1$ soliton, and the remainder is separated to the interval of blow-up and the other of decaying solutions. 
Thus we obtain a set of initial data for $u^\ci$ that are converging to the $Q_1$ solitons in the graph form of a Lipschitz function
\EQ{
 \U a^+_1 = H_1(\U a^+_2,\U\be), \pq \s_2(\U a^+_2 - G_2(\U\be))<0,}
within the small neighborhood \eqref{rect region}. 
By symmetry, we also obtain the graph for the $Q_2$ solitons in the form 
\EQ{
 \U a^+_2 = H_2(\U a^+_1,\U\be), \pq \s_1(\U a^+_1 - G_1(\U\be))<0.}
The uniqueness of $1$-solitons (including $2$-solitons) by the theorem implies
\EQ{
 H_1(G_2(\U\be),\U\be)=G_1(\U\be), \pq H_2(G_1(\U\be),\U\be)=G_2(\U\be),}
thereby connecting the graphs of asymptotic $Q_1$ solitons and those of $Q_2$ at the point $G(\be)\in\R^2$ of the asymptotic $2$-soliton. 

All the other solutions in the neighborhood are either blowing-up or decaying to $0$, and both are stable for initial perturbation, by Theorems \ref{thm:dyn near Ksol} and \ref{thm:col Ksol}. 
Since the square region $[-\de_\he,\de_\he]^2/\sqrt{2}$ of $a^{+}$ is separated into two connected (relatively) open sets by the graphs of $H_1$ and $H_2$ that are connected at the point of $G$, and each open set should be included entirely either in the blow-up case or in the decaying case. 
Since one of them contains $A_{+,+}$ and the other contains $A_{-,-}$, we may determine the cases. 
To write it explicitly, it is convenient to extend the graph to the other side as constants:
\EQ{
 \pt \s_2(\U a^+_2 - G_2(\U\be))\ge 0 \implies H_1(\U a^+_2,\U\be) := G_1(\U\be),
 \pr \s_1(\U a^+_1 - G_1(\U\be))\ge 0 \implies H_2(\U a^+_1,\U\be) := G_2(\U\be).}
Then the solution $u^\ci$ is decaying to $0$ in the region of initial data
\EQ{
 \s_1(\U a^+_1-H_1(\U a^+_2,\be)) < 0 \tand \s_2(\U a^+_2-H_2(\U a^+_1,\be)) < 0,}
while $u^\ci$ is blowing up in the region of initial data
\EQ{
 \s_1(\U a^+_1-H_1(\U a^+_2,\be)) > 0 \tor \s_2(\U a^+_2-H_2(\U a^+_1,\be)) > 0,}
both within the small neighborhood \eqref{rect region}. 
This concludes the proof of Theorem \ref{thm:main}.  

\section{Multi-solitons with symmetry} \label{sect:sym Ksol} 
In this section, we consider $K$-solitons of $K\ge 3$ under some symmetry restrictions which ensure the strong repulsivity condition \eqref{repul strong}. 
We may consider in general symmetries with respect to any subgroup of the orthogonal group $O(N)$ on $\R^N$, but for simplicity we consider just two examples generated by reflections and permutations.

Let $G_R\subset O(N)$ be the subgroup generated by the reflections $R_j:x_j\mapsto-x_j$ for $j=1\etc N$. For any non-empty $I\subset\{1\etc N\}$, 
let $G_I\subset O(N)$ be the subgroup generated by the exchanges $X_{j,k}:x_j\mapsto x_k$ for $j,k\in I$. We consider the symmetry generated by them, namely the subgroup
\EQ{
 G:=G_R\rtimes G_I \subset O(N).}

\subsection{Odd symmtry}
Let $I\subset \{1\etc N\}$, $\ta:I\to\{\pm 1\}$ and $I_\pm:=\ta^{-1}(\{\pm 1\})$. 
Assume that $|I_+|=|I_-|\ge 1$. 
Let $G=G_R\rtimes G_I$ be as above, and let $\ta_*:G\to\{\pm 1\}$ be the unique group homomorphism satisfying 
\EQ{
 \ta_*(R_j)=1, \pq \ta_*(X_{j,k})=\ta(j)\ta(k).}

We consider the equation \eqref{NLKG} under the symmetry restriction
\EQ{
 \cH_G^\ta := \{\fy\in\cH \mid \forall g\in G,\ \fy(gx)=\ta_*(g)\fy(x)\}.}
We have a one-parameter family of $2|I|$-solitons in this subspace
\EQ{
 \cQ_\ta(r) := \sum_{i\in I} \ta(i)[Q(x-re_i)+Q(x+re_i)]}
for $r\in\R$. The number of solitons $2|I|=4|I_+|$ is at least $4$. 
The case of $2|I|=4$ is included in the asymptotic solitons constructed in \cite{F}, which are in the shape of regular polygons on a plane. 
$\cQ_\ta(r)$ satisfies the strong repulsivity condition \eqref{repul strong} for $r\gg 1$. Indeed, 
\EQ{
 V(\cQ_\ta(r)) \pt= -\tf12\sum_{i,j\in I}^{i\not=j}\sum_{\ta_1,\ta_2\in\{\pm 1\}} \ta(i)\ta(j)\eN_0(r|\ta_1e_i-\ta_2e_j|)
 \pn- \sum_{i\in I} \eN_0(2r)
 \pr=|I|(\tf12\eN_0(\sqrt{2}r)-\eN_0(2r)) \sim  |I|\eN(D_z),}
because of the exponential decay of $\eN_0$, and similarly for each $i\in I$, 
\EQ{
 \na_{z_i^\pm} V(\cQ_\ta(r)) \pt= -\sum_{j\in I\setminus\{i\}} \sum_{\ta_1=\pm 1}\ta(i)\ta(j)\tf{\ta_1 e_i}{\sqrt{2}}\eN_0'(\sqrt{2}r) \pm e_i \eN_0'(2r)
 \pr= \pm e_i[\tf{-1}{\sqrt{2}}\eN_0'(\sqrt{2}r)+\eN_0'(2r)] \sim \pm e_i \eN(D_z),}
where $z_i^\pm$ is the coordinate of the soliton at $x=\pm re_i$. 

Moreover, the reflection symmetry implies that $\dot z_i^\pm$ is in the $\pm e_i$ direction even in the full equation depending on the solution $u$ in $\cH_G^\ta$, and the permutation symmetry implies that the force is the same at all the soliton centers. In other words, the equation of centers \eqref{eq z asy} is reduced to the form
\EQ{
 \dot z_i^\pm = \pm e_iF(r,v), \pq F(r,v) \sim \eN(D_z)+O(\eN(D_z)^{p_1}+\|v\|_\cH^2).}
The expansion around the soliton is essentially the same at every point, upto a change of the sign. In particular, we have the unstable modes
\EQ{
 a^+_{i,+} = a^+_{i,-} = \ta(i)\ta(j)a^+_{j,\pm}.}

Therefore, the dynamical analysis around $\cQ_\ta(r)$ of $r\gg 1$ in the symmetry $\cH_G^\ta$ is even easier than the general $2$-soliton case. 
In particular, we can prove, by the same argument as above, existence of asymptotic $2|I|$-solitons, which makes a codimension-$1$ Lipschitz manifold, separating the rest of small neighborhood into decaying and blow-up solutions. 
To state the precise result, let
\EQ{  \label{def YG}
 \pt \Y_\perp^\xu(r;\de)_G^\ta := \{\fy\in H_G^\ta \mid \forall i\in I,\ \om(\fy,Y^+(x\mp re_i))=0,\ \|\fy\|_\cH<\de\}.}
\begin{prop}
There exists a small constant $\de=\de(N,\al,p)>0$ such that for any $r>1/\de$, non-empty $I\subset \{1\etc N\}$ and $\ta:I\to\{\pm 1\}$ satisfying $|\ta^{-1}(\{+1\})|=|\ta^{-1}(\{-1\})|$,  
there is a Lipschitz function $H:\Y_\perp^\xu(r;\de)_G^\ta\to(-\de,\de)$ with the following properties. 
For any $h\in(-\de,\de)$ and $\fy\in\Y_\perp^\xu(r;\de)_G^\ta$, 
let $u$ be the solution of \eqref{NLKG} with the initial data
\EQ{ 
 \vec u(0) = \sum_{i\in I}\sum_{\ta=\pm 1} \ta(i)[\vec Q+h Y^+](x-\ta re_i)  + \fy \in \cH_G^\ta.}
Then its global behavior is classified by the initial data as follows. 
\begin{enumerate}
\item If $h<H(\fy)$ then $u$ is global and decaying.
\item If $h=H(\fy)$ then $u$ is global and 
\EQ{
 \vec u(t) = \sum_{i\in I}\sum_{\ta_1=\pm 1} \ta(i)\vec Q(x-\ta_1 r(t)e_i) +o(1) \IN{\cH}} 
as $t\to\I$ for some $r:[0,\I)\to(0,\I)$ satisfying $r(t)\to\I$, namely an asymptotic $2|I|$-soliton.
\item If $h>H(\fy)$ then $u$ blows up in finite time. 
\end{enumerate}
\end{prop}
\begin{proof}
Since we have the strong repulsivity condition, the existence of the asymptotic $2|I|$-soliton follows in the same way as in Section \ref{ss:exist 2sol}. 
The difference estimates in Section \ref{sect:diff Ksol} is trivialized as in the $1$-soliton case, because all the unstable modes are the same (up to the sign) by the symmetry, which removes all the issues from the interaction among the unstable modes. 
Hence we have a difference estimate similar to Lemma \ref{lem:diff 1sol}, although the growth rate is affected by the soliton interactions as $e^{m(t)}$ in place of $e^{\nu^+t}$. 
We do not need the difference estimate around collapsing solitons in Section \ref{sect:diff col}. 
In other words, we obtain directly from Section \ref{sect:diff Ksol} a difference estimate similar to Theorem \ref{thm:1sol Lip} in the case where both $u^0,u^1$ are asymptotic $2|I|$-solitons. 
Thus we obtain the uniqueness and Lipschitz continuity of $H$. 
For the other solutions, it suffices to note that once the solution exits from the soliton neighborhood into the decaying region $\La_-(T)$, it happens to all the components by symmetry, so the solution is decaying. Thus we obtain the dichotomy into decaying and blow-up in the complement region. 
\end{proof}

Note that the permutation symmetry of the solitons $\cQ_\ta$ by $G_I$ is unstable in the sense that if the distance $2r$ between the pair of solitons $Q(x-re_i)+Q(x+re_i)$ on the $x_i$ axis depends on $i$, then those at the longer distance move the faster by the gradient flow $\dot z=-\na_z V(z)$. 
It is easy to see that the solutions of the gradient flow lead to soliton collisions in finite time unless the distance $r$ is independent of $i$. 
This is another instability of $\cQ_\ta$ (for less symmetric perturbations) besides the unstable modes of the individual solitons.

\subsection{Even symmetry}
Let $I\subset\{1\etc N\}$ be non-empty, and $G=G_R\rtimes G_I$ as before. We consider the equation \eqref{NLKG} under the symmetry restriction 
\EQ{
 \cH_G := \{\fy\in\cH \mid \forall g\in G,\ \fy(gx)=\fy(x)\}. }
We have a one-parameter family of solitons in this subspace 
\EQ{
 \cQ_0^I(r):= \sum_{i\in I}[Q(x-re_i)+Q(x+re_i)]-Q(x)}
for $r\in\R$. The number of solitons $2|I|+1$ is at least $3$. 
$\cQ_0^I$ satisfy the strong repulsivity condition \eqref{repul strong} for $r\gg 1$. Indeed, 
\EQ{
 V(\cQ_0^I(r)) \pt= -\tf12\sum_{i,j\in I}^{i\not=j}\sum_{\ta_1,\ta_2\in\{\pm 1\}} \eN_0(\sqrt{2}r)-\sum_{i\in I}\eN_0(2r) + \sum_{i\in I}\sum_{\ta=\pm 1}\eN_0(r)
 \pr= |I|\BR{(1-|I|)\eN_0(\sqrt{2}r) - \eN_0(2r) + 2\eN_0(r)} \sim |I|\eN(D_z),}
and similarly 
\EQ{
 \na_{z_i^\pm}V(\cQ_0^I(r)) \pt= 2(|I|-1)\tf{\pm e_i}{\sqrt{2}}\eN_0'(\sqrt{2}r)+\pm e_i(\eN_0'(2r)-\eN_0'(r))
 \pn\sim \pm e_i \eN(D_z),}
while the force at $x=0$ is obviously $0$. 

In the same way as in the previous section, under the symmetry $\cH_G$, the equation is reduced essentially to a $2$-soliton case:
\EQ{
 \dot z_i^\pm = \pm e_i F(r,v), \pq a^+_{i,+}=a^+_{j,\pm}, \pq \ga=g^*\ga,}
coupled with the soliton at $x=0$, which is not moving. 

Hence the construction of asymptotic $(2|I|+1)$-solitons works in the same way, but the dynamical analysis around them is far more complicated: 
If the soliton at the origin collapses earlier, then the repulsivity condition is lost for the remaining solitons, and they start attracting each other. 

To state the precise result on the asymptotic solitons, let
\EQ{ 
 \Y_\perp^\xu(r;\de)_G := \{\fy\in \Y_\perp^\xu(r;\de)_G^1 \mid \om(\fy,Y^+)=0\},}
which has one more orthogonality than \eqref{def YG} with the constant map $\ta=1:I\to\{1\}$. 
\begin{prop}
There exists a small constant $\de=\de(N,\al,p)>0$ such that for any $r>1/\de$ and $I\subset \{1\etc N\}$, there is a Lipschitz function $H:\Y_\perp^\xu(r;\de)_G\to(-\de,\de)^2$ with the following properties. 
For any $h\in(-\de,\de)^2$ and $\fy\in\Y_\perp^\xu(r;\de)_G$, 
let $u$ be the solution of \eqref{NLKG} with the initial data
\EQ{
 \vec u(0) = \sum_{i\in I}\sum_{\ta=\pm 1} [\vec Q+h_1 Y^+](x-\ta re_i)  - [\vec Q+h_2Y^+](x) + \fy \in \cH_G.}
If $h=H(\fy)$, $u$ is global and for some $r:[0,\I)\to(0,\I)$ satisfying $r(t)\to\I$,
\EQ{
 \vec u(t) = \sum_{i\in I}\sum_{\ta=\pm 1} \vec Q(x-\ta r(t)e_i) - \vec Q(x) + o(1) \IN{\cH}} 
as $t\to\I$, namely an asymptotic $(2|I|+1)$-soliton. 
If $h\not=H(\fy)$, then $u$ is either blowing up or $E(\vec u)<(2|I|+1)E(\vec Q)$ at some $t>0$, so it can not be in the above form. 
\end{prop}
\begin{proof}
The existence is the same as in the previous case. 
The difference estimate is more involved, as we have essentially two unstable components, which do not have any apparent symmetry or cancellation as in the case of $2$-solitons in Section \ref{sect:even decay}. 
However, if we do not care about the direction of the unstable modes, the difference estimates may be simpler than Section \ref{sect:diff Ksol}, and as in \cite{CMYZ}. Indeed, we have
\EQ{
 \p_t|a^{+'}|^2 = 2a^{+'}\cdot \dot a^{+'} = 2\nu^+|a^{+'}|^2 + O(|a^{+'}|(\|\M\|_{\op}\|v'\|_\cH+\|v'\|_\cH^2)),}
which may be rewritten under the assumption $\|\M\|_{\op}+\|v'\|_\cH\le \de\ll 1$ as
\EQ{ \label{est a+'}
 |\p_t |a^{+'}| - \nu^+|a^{+'}|| \le C_1\de_2\|v'\|_\cH. }
Henceforth $C_j$ ($j=1,2,3,4$) are some constants depending only on $N,\al,p$. 
For the other components, we have as before
\EQ{
 B \le C_2\|v'(0)\|_\cH + C_2\cI, \pq \cI:=\de\int_0^t\|v'(s)\|_\cH ds,}
where $B$ is defined in \eqref{def B}, satisfying $\|v'\|_\cH\le C_3(|a^{+'}|+B)$. 
Suppose that for $0\le t\le T$, it holds 
\EQ{ \label{Lip C4}
 C_4 \de\|v'\|_\cH \le |a^{+'}|.}
Then combining it with the above estimates, we have for $0\le t\le T$ 
\EQ{
 \pt C_4\de\|v'\|_\cH-|a^{+'}|  \le C_4C_3\de B+(C_4C_3\de-1)|a^{+'}|
 \pr\le C_4C_3C_2\de\|v'(0)\|_\cH + C_4C_3C_2\de\cI 
 \pn - |a^{+'}(0)| - (\nu^+C_4-C_1)\cI
 \pr\le C_4C_3C_2\de\|v'(0)\|_\cH - |a^{+'}(0)| < 0,}
provided that $\de C_3C_4<1$, $(\nu^+-\de C_2C_3)C_4 \ge C_1$ and
\EQ{ \label{init Lip}
 |a^{+'}(0)| > C_2C_3C_4\de\|v'(0)\|_\cH.}
The first two conditions are easily fulfilled by taking $\de>0$ small enough and then choosing $C_4$ big enough. Then by continuity in time, the estimate \eqref{Lip C4} is valid as long as $\|\M\|_{\op}+\|v'\|_\cH\le\de$, if the initial difference satisfies \eqref{init Lip}. 

This is enough for the uniqueness and Lipschitz continuity of $H$, and all the other solutions have to exit the soliton neighborhood in finite time, losing some amount of energy. 
\end{proof}
In contrast to the previous example of $\cQ_\ta$, the permutation symmetry of $\cQ_0^I$ by $G_I$ is stable in the sense that if the distance of the pairs on the $x_i$ axis depends on $i$, then those at the longer distance move the slower by the gradient flow $\dot z=-\na_z V(z)$.

\subsection{Dynamics around $3$-solitons and soliton merger}
For further consideration, let us restrict the above problem to the simplest case $N=1$. 
Then $\cH_G$ is just the energy space with the even symmetry on $\R$, and 
\EQ{
 \cQ_0^I(r) = Q(x-r) - Q(x) + Q(x+r).}
Suppose that $u$ starts from a small neighborhood of $\cQ_0^I(r)$ with $r\gg 1$ in $\cH_G$, and that it is not an asymptotic $3$-soliton. 
Then the soliton resolution of \cite{CMY} implies that $u$ is either decaying, blowing up, an asymptotic $1$-soliton or an asymptotic $2$-soliton.
The last option is precluded by the symmetry, since asymptotic $2$-solitons cannot be even due to $\s_1\s_2=-1$ as proven in \cite{CMY}. 

Let $a^+_0$ and $a^+_1$ be the unstable modes at $x=0$ and at $x=r$ respectively. If we choose initial data satisfying 
\EQ{ \label{Q0 col}
 1 \gg a^+_0(0) \gg |a^+_1(0)|+\|P^\xu_\perp(z)v(0)\|_\cH, \pq r(0)\gg|\log |a^+_0(0)||,}
Then by Theorem \ref{thm:col Ksol}, the soliton at $x=0$ collapses earlier, and $u$ becomes close to the $2$-soliton $Q(x-r)+Q(x+r)$ after some time. If we further impose the condition 
\EQ{
 1 \gg a^+_0(0) \gg |a^+_1(0)| \gg \|P^\xu_\perp(z)v(0)\|_\cH, \pq r(0)\gg|\log |a^+_1(0)||, }
then the solitons at $x=\pm r$ also collapse at later time and before the soliton starts attracting each other. So $u$ is decaying if $a^+_1(0)<0$ and blowing up if $a^+_1(0)>0$.  
Since the set of initial data of decaying and that of blow-up are both open in $N=1$, 
there must be other solutions within the connected set of initial data \eqref{Q0 col}, for which the only possibility is convergence to $\pm Q$. 

In short, for any $\de\in(0,1)$, there is $R\in(1,\I)$ such that for any $r>R$ there is a global even solution $u$ on $\R$ satisfying 
\EQ{
 \|\vec u(0)-Q(x-r)-Q(x+r)\|_\cH<\de, \pq \lim_{t\to\I}\|\vec u(t)- \s_*\vec Q\|_\cH=0}
with either $\s_*=1$ or $\s_*=-1$. Because of the even symmetry, both the solitons at $x=\pm r$ have to collapse, but then another soliton at $x=0$ emerges, taking energy from the two collapsing solitons or the aftermath.  
This behavior is totally different from the asymptotic $1$-solitons in Theorem \ref{thm:main}, in which one of the two solitons collapses and disappears, while the other remains for all time (cf.~Theorem \ref{thm:col Ksol}).

\end{document}